\documentclass{article} 
\usepackage{amsmath,amsthm}
\usepackage{amssymb}
\usepackage{epsfig,tikz}
\usepackage{graphicx}
\usepackage[margin=0.8in]{geometry}
\usepackage{setspace}
\usepackage{enumerate}
\usepackage{tikz}
\usepackage{tikz-cd}
\usepackage{xcolor}
\usepackage{mathtools}
\usepackage{mathabx} 
\usepackage{stmaryrd}
\usepackage{thmtools}
\usepackage{etoolbox} 

\newbool{defendused}

\usepackage[backref=page]{hyperref}
\hypersetup{
  colorlinks   = true,          
  urlcolor     = blue,          
  linkcolor    = purple,          
  citecolor   = violet             
}

\newcommand{\defend}{\hfill\mbox{$\lozenge$}}

\newcommand{\defendhere}{%
  \nobreak\hfill\nobreak\mbox{$\lozenge$}%
  \gdef\defend{}
  \AtEndEnvironment{definition}{\gdef\defend{\hfill\mbox{$\lozenge$}}}
}

\theoremstyle{definition}
\newtheorem{definition}{Definition}
\newtheorem{lemma}[definition]{Lemma}
\newtheorem{proposition}[definition]{Proposition}
\newtheorem{theorem}[definition]{Theorem}
\newtheorem{corollary}[definition]{Corollary}
\newtheorem{example}[definition]{Example}
\newtheorem{construction}[definition]{Construction}
\newtheorem{notation}[definition]{Notation}

\newtheorem{remark}[definition]{Remark}

 \AtEndEnvironment{definition}{\defend}

\renewcommand{\thmcontinues}[1]{continued}

\let\tilde\widetilde

\newcommand{\A}{{\mathbb{A}}}

\newcommand{\PP}{\mathbb{P}}         
\newcommand{\QQ} {{\mathbb Q}}		
\newcommand{\RR} {{\mathbb R}}

\def\E{\mathrm{E}}
\def\n{\mathrm{n}}
\def\L{\mathrm{L}}
\def\V{\mathrm{V}}
\def\H{\mathrm{H}}
\def\g{\mathrm{g}}
\newcommand{\cS}{{\mathsf{S}}}

\def\cM{\mathcal{M}}

\newcommand{\Mbar}{\overline{\cM}\vphantom{\cM}}
\newcommand{\Mtilde}{\widetilde{\cM}\vphantom{\cM}}

\def\trop{\mathsf{trop}}

\def\log{\mathsf{log}}

\newcommand{\Spec}{\operatorname{Spec}}

\newcommand{\R}{\mathsf{R}}
\newcommand{\afanmap}{\mathfrak{t}}
\newcommand{\deco}{\gamma}

\usepackage{subcaption}

\newcommand{\isom}{\stackrel{\sim}{\longrightarrow}}

\renewcommand{\log}{{\mathsf {log}}}
\newcommand{\op}{{\mathsf {op}}}

\newcommand{\gp}{{\mathsf {gp}}}

\newcommand{\gl}{{\mathsf{gl}}}

\newcommand{\sPP}{{\mathsf{sPP}}}
\newcommand{\ghost}{\overline{{\mathsf {M}}}}

\newcommand{\pPP}{\mathsf{PP}}

\newcommand{\B}{\mathbb{B}}

\newcommand{\G}{\mathbb{G}}

\newcommand{\N}{\mathbb{N}}
\renewcommand{\P}{\mathbb{P}}
\newcommand{\Q}{\mathbb{Q}}

\newcommand{\Z}{\mathbb{Z}}
\newcommand{\Acal}{\mathcal{A}}
\newcommand{\Bcal}{\mathcal{B}}

\newcommand{\Mcal}{\mathcal{M}}

\newcommand{\Ocal}{\mathcal{O}}

\newcommand{\str}{\text{str}}

\newcommand{\LogSch}{\operatorname{LogSch}}
\newcommand{\colim}{\operatornamewithlimits{colim}}

\newcommand{\ul}[1]{{\underline{#1}}}

\DeclareMathOperator{\CH}{\mathsf{CH}}
\DeclareMathOperator{\LogCH}{\mathsf{logCH}}
\DeclareMathOperator{\LogS}{\mathsf{logS}}
\DeclareMathOperator{\Hom}{\mathsf{Hom}}
\newcommand{\logS}{\LogS}
\newcommand{\logCH}{\LogCH}

\newcommand{\Aut}{\mathsf{Aut}}

\makeatletter
\newcommand*{\doublerightarrow}[2]{\mathrel{
  \settowidth{\@tempdima}{$\scriptstyle#1$}
  \settowidth{\@tempdimb}{$\scriptstyle#2$}
  \ifdim\@tempdimb>\@tempdima \@tempdima=\@tempdimb\fi
  \mathop{\vcenter{
    \offinterlineskip\ialign{\hbox to\dimexpr\@tempdima+1em{##}\cr
    \rightarrowfill\cr\noalign{\kern.5ex}
    \rightarrowfill\cr}}}\limits^{\!#1}_{\!#2}}}

\newcommand{\ol}[1]{\overline{#1}}
\newcommand{\id}{\mathsf{id}}
\newcommand{\Star}{\mathsf{Star}}

\title{Logarithmic tautological rings of the moduli spaces of curves}
\author{Rahul Pandharipande, Dhruv Ranganathan, Johannes Schmitt and Pim Spelier}
\date{April 2025}

\begin{document}
\maketitle
\begin{abstract}
We define the logarithmic tautological rings of the moduli
spaces of Deligne--Mumford stable curves (together with a set
of additive generators lifting the decorated strata classes of the
standard tautological rings). While these algebras are infinite dimensional, a connection to polyhedral combinatorics 
via a new theory of homological piecewise polynomials
allows
an effective study.
A complete calculation is given in
genus 0 via the algebra
of piecewise polynomials on the cone stack of the associated Artin fan (lifting Keel's presentation of the Chow ring of 
$\overline{\mathcal{M}}_{0,n}$). 
Counterexamples to the simplest
generalizations in genus 1 are presented.
We show, however, that the
structure of the log tautological
rings is determined
by the complete knowledge of all
relations in the standard tautological
rings of the moduli spaces of curves. In particular, Pixton's conjecture
concerning relations in the standard tautological rings lifts
to a complete conjecture for relations in the log tautological rings of
the moduli spaces of curves.
Several open questions are discussed. 

We develop the entire theory of logarithmic tautological classes in the context of arbitrary smooth normal crossings pairs $(X,D)$
with explicit formulas for intersection products. As a special case, we give an explicit set of
additive generators of the full logarithmic Chow ring of $(X,D)$ in terms of Chow classes on the strata of $X$ and piecewise polynomials on the cone stack.
\end{abstract}

\tableofcontents


\section{Introduction}

\subsection{Overview}
The Chow ring of the moduli space $\Mbar_{g,n}$ of Deligne--Mumford stable curves{\footnote{We 
view $\Mbar_{g,n}$ as a nonsingular Deligne--Mumford stack and always
consider the Chow theory with ${\mathbb{Q}}$-coefficients. All of our Chow and tautological rings are algebras
over $\mathbb{Q}$.}} contains a distinguished subring of {\em tautological classes} 
\[\mathsf{R}^\star(\Mbar_{g,n})\subset \mathsf{CH}^\star(\Mbar_{g,n})\, .\] 
The structure of $\mathsf{R}^\star(\Mbar_{g,n})$ is related to  
the geometry of stable maps, Abel-Jacobi theory, the classification of CohFTs, and many
other directions, see \cite{FP-TNT,Pan15} for a survey.

Since the boundary $\Delta \subset \Mbar_{g,n}$, defined by the locus of nodal curves, 
is a divisor with normal crossings, we may view the pair $(\Mbar_{g,n}, \Delta)$
as a log scheme. The {\em logarithmic Chow ring} of  $(\Mbar_{g,n}, \Delta)$ is
defined{\footnote{Formally, we should write 
$\mathsf{logCH}^\star(\Mbar_{g,n},\Delta)$ to specify the log
structure. Since we will not consider any other log structure on
$\Mbar_{g,n}$, the boundary $\Delta$ will be omitted from the
notation.}} by
\begin{equation} \label{deflogch}
\mathsf{logCH}^\star(\Mbar_{g,n}) = \varinjlim _{\widetilde{\mathcal M}_{g,n}\to \Mbar_{g,n}} \mathsf{CH}^\star(\widetilde{\mathcal M}_{g,n})\, ,
\end{equation}
where the direct limit is taken over all logarithmic modifications{\footnote{For singular log modifications $\widetilde{\mathcal{M}}_{g,n}$, the notation $\mathsf{CH}^\star$ denotes operational Chow.
In some contexts, the flexibility
of considering singular log modifications can be useful.}}
with respect to the log structure
$(\Mbar_{g,n}, \Delta)$, and the transition maps are given by  pullback.
The simplest log modification is a blowup along a nonsingular
stratum of the log structure. Since the compositions of such 
log modifications are cofinal in the system of all log modifications, we can 
restrict to compositions of blowups of nonsingular strata in the direct 
limit in definition \eqref{deflogch}. 

There is a canonical injection via
pullback, 
$$\alpha:\mathsf{CH}^\star(\Mbar_{g,n}) \hookrightarrow \mathsf{logCH}^\star(\Mbar_{g,n})\, ,$$
so every Chow class is also a log Chow class.
There is also a canonical surjection via pushforward
$$\beta: \mathsf{logCH}^\star(\Mbar_{g,n}) \twoheadrightarrow \mathsf{CH}^\star(\Mbar_{g,n}) \, $$
satisfying $\beta \circ \alpha = \text{Id}$.

Our goal here is to define and to begin the study of the 
tautological subring of the logarithmic Chow ring of the moduli space of curves,
$$\mathsf{logR}^\star(\Mbar_{g,n}) \subset
\mathsf{logCH}^\star(\Mbar_{g,n}) \, .$$
The first motivation for the study of $\mathsf{logCH}^\star(\Mbar_{g,n})$
comes from logarithmic Gromov-Witten theory: log Chow groups are essential for
the log product \cite{Herr, R19b} and degeneration formulas \cite{ACGS15,R19} and the logarithmic double ramification
cycle \cite{HMPPS,HS22,MR21}. At a more fundamental level, the motivation is that 
the log  structure on $\Mbar_{g,n}$
is an intrinsic aspect
of the geometry of stable curves, and the corresponding log Chow theory
can not be avoided. The tautological subring $\mathsf{logR}^\star(\Mbar_{g,n})$
represents the most tractable log Chow classes.

\subsection{The strata algebra of \texorpdfstring{$\Mbar_{g,n}$}{Mbargn}}
We review here the construction of the strata algebra
$\cS^\star(\Mbar_{g,n})$
following \cite{GP03}.
The strata algebra provides a basic framework for the
study of tautological classes on the moduli spaces of curves.
\label{straaaa}

\subsubsection{Stable graphs}
The strata of
the logarithmic boundary of the moduli space $\Mbar_{g,n}$   correspond
to {\em stable graphs}. A stable graph
$\Gamma$ 
consists of the data
$$\Gamma=(\V, \H,\L, \ \mathrm{g}:\V \rightarrow \Z_{\geq 0},
\ v:\H\rightarrow \V, 
\ \iota : \H\rightarrow \H)$$
satisfying the following properties:
\begin{enumerate}
\item[(i)] $\V$ is a vertex set with a genus function 
$\g:\V\to \Z_{\geq 0}$,
\item[(ii)] $\H$ is a half-edge set equipped with a 
vertex assignment $$v:\H \to \V$$ and an 
involution $\iota:\H\to \H$,
\item[(iii)] $\E$, the edge set, is defined by the
2-cycles of $\iota$ in $\H$ (self-edges at vertices
are permitted),
\item[(iv)] $\L$, the set of legs, is defined by the fixed points of $\iota$ and is endowed with a bijective correspondence with the set of markings $$\L \leftrightarrow \{1,\ldots, n\}\, ,$$
\item[(v)] the pair $(\V,\E)$ defines a {\em connected} graph,
\item[(vi)] for each vertex $v$, the stability condition holds:
$$2\g(v)-2+ \n(v) >0,$$
where $\n(v)$ is the valence of $\Gamma$ at $v$ including 
both edges and legs.
\end{enumerate}
An automorphism of $\Gamma$ consists of automorphisms
of the sets $\V$ and $\H$ which leave invariant the
structures $\mathrm{g}$, $\iota$, and $v$ (and hence respect $\E$ and $\L$).
Let $\text{Aut}(\Gamma)$ denote the automorphism group of $\Gamma$.

The genus of a stable graph $\Gamma$ is defined by
$$\g(\Gamma)= \sum_{v\in V} \g(v) + h^1(\Gamma).$$
A boundary stratum of the moduli space $\Mbar_{g,n}$ 
 naturally determines
a stable graph of genus $g$ with $n$ legs by considering the dual graph of a generic pointed curve parameterized by the stratum.

To each stable graph $\Gamma$, we associate the moduli space
\begin{equation*}
\Mbar_\Gamma =\prod_{v\in \V} \Mbar_{\g(v),\n(v)}.
\end{equation*}
There is a
canonical
morphism 
\begin{equation*}
\iota_{\Gamma}: \Mbar_{\Gamma} \rightarrow \Mbar_{g,n}
\end{equation*}
 with image{\footnote{
The degree of $\iota_\Gamma$ is $|\text{Aut}(\Gamma)|$.}}
equal to the closure of the boundary stratum
associated to the graph $\Gamma$.  To construct $\iota_\Gamma$, 
a family of stable pointed curves over $\Mbar_\Gamma$ is required.  Such a family
is easily defined 
by attaching the pullbacks of the universal families over each of the 
$\Mbar_{\g(v),\n(v)}$  along the sections corresponding to half-edges.
Let 
$$[{\Gamma}] \in \mathsf{CH}^\star(
{\overline{\mathcal{M}}_{g,n}})$$
denote the pushforward under $\iota_\Gamma$ of the fundamental
class of $\Mbar_{\Gamma}$.

\subsubsection{Strata algebras} \label{straa}
The strata algebra $\cS_{g,n}^\star$ is defined as the $\Q$-vector space with basis given by the {\em decorated strata classes}
$[\Gamma, \gamma]$ where
\begin{enumerate}
\item[(i)] $\Gamma$ is a 
stable graph corresponding to a stratum of the moduli space, 
$$\iota_\Gamma:\overline{\mathcal{M}}_\Gamma \rightarrow \overline{\mathcal{M}}_{g,n}\, ,$$
\item[(ii)]
$\gamma$ is a product of $\kappa$ and $\psi$
classes on $\overline{\mathcal{M}}_\Gamma$. 
\end{enumerate}
In (ii), the $\kappa$ classes
are associated to the vertices, and the $\psi$ classes are
associated to the half-edges. The only condition imposed is
that the degrees of the $\kappa$ and $\psi$ classes 
associated to a vertex $v\in \V(\Gamma)$ together do {\em not} exceed
the dimension 
$3\g(v)-3+\n(v)$ of the moduli space at $v$.

The strata algebra $\cS_{g,n}^\star$ is of finite dimension as a $\mathbb{Q}$-vector space, graded by
the natural codimension of classes, 
and carries a product for which
the natural pushforward map
\begin{align}\label{v123}
\cS_{g,n}^\star \rightarrow \mathsf{CH}^\star(\overline{\mathcal{M}}_{g,n})\, ,\ \ \ 
[\Gamma,\gamma]  \mapsto (\iota_\Gamma)_\star \gamma
\end{align}
is a homomorphism of graded $\mathbb{Q}$-algebras,
 see \cite[Section 0.3]{PPZ}.
The image of \eqref{v123} is defined to be the subalgebra of {\em tautological classes}
$$\mathsf{R}^\star(\overline{\mathcal{M}}_{g,n})\subset \mathsf{CH}^\star(\overline{\mathcal{M}}_{g,n})\ .$$
The kernel of the quotient map,
$$ \cS_{g,n}^\star
\stackrel{q}{\longrightarrow} \mathsf{R}^\star(\overline{\mathcal{M}}_{g,n}) \longrightarrow 0\ ,$$
is the ideal of {\em tautological relations}.

\subsection{The logarithmic strata algebra of \texorpdfstring{$\Mbar_{g,n}$}{Mbargn}}

We present here a new perspective on the subring of 
tautological classes of the logarithmic Chow ring of the moduli space of curves which is parallel to the
above constructions in Section \ref{straaaa} for the usual Chow ring. While the full foundational development
is given in Sections \ref{sec:hompp} and \ref{sec:logtautgeneral}, the parallel structure of the logarithmic
construction can be seen
without the complete definitions.

Let $\Sigma_{{g,n}}$ be the moduli space
of tropical curves as defined in \cite{CCUW}. The construction of 
$\Sigma_{g,n}$ with the structure of
a cone stack is reviewed in Section \ref{sec:Mgntrop_definitions}.
Associated to a stable graph $\Gamma$,
there is a {\em cone stack with boundary} 
$(\Star_{\Gamma}(\Sigma_{g,n}), \Delta_\Gamma)$ 
associated to the space $\Mbar_\Gamma$ endowed with the strict log structure for the morphism
$$\iota_\Gamma:\overline{\mathcal{M}}_\Gamma \rightarrow \overline{\mathcal{M}}_{g,n}\, .$$
On the level of cone stacks, we define
\begin{equation}
    \Star_{\Gamma}(\Sigma_{g,n}) = \left( \prod_{v \in V(\Gamma)} \Sigma_{g(v), n(v)} \right) \times (\RR_{\geq 0})^{E(\Gamma)}\,,
\end{equation}
parameterizing a tuple of tropical curves $\Gamma_v$ (for each vertex $v \in V(\Gamma)$) and edge lengths $\ell_e \geq 0$ (for each edge $e \in E(\Gamma)$. The natural tropical gluing map 
$$ \Star_{\Gamma}(\Sigma_{g,n}) \to \Sigma_{g,n}$$ 
connects the various graphs $\Gamma_v$ with edges of lengths $\ell_e$. Via the gluing map, the cone stack $\Star_{\Gamma}(\Sigma_{g,n})$ defines a finite cover of the {\em star}\footnote{The
star of $\sigma_\Gamma$ is the set of cones of $\Sigma_{g,n}$ containing $\sigma_\Gamma$ as a face.}
in $\Sigma_{g,n}$ of
the cone $\sigma_\Gamma \in \Sigma_{g,n}$ associated to $\Gamma$. 
The boundary $\Delta_\Gamma$ of $\Star_{\Gamma}(\Sigma_{g,n})$ then consists of all cones $((\sigma_{\Gamma_v})_{v \in V(\Gamma)}, \tau)$  of $\Star_{\Gamma}(\Sigma_{g,n})$ such that $\tau \precneq (\RR_{\geq 0})^{E(\Gamma)}$ is a \emph{proper} face of $(\RR_{\geq 0})^{E(\Gamma)}$.

There is a canonical
$\mathbb{Q}$-vector space $\pPP_\star(\mathsf{Star}_\Gamma(\Sigma_{g,n}), \Delta_\Gamma)$ of {\em homological piecewise polynomials} on $\Star_{\Gamma}(\Sigma_{g,n})$ defined in Section \ref{sec:hompp}. The {\em homological} condition here requires the piecewise polynomials to
vanish on the boundary $\Delta_\Gamma$ of the cone stack, in particular making $\pPP_\star(\mathsf{Star}_\Gamma(\Sigma_{g,n}), \Delta_\Gamma)$ a module over the $\QQ$-algebra $\sPP^\star(\mathsf{Star}_\Gamma(\Sigma_{g,n}))$ of all strict piecewise polynomials.
The strata classes we will consider in the
logarithmic context carry the additional decoration
of a homological piecewise polynomial.

Let $\Gamma$ be a stable graph of genus $g$ with $n$ markings.
A {\em decorated log strata class}
$[\Gamma, f, \gamma]$ is defined by the following conditions:
\begin{enumerate}
\item[(i)] $\Gamma$ is a stable graph,
\item[(ii)] $f$ is a homological piecewise
polynomial on
($\Star_{\Gamma}(\Sigma_{g,n}), \Delta_\Gamma)$,
\item[(iii)] $\gamma = \prod_{v \in V(\Gamma)} \gamma_v$ is a product of decorated strata classes on the vertices of $\Gamma$.\footnote{For the discussion below and in particular equation \eqref{eqn:logSstar_introduction}, it is convenient to generalize the decorations $\gamma$ from products of $\kappa$ and $\psi$-classes to arbitrary decorated strata classes. As discussed in Remark \ref{Rmk:strata_log_lift} (c) below, we obtain the same space of log tautological classes by restricting to decorations $\gamma$ which are products of $\kappa$ and $\psi$-classes as in Section \ref{straa}.}
\end{enumerate}
We define the \emph{logarithmic strata algebra} $\mathsf{log}\cS_{g,n}^\star$ as the (in general infinite-dimensional) $\mathbb{Q}$-vector space
\begin{equation} \label{eqn:logSstar_introduction}
\mathsf{log}\cS_{g,n}^\star = \bigoplus_{\Gamma} \left( \pPP_\star(\mathsf{Star}_\Gamma(\Sigma_{g,n}), \Delta_\Gamma) \otimes_{\sPP^\star(\mathsf{Star}_\Gamma(\Sigma_{g,n}))} \bigotimes_{v \in V(\Gamma)} \cS_{g(v),n(v)}^\star \right)\,.
\end{equation}
We pause for a moment to explain the pieces of this formula. The sum is over all isomorphism classes $\Gamma$ of stable graphs. For each summand, the latter factors in the tensor product are vertex terms, each corresponding to the ordinary strata algebras of the moduli space ``at that vertex''. Since the stratum associated to a graph is a quotient of a product over vertices, we take the tensor product over these vertices. The initial term is the group of {\it homological piecewise polynomials} noted above, and defined in the main text. To make sense of the tensor product, we use the existence of a natural map
\[
\sPP^\star(\mathsf{Star}_\Gamma(\Sigma_{g,n})) \to \bigotimes_{v \in V(\Gamma)} \cS_{g(v),n(v)}^\star
\]
as is explained in Remark \ref{Rmk:strata_log_lift} below.

The algebra $\mathsf{log}\cS_{g,n}^\star$ is graded by
codimension
and carries a product for which
the natural pushforward{\footnote{Pushforward  in the logarithmic context is a delicate operation.
The boundary vanishing of homological piecewise polynomials allows an extension by 0, see Section \ref{Sect:push_forward_PP}.}} map
\begin{equation}\label{v123log}
\mathsf{log}\cS_{g,n}^\star \rightarrow \mathsf{logCH}^\star(\overline{\mathcal{M}}_{g,n})\, , \ \ \ \sum_\Gamma f_\Gamma \otimes \gamma_\Gamma \mapsto \sum_\Gamma\  [\Gamma, f_\Gamma, \gamma_\Gamma]
\end{equation}
is a homomorphism of graded $\mathbb{Q}$-algebras,
 see Theorem \ref{thm:logSproduct}. 
The image of \eqref{v123log} is defined to be the subalgebra of {\em logarithmic tautological classes}
$$\mathsf{logR}^\star(\overline{\mathcal{M}}_{g,n})\subset \mathsf{logCH}^\star(\overline{\mathcal{M}}_{g,n})\ .$$
The kernel of the quotient map,
$$ \mathsf{log}\cS_{g,n}^\star
\stackrel{q}{\longrightarrow} \mathsf{logR}^\star(\overline{\mathcal{M}}_{g,n}) \longrightarrow 0\ ,$$
is the ideal of {\em logarithmic tautological relations}. 

As in the case of the standard strata algebra
$\mathsf{S}^\star_{g,n}$, the logarithmic
strata algebra $\mathsf{logS}^\star_{g,n}$ 
is a natural setting for both theoretical results and calculations: the
$\mathbb{Q}$-vector space structure and the product
are completely explicit (and can be implemented on
computational interfaces). The complexity of the
tautological cycle theory in both the standard
and log cases lies
in the kernel of $q$.

There are two simple sources of logarithmic tautological classes.
First, every tautological class on $\Mbar_{g,n}$ lifts to
a logarithmic class via a commutative diagram of $\mathbb{Q}$-algebras (as explained
in Remark \ref{Rmk:strata_log_lift}):
\[
\begin{tikzcd}
\mathsf{S}^\star_{g,n}  \arrow[r,"q"] \arrow[d] & \mathsf{R}^\star(\Mbar_{g,n}) \arrow[d]\\
\mathsf{logS}^\star_{g,n} \arrow[r, "q"] & \mathsf{logR}^\star(\Mbar_{g,n})\,.
\end{tikzcd}
\]
Second, there is a canonical map from the $\mathbb{Q}$-algebra
of piecewise polynomials of $\Sigma_{g,n}$,
$$\Phi_{g,n}^\mathrm{log}:\mathsf{PP}^\star(\Sigma_{g,n})\rightarrow
\mathsf{logR}^\star(\Mbar_{g,n})\, ,$$
as discussed in Section \ref{sec:Mgntrop_definitions}.
Since both $\mathsf{R}^\star(\Mbar_{g,n})$
and $\mathsf{PP}^\star(\Sigma_{g,n})$ are algebras over 
the $\mathbb{Q}$-algebra 
of strict piecewise polynomials $\mathsf{sPP}^\star(\Sigma_{g,n})$, we
obtain a canonical homomorphism, 
$$\mu_{g,n}^{\mathsf{R}}:\mathsf{R}^\star(\Mbar_{g,n}) \otimes_{\mathsf{sPP}^\star(\Sigma_{g,n})}
\mathsf{PP}^\star(\Sigma_{g,n})
\rightarrow
\mathsf{logR}^\star(\Mbar_{g,n})\, .$$

\subsection{Results in genus 0,1, and higher genus}

In addition to the foundational development of the theory of decorated log strata classes, 
we present several results about the structure of
$\mathsf{logR}^\star(\Mbar_{g,n})$:

\vspace{10pt}
\noindent $\bullet$ A complete calculation of
$\mathsf{logR}^\star(\Mbar_{0,n})$ is given in Section \ref{sec:genus0} by the following result.

\begin{theorem} \label{Thm:logCHMbar0n-F}
    The map $\mu^{\mathsf{R}}_{0,n}$ is an isomorphism,
$$\mu_{0,n}^{\mathsf{R}}:\mathsf{R}^\star(\Mbar_{0,n}) \otimes_{\mathsf{sPP}^\star(\Sigma_{0,n})}
\mathsf{PP}^\star(\Sigma_{0,n})\, 
\stackrel{\sim}{\rightarrow}\
\mathsf{logR}^\star(\Mbar_{0,n})\, .$$
\end{theorem}

In other words, 
$\mathsf{logR}^\star(\Mbar_{0,n})$ is canonically
isomorphic to the algebra 
of piecewise polynomials
on the Artin fan 
$\Sigma_{\Mbar_{0,n}}$
of $\Mbar_{0,n}$ modulo{\footnote{Keel's disjoint boundary divisor equations already hold in piecewise polynomials.}} the WDVV relations. The result can be viewed as a logarithmic lift of Keel's calculation \cite{Kee92} of $\mathsf{CH}^\star(\Mbar_{0,n})$, since
$$\mathsf{logR}^\star(\Mbar_{0,n})=\mathsf{logCH}^\star(\Mbar_{0,n})\, .$$
Our method of proof uses Kapranov's approach to $\Mbar_{0,n}$ via the Chow quotient
of the Grassmannian \cite{Kap93} and related results of Gibney--Maclagan \cite{GM07}, Hacking--Keel--Tevelev \cite{HKT}, and Tevelev~\cite{Tev07}.

\vspace{10pt}
\noindent $\bullet$ 
The genus 1 case is studied in Section \ref{sec:mapsspaces}.
We
prove $\mu_{1,n}^{\mathsf{R}}$ is always surjective in Proposition \ref{Pro:logR_pp_genus_1},
but $\mu_{1,n}^{\mathsf{R}}$ is {\em not} in general an isomorphism. Nontrivial elements of the 
kernel of $\mu_{1,n}^{\mathsf{R}}$
are found in Proposition \ref{Prop:mu_1_n_not_injective} for $n\geq 3$. How to write  a simple and explicit set of generators of the kernel in genus 1 is
an open question.

\vspace{10pt}
\noindent $\bullet$ For $g\geq 7$, the map 
$\mu_{g,n}^{\mathsf{R}}$ is not surjective by Proposition \ref{Pro:mugn_not_surjective}, so
a different approach must be taken to control $\mathsf{logR}^\star(\Mbar_{g,n})$
in high genus.

\vspace{10pt}
\noindent $\bullet$ We prove in Theorem \ref{Thm:logR_det_by_Pixton} of Section
\ref{sec:highergenus}
that 
the
structure of $\mathsf{logR}^\star(\Mbar_{g,n})$
is determined
by the complete knowledge of {all}
relations in the standard tautological
rings of (products of) moduli spaces of curves.
In particular, Pixton's conjecture \cite{Pixton_relations}
concerning relations in the standard tautological rings yields
a {\em complete conjecture for all relations in the log tautological rings of
the moduli spaces of curves}.

\vspace{10pt}

A fundamental open question (related to the missing
presentation of $\mathsf{logR}^\star(\Mbar_{1,n})$)
is whether there exists a non-trivial logarithmic lift of the
formula of Pixton's relations. Pixton's DR cycle
relations have non-trivial logarithmic lifts
(specified by a choice of stability condition on
line bundles on curves) obtained from the
study of the logarithmic DR cycle \cite{HMPPS}.

\subsection{Log tautological rings of arbitrary normal crossing pairs}

The constructions for $\Mbar_{g,n}$ are valid in a more general setting. For any nonsingular DM stack $X$ with a normal crossings divisor $D$, there is a natural log structure for the pair $(X,D)$ and a natural notion of strata. 
In Section~\ref{sec:logtautgeneral}, we define log decorated strata classes, the log strata algebra, and the log tautological ring for $(X,D)$.

Our construction takes as input a \emph{tautological system}
on $X$,  a set of subrings \[ \big\{ \, \mathsf{R}^\star(P) \subset \CH^\star(P)\, \big\}_P\, , \] for every stratum closure $P$ in $X$, which is required to be closed under pullback and pushforward along strata inclusions.\footnote{If $D$ is not simple normal crossings, $P$ is not a
stratum closure, but a certain monodromy torsor over the normalization of a stratum closures. See Section \ref{sec:logtautgeneral} for details on these monodromy torsors and  Definition \ref{def:tautsystems} for the formal definition of a tautological system.} For example, the log tautological ring of $\Mbar_{g,n}$ discussed above is 
obtained from the tautological system on
$(\Mbar_{g,n},\Delta)$ defined 
by taking $\mathsf{R}^\star(\Mbar_{\Gamma})$ to be the tautological ring of $\Mbar_{\Gamma}$ in the classical sense.

Another important special case of a tautological system is the \emph{Chow system}, where $\mathsf{R}^\star(P) = \CH^\star(P)$ for all $P$. We can then write a  presentation of $\logCH^\star(X,D)$.

\begin{theorem} \label{Thm:logTautX_intro}
The log tautological ring $\mathsf{logR}^\star(X,D)$ induced by the Chow system is equal to the full log Chow ring $\logCH^\star(X,D)$. Hence the map $\logS^\star(X,D) \to \logCH^\star(X)$ is a surjective ring morphism, and the log decorated strata are additive generators for $\logCH^\star(X,D)$.
\end{theorem}

Technical innovations required here  include the notions of \emph{idealised Artin fans} (Definition \ref{def:idealisedArtinfan}) and \emph{cone stacks with boundary} (Definition \ref{def:conestack}). Just as
Artin fans and cone stacks capture the behaviour of log smooth stacks, idealised Artin fans and cone stacks with boundary capture the behaviour of strata of log smooth stacks, and more generally, of logarithmic stacks that are smooth over substacks of Artin fans. 
In Section~\ref{sec:hompp}, we work out their theory, as well as the theory of \emph{homological} piecewise polynomials $\sPP_\star(\Sigma, \Delta)$ on cone stacks $(\Sigma, \Delta)$ with boundaries. Elements of $\sPP_\star(\Sigma, \Delta)$ probe the intersection theory of idealised log smooth stacks, in a similar manner to how piecewise polynomials probe the intersection theory of log smooth stacks. The \emph{homological} condition is the requirement that piecewise polynomials on $\Sigma$ vanish on the boundary $\Delta$. 
Among the new ideas are introduced in Section~\ref{sec:hompp} are the following:
\begin{itemize}
    \item Given a log stack $X$, we define the notion of a \emph{choice of an Artin fan} $X \to \mathcal{A}$ (Definition \ref{Def:an_Artin_fan}), generalizing the construction of the (canonical) Artin fan $\mathcal{A}_X^\textup{can}$ of \cite{ACMUW}. The new notion has better functoriality properties than $\mathcal{A}_X^\textup{can}$, which allow us 
    much more flexibility
(see also \cite[Remark 5]{HMPPS} for a related discussion).
    \item There is a natural correspondence between cone stacks $(\Sigma, \Delta)$ with boundary and pairs $(\mathcal{A}, \mathcal{B})$ of an Artin fan $\mathcal{A}$ and a closed reduced substack $\mathcal{B} \subseteq \mathcal{A}$. Under the correspondence, we prove an identification (in  Theorem \ref{Thm:sPP_lowerstar_isom})
    of the Chow group of $\mathcal{B}$ with the homological piecewise polynomials on $(\Sigma, \Delta)$, 
    \begin{equation} \label{eqn:CH_spp_correspondence_intro}
    \CH_\star(\mathcal{B}) \cong \sPP_\star(\Sigma, \Delta)\, ,
    \end{equation}
        generalizing the case $\mathcal{A} = \mathcal{B}$ from \cite[Theorem 14]{MPS23} .
    \item We give a combinatorial formula for proper pushforwards of homological piecewise polynomials and show that, under the identification \eqref{eqn:CH_spp_correspondence_intro}, we obtain the usual proper pushforwards of Chow groups (Proposition \ref{prop:pushforwardppstack}).
\end{itemize}
All of these tools are used extensively in the general construction of log tautological classes on $(X,D)$ and the proof of Theorem \ref{Thm:logTautX_intro}
in Section \ref{sec:logtautgeneral}.



\subsection*{Acknowledgments}
We thank Renzo Cavalieri, Tom Graber,
Leo Herr, Diane Maclagan, Sam Molcho, Aaron Pixton, and Jonathan Wise for many conversations related to the logarithmic Chow 
ring. 
We are grateful to Younghan Bae for detailed answers to technical questions on higher Chow groups and to David Holmes for extensive discussions about homological piecewise polynomials and cone stacks with boundaries (and for proposing their study for idealised log smooth schemes).

R.P. was supported by
SNF-200020-182181,
SNF-200020-219369,
ERC-2017-AdG-786580-MACI, and SwissMAP. D.R was supported by EPSRC New Investigator Award EP/V051830/1 and  EPSRC Frontier Research Grant EP/Y037162/1.
 J.S. was supported by SNF Early Postdoc
Mobility grant 184245 held at the Max Planck Institute for Mathematics in Bonn,
SNF-184613 at the University of Z\"urich, SNF-200020-219369,
and SwissMAP. 
P.S. was supported by NWO grant VI.Vidi.193.006 and ERC Consolidator Grant FourSurf 101087365.
The research presented here was completed in part 
at the SwissMAP center at Les Diablerets during
the {\em Helvetic Algebraic Geometry Seminar} in June 2023.

This project has
received funding from the European Research Council (ERC) under the European Union Horizon 2020 research and innovation program (grant agreement No. 786580).

\section{Logarithmic tautological rings of moduli spaces of curves}
While the general treatment of log tautological rings of normal crossing pairs $(X,D)$ is covered in Section \ref{sec:logtautgeneral} below, we begin here by introducing the fundamental case of the moduli space of curves
$(\Mbar_{g,n},\Delta)$.

\subsection{Definitions and comparisons} \label{sec:Mgntrop_definitions}

\subsubsection{Cone stacks} \label{Sect:Cone_stacks_Mbar}
The logarithmic Chow ring of $\Mbar_{g,n}$ is defined as the colimit of Chow rings of iterated boundary blowups. The combinatorics of the boundary stratification is captured by the cone stack $\Sigma_{g,n}$ of tropical curves, constructed in~\cite{CCUW}. 

\begin{definition}    \label{Def:Mgntrop}
The {\emph {cone stack}} $\Sigma_{g,n}$ is the collection of cones
\[
\sigma_\Gamma = \{\ell : E(\Gamma) \to \mathbb{R}_{\geq 0}\} \cong (\mathbb{R}_{\geq 0})^{E(\Gamma)} \, \text{ for $\Gamma$ a stable graph of $\Mbar_{g,n}$}\,,
\]
together with face inclusion morphisms
\[
\iota_{\varphi} : \sigma_\Gamma \to \sigma_{\Gamma'} \, \text{ for $\varphi: \Gamma' \to \Gamma$ a morphism of stable graphs}\,.
\]
Viewing $\Gamma$ as an edge-contraction of $\Gamma'$ via $\varphi$, the map $\iota_\varphi$ includes $\sigma_\Gamma$ as the face of $\sigma_{\Gamma'}$ where the lengths of all contracted edges are set to zero.
\end{definition}

An \emph{Artin fan} $\mathcal{A}_{\Sigma_{g,n}}$
associated to the cone stack $\Sigma_{g,n}$ is constructed in \cite{CCUW}.
The Artin fan is a smooth algebraic stack which has a locally closed stratification by 1-point subsets,
$$\mathcal{S}_\Gamma \cong B(\mathbb{G}_m^{E(\Gamma)} \rtimes \Aut(\Gamma)) \subseteq \mathcal{A}_{\Sigma_{g,n}}\, ,$$
for each isomorphism class $\Gamma$ of stable graphs. The union of the $\mathcal{S}_\Gamma$ for $\Gamma$ non-trivial forms a normal crossing divisor. The stack $\mathcal{A}_{\Sigma_{g,n}}$ receives a strict, smooth, and surjective map, 
\begin{equation} \label{eqn:q_map_intro}
\afanmap: \Mbar_{g,n} \to \mathcal{A}_{\Sigma_{g,n}}\, ,
\end{equation}
which satisfies the following property: the preimage of $\mathcal{S}_\Gamma$ is precisely the locally closed stratum of curves of dual graph $\Gamma$ in $\Mbar_{g,n}$ (see  \cite[Theorem 4]{CCUW}).

Given a subdivision $\widehat \Sigma \to \Sigma_{g,n}$ of cone stacks (specified by a collection of fans with supports $\sigma_\Gamma$ compatible under the face inclusion morphisms $\iota_\varphi$), we obtain a log blowup $\widehat{\mathcal{A}} \to \mathcal{A}_{\Sigma_{g,n}}$ via the equivalence of categories between cone stacks and Artin fans (\cite[Theorem 3]{CCUW}). Taking a fiber product 
\[
\begin{tikzcd}
\widehat{\mathcal{M}}\arrow[r] \arrow[d] & \Mbar_{g,n}\arrow[d, "\afanmap"]\\
\widehat{\mathcal{A}} \arrow[r] & \mathcal{A}_{\Sigma_{g,n}}
\end{tikzcd}
\]
with respect to the map $\afanmap$ above, we obtain a log blowup $\widehat{\mathcal{M}} \to \Mbar_{g,n}$. Conversely,  every log blowup of $\Mbar_{g,n}$ is obtained from such a fiber product 
associated to a subdivision of $\Sigma_{g,n}$.
We refer the reader to \cite[Section 3]{ACMW} and \cite{CCUW} for further details on Artin fans and the associated cone complexes.

A central tool for constructing cycle classes on $\Mbar_{g,n}$ (and on the log blowups $\widehat{\mathcal M}$) is a complete description of the intersection theory of Artin fans $\mathcal{A}_\Sigma$  in terms of the 
$\mathbb{Q}$-algebras
$\sPP^\star(\Sigma)$ of \emph{strict piecewise polynomials} on the associated cone stack $\Sigma$. An element of $\sPP^\star(\Sigma)$ is a collection of $\mathbb{R}$-valued polynomial functions on the cones $\sigma \in \Sigma$ (with $\mathbb{Q}$-coefficients) compatible under all face inclusion morphisms in $\Sigma$. An isomorphism
\begin{equation} \label{eqn:Phi_intro}
\Phi : \sPP^\star(\Sigma) \to \CH^\star(\mathcal{A}_\Sigma)
\end{equation}
was constructed in \cite[Theorem 14]{MPS23}. 

To extend $\Phi$, we define the ring $\pPP^\star(\Sigma)$ of \emph{piecewise polynomials} as the set of functions on the cones $\sigma \in \Sigma$ which become strict piecewise polynomial on \emph{some} subdivision $\widehat \Sigma \to \Sigma$. In other words, it is the direct limit of strict piecewise polynomial rings taken over all subdivisions, with transition maps given by pulling back along a subdivision. The collection of the maps $\Phi$ for such subdivisions $\widehat \Sigma$ then induces an isomorphism
\begin{equation} \label{eqn:Phi_log_intro}
\Phi^\mathrm{log} : \pPP^\star(\Sigma) \to \logCH^\star(\mathcal{A}_\Sigma)\,.
\end{equation}
Since a class in logarithmic Chow is an ordinary Chow class on some blowup, and blowups are proper, we have natural projection morphisms induced by proper pushforward:
\[
 \logCH^\star(\mathcal{A}_\Sigma)\to  \mathsf{CH}^\star(\mathcal{A}_\Sigma).
\]
This pushforward splits the natural injection $\mathsf{CH}^\star(\mathcal{A}_\Sigma)\hookrightarrow  \logCH^\star(\mathcal{A}_\Sigma)$. There are corresponding operations on piecewise polynomials.

Returning to $\Mbar_{g,n}$ and the cone stack $\Sigma=\Sigma_{g,n}$, we can compose the two maps \eqref{eqn:Phi_intro} and \eqref{eqn:Phi_log_intro} with pullback by the map $\afanmap$  of \eqref{eqn:q_map_intro} to obtain $\mathbb{Q}$-algebra homomorphisms
\begin{equation} \label{eqn:Phi_log_intro_eqn}
    \sPP^\star(\Sigma_{g,n}) \xrightarrow{\Phi_{g,n}} \CH^\star(\Mbar_{g,n}) \ \text{ and } \ \pPP^\star(\Sigma_{g,n}) \xrightarrow{\Phi^\mathrm{log}_{g,n}} \logCH^\star(\Mbar_{g,n})\,.
\end{equation}
The map $\Phi_{g,n}$ factors through  $\R^\star(\Mbar_{g,n}) \subseteq \CH^\star(\Mbar_{g,n})$, with image spanned by \emph{normally decorated strata classes} (fundamental classes of strata closures decorated by Chern classes of their normal bundles, see \cite[Theorem 13]{MPS23}).

\subsubsection{Definitions}
Using the map $\Phi^\mathrm{log}_{g,n}$, we construct the first (and smallest) type of logarithmic tautological rings of $\Mbar_{g,n}$.
\begin{definition} \label{ooo} 
The {\em piecewise polynomial tautological ring} $\mathsf{logR}^\star_{\mathrm{pp}}(\Mbar_{g,n})$ is the image
\[
\mathsf{logR}^\star_{\mathrm{pp}}(\Mbar_{g,n}) = \Phi_{g,n}(\pPP^\star(\Sigma_{g,n})) \subset
\mathsf{logCH}^\star(\Mbar_{g,n})
\]
of the piecewise polynomials on the cone stack $\Sigma_{g,n}$.
\end{definition}

The top Chern class $\lambda_g$ of the Hodge bundle over $\Mbar_{g,n}$
is an example of an
interesting class contained in the piecewise polynomial tautological ring, see \cite[Theorem 6]{MPS23}. However, in general, even basic tautological classes like $\kappa_1 \in \R^\star(\Mbar_{g,n})$ are not contained in $\mathsf{logR}^\star_{\mathrm{pp}}(\Mbar_{g,n})$. Indeed, the restriction of the latter ring to $\mathcal{M}_{g,n}$ is just the span of the fundamental class $[\mathcal{M}_{g,n}]$.




To include the $\kappa$ classes, observe that both $\mathsf{R}^\star(\Mbar_{g,n})$
and $\mathsf{PP}^\star(\Sigma_{g,n})$ are algebras over 
the $\mathbb{Q}$-algebra 
of strict piecewise polynomials $\mathsf{sPP}^\star(\Sigma_{g,n})$. We therefore
obtain a canonical homomorphism, 
$$\mu_{g,n}^{\mathsf{R}}:\mathsf{R}^\star(\Mbar_{g,n}) \otimes_{\mathsf{sPP}^\star(\Sigma_{g,n})}
\mathsf{PP}^\star(\Sigma_{g,n})
\rightarrow
\logCH^\star(\Mbar_{g,n})\, .$$

\begin{definition} \label{Def:small_taut_ring} 
The {\em small tautological ring} 
$\mathsf{logR}^\star_{\mathrm{sm}}(\Mbar_{g,n})$
is the $\mathbb{Q}$-subalgebra
\[
\mathsf{logR}^\star_{\mathrm{sm}}(\Mbar_{g,n}) = \mu_{g,n}^{\mathsf{R}}(\mathsf{R}^\star(\Mbar_{g,n}) \otimes_{\mathsf{sPP}^\star(\Sigma_{g,n})}
\mathsf{PP}^\star(\Sigma_{g,n})) \subset
\mathsf{logCH}^\star(\Mbar_{g,n})
\]
generated by tautological classes $\R^\star(\Mbar_{g,n}) \subseteq \CH^\star(\Mbar_{g,n}) \subseteq \logCH^\star(\Mbar_{g,n})$ and classes coming from piecewise polynomials on $\Sigma_{g,n}$.
\end{definition}
As mentioned above, the existence of $\kappa$- and $\psi$-classes which are non-trivial on the interior 
${\mathcal M}_{g,n} \subset \Mbar_{g,n}$ implies that 
the inclusion $\mathsf{logR}^\star_{\mathrm{pp}}(\Mbar_{g,n}) \subset
\mathsf{logR}^\star_{\mathrm{sm}}(\Mbar_{g,n})$ is, in general,
strict. 
Notable examples of classes contained in the small tautological ring are the (logarithmic) double ramification cycles
$$\mathsf{DR}_{g,A}
\in \mathsf{CH}^g(\Mbar_{g,n})\ \ \ \text{and} \ \ \ 
\mathsf{logDR}_{g,A} \in \mathsf{logCH}^g(\Mbar_{g,n})\,,$$
as proven in \cite[Theorem 19]{MPS23}, \cite{MR21} and \cite[Theorem 4.22]{HS22}.
More generally, Molcho \cite{Mol22} shows that Abel--Jacobi pullbacks of Brill--Noether classes, of which the double ramification cycle is just one example, lie in the small logarithmic tautological ring. 

While the ring $\mathsf{logR}^\star_\mathrm{sm}(\Mbar_{g,n})$ allows us to combine normally decorated strata classes in log blowups of $\Mbar_{g,n}$ with tautological classes, a weakness of the definition is that the tautological class must always be defined on all of $\Mbar_{g,n}$. For example, we could blow up  a boundary stratum and then decorate the exceptional divisor with a product of $\kappa$- and $\psi$-classes only defined on the stratum itself. Developing a formalism for such logarithmic classes is a central motivation of the present paper.

To formulate the corresponding construction, we first prove\footnote{See Corollary \ref{Cor:Psigma_cartesian_diagram}, which is formulated in the case of arbitrary pairs $(X,D)$ of a smooth DM-stack and normal crossings divisor $D$.} that given a stable graph $\Gamma$ on $\Mbar_{g,n}$ there exists an Artin stack $\mathcal{P}_\Gamma$ and a map $j_\Gamma : \mathcal{P}_\Gamma \to \mathcal{A}_{\Sigma_{g,n}}$ such that we have a fiber diagram
\begin{equation} \label{5ttgd}
\begin{tikzcd}
\Mbar_\Gamma \arrow[r, "\iota_\Gamma"] \arrow[d, "\afanmap_\Gamma"] & \Mbar_{g,n} \arrow[d,"\afanmap"]\\
\mathcal{P}_\Gamma \arrow[r, "j_\Gamma"] & \, \mathcal{A}_{\Sigma_{g,n}}
\end{tikzcd}
\end{equation}
with both vertical maps being smooth surjections.
The morphism $j_\Gamma$ is a finite cover of the closure $\overline{\mathcal{S}}_\Gamma \subseteq \mathcal{A}_{\Sigma_{g,n}}$ of the stratum ${\mathcal{S}}_\Gamma$ associated to $\Gamma$. 
Via the diagram \eqref{5ttgd}, the morphism $j_\Gamma$ can be seen as a smooth local model of the gluing morphism $\iota_\Gamma$ inside the Artin fan $\mathcal{A}_{\Sigma_{g,n}}$.
The natural cone stack associated to the strict
log structures on $\Mbar_\Gamma$ and $\mathcal{P}_\Gamma$
induced by the horizontal maps of \eqref{5ttgd} 
is the product
\begin{equation} \label{eqn:Sigma_Gamma}
\Star_{\Gamma}(\Sigma_{g,n}) :=\Sigma_\Gamma := \left(\prod_{v \in V(\Gamma)} \Sigma_{g(v),n(v)}\right) \times \mathbb{R}_{\geq 0}^{E(\Gamma)}\,.
\end{equation}
In general, the stacks $\Mbar_\Gamma$ and $\mathcal{P}_\Gamma$ are not log smooth with the above induced log structure, since these log structures can be generically nontrivial.
However, they are \emph{idealized log smooth} (see Definition \ref{Def:idealized_log_smooth}), which roughly means that they are cut out inside a log smooth space by an ideal that is monomial with respect to the log structure.\footnote{Just as the basic model for a log smooth space is a toric variety, the basic model for an idealised log smooth space is a torus invariant subscheme inside a toric variety.} 
On the combinatorial side, the existence of this monomial ideal is reflected by the fact that $\Sigma_\Gamma$ carries the structure of a \emph{cone stack with boundary} (see Definition \ref{def:conestack}). 
The boundary $\Delta_\Gamma$ of $\Sigma_\Gamma$ consists of the collection of cones in $\Sigma_\Gamma$ of the form
\[
\prod_{v \in V(\Gamma)} \sigma_{v} \times \tau
\]
for $\tau \precneq \mathbb{R}_{\geq 0}^{E(\Gamma)}$ a \emph{proper} face of $\mathbb{R}_{\geq 0}^{E(\Gamma)}$.
The significance of the boundary is explained in Theorem \ref{Thm:sPP_lowerstar_isom} where we show that 
parallel to equation \eqref{eqn:Phi_log_intro_eqn} we have isomorphisms
\begin{equation}
\sPP_\star(\Sigma_{\Gamma}, \Delta_\Gamma) \xrightarrow{\Psi_{\Gamma}} \CH_\star(\mathcal{P}_\Gamma) \ \text{ and }\ \pPP_\star(\Sigma_{\Gamma}, \Delta_\Gamma) \xrightarrow{\Psi^\mathrm{log}_{\Gamma}} \logCH_\star(\mathcal{P}_\Gamma)\,,
\end{equation}
where $\mathsf{(s)PP}_\star(\Sigma_{\Gamma}, \Delta_\Gamma)$ denotes the set of (strict) piecewise polynomials on $\Sigma_{\Gamma}$ vanishing on all cones of $\Delta_\Gamma$. Such functions are called \emph{homological piecewise polynomials} below, and correspondingly $\logCH_\star(\mathcal{P}_\Gamma)$ denotes the homological log Chow group of $\mathcal{P}_\Gamma$ defined in \cite{Barrott2019Logarithmic-Cho}.\footnote{We warn the reader that $\logCH_\star(X)$ of a log scheme $X$ that is not log smooth is not necessarily as well behaved as the log Chow ring of a log smooth log scheme. For example, $\logCH_\star(X)$ is not necessarily supported in degrees $0$ through $\dim X$ even if $\dim X = 0$. See Example~\ref{ex:ptwithnr} and Proposition~\ref{Pro:LogCH_pt_Nr} for examples on logarithmic Chow groups and homological piecewise polynomials.} Since such homological log Chow classes admit flat pullbacks and proper pushforwards along log maps, we can use them to define our third ring of log tautological classes.

Given $f \in \pPP_\star(\Sigma_{\Gamma}, \Delta_\Gamma)$ and $\gamma \in \R^\star(\Mbar_\Gamma)$, we define a \emph{log decorated stratum class} as the cycle
\begin{equation}
[\Gamma, f, \gamma] = (\iota_\Gamma)_\star\left(\gamma \cdot \afanmap_\Gamma^\star \Psi^\mathrm{log}_{\Gamma}(f)  \right) \in \logCH^\star(\Mbar_{g,n})\,.
\end{equation}
Before moving on with the general theory, let us list some examples and properties of these log decorated stratum classes:
\begin{enumerate}
    \item[a)] Let $\Gamma$ be a stable graph with precisely two vertices $v_1, v_2$ connected by a pair of edges (associated to a stratum $\mathcal{M}^\Gamma \subseteq \Mbar_{g,n}$ of codimension $2$). On $\Sigma_\Gamma = \Sigma_{g(v_1), n(v_1)} \times \Sigma_{g(v_2), n(v_2)} \times \mathbb{R}_{\geq 2}^2$ consider the piecewise polynomial function $f=\min(x,y)$, where $x,y$ are the coordinates on the last factor $\mathbb{R}_{\geq 2}^2$. Then the class $\mathfrak{t}_\Gamma^\star \Psi_\Gamma^\textup{log}(f) \in \logCH^\star(\Mbar_\Gamma)$ is given by the fundamental class $[\widehat{\mathcal{M}}_\Gamma]$ of a log blowup $p: \widehat{\mathcal{M}}_\Gamma \to \Mbar_\Gamma$, where $p$ is a $\PP^1$-bundle. If we set the decoration $\gamma=1$, then $$[\Gamma, \min(x,y), 1] = (\iota_\Gamma)_\star [\widehat{\mathcal{M}}_\Gamma]$$
    maps to a multiple of the exceptional divisor of the blowup of $\overline{\mathcal{M}}^\Gamma$ inside $\Mbar_{g,n}$. For an arbitrary decoration $\gamma$, we would just replace the fundamental class $[\widehat{\mathcal{M}}_\Gamma]$ by $(p^\star \gamma) \cap [\widehat{\mathcal{M}}_\Gamma]$ in this formula. We see that the intuition of the formalism is to allow us to combine log Chow classes from piecewise polynomials (like $[\widehat{\mathcal{M}}_\Gamma]$) with decorations $\gamma$ that are only defined on the domain $\Mbar_\Gamma$ of the gluing map $\iota_\Gamma$. For more details and an application of this example, see the proof of Proposition \ref{Pro:mugn_not_surjective}.
    \item[b)] In general, when the decoration $\gamma=1$ is trivial, the class $[\Gamma, f, 1]$ can be calculated as
    \[
    [\Gamma, f, 1] = \Phi_{g,n}^\textup{log}( g ) \in \logCH^\star(\Mbar_{g,n}) \text{ for an explicit } g = (\iota_\Gamma^\textup{trop})_\star f \in \pPP^\star(\Mbar_{g,n}),
    \]
    see Proposition \ref{Prop:sPP_pushforward} (where the notation $(\iota_\Gamma^\textup{trop})_\star$ is explained for the tropical pushforward associated to $\iota_\Gamma$ is explained). 
    This tropical pushforward $g$ of $f$ can be calculated using a formula adapted from \cite{Brion} (see Proposition \ref{prop:pushforwardppstack}). In the simplest situation when the map $\Sigma_\Gamma \to \Sigma_{g,n}$ is just an inclusion of a cone sub-complex (which happens, for example, when $g=0$), the function $g$ is just the extension of $f$ by zero on all cones not in the image of $\Sigma_\Gamma$. The condition that $f$ vanishes on the boundary of $\Sigma_\Gamma$ is exactly what makes this a well-defined piecewise polynomial.
    \item[c)] As with the traditional tautological subrings \cite{GP03} of decorated strata classes on $\Mbar_{g,n}$, there is a product formula, expressing intersections of cycles $[\Gamma, f, \gamma]$ as linear combinations of further log decorated strata classes (\ref{Prop:iota_pushforward_intersection}). 
\end{enumerate}
\begin{definition}
The \emph{log strata algebra} $\mathsf{logS}_{g,n}^\star$
is the $\mathbb{Q}$-vector space
\[
\mathsf{logS}_{g,n}^\star = \bigoplus_{\Gamma} \pPP_\star(\Sigma_{\Gamma}, \Delta_\Gamma) \otimes_{\QQ} \bigotimes_{v \in V(\Gamma)} \mathsf{S}^\star_{g(v),n(v)}
\]
with a product defined by
the product formula \eqref{Prop:iota_pushforward_intersection}.
\end{definition}

We obtain a well-defined  homomorphism of $\mathbb{Q}$-algebras
\[
q : \mathsf{logS}_{g,n}^\star \to \logCH^\star(\Mbar_{g,n})\,.
\]
which is used to define our final (and largest) log tautological ring.
\begin{definition} \label{Def:large_taut_ring}
The {\em large tautological ring} 
$\mathsf{logR}^\star(\Mbar_{g,n})$
is the image
\[
\mathsf{logR}^\star(\Mbar_{g,n}) = q(\mathsf{logS}_{g,n}^\star) \subset
\mathsf{logCH}^\star(\Mbar_{g,n})
\]
of the log decorated strata classes $[\Gamma, f, \gamma]$.
\end{definition}

\subsubsection{Comparisons} \label{Sect:Comparisons}
We first show that the small tautological ring is contained in the large tautological ring.
Let $\Gamma=\Gamma_0$ be the stable graph with a single vertex 
(and no edges) associated to the main stratum of $\Mbar_{g,n}$. Then, for $$f \in \sPP_\star(\Sigma_{\Gamma_0}, \Delta_{\Gamma_0}) = \sPP^\star(\Sigma_{g,n})$$ and $\gamma \in \R^\star(\Mbar_{\Gamma_0})=\R^\star(\Mbar_{g,n})$,  we have
\[
[\Gamma, f, \gamma] = \gamma \cdot \Phi(f)\,.
\]
Since the classes $\gamma \cdot \Phi(f)$ generate the small log tautological ring $\mathsf{logR}^\star_{\mathrm{sm}}(\Mbar_{g,n})$, we obtain inclusions
$$\mathsf{logR}^\star_{\mathrm{pp}}(\Mbar_{g,n}) \subset
\mathsf{logR}^\star_{\mathrm{sm}}(\Mbar_{g,n}) \subset
\mathsf{logR}^\star(\Mbar_{g,n})\, .$$
By the following result, the second inclusion is also, in general, strict.

\begin{proposition} \label{Pro:mugn_not_surjective}
For $g\geq 7$ and $n\geq 0$, 
$\mathsf{logR}^\star_{\mathrm{sm}}(\Mbar_{g,n}) \subsetneq
\mathsf{logR}^\star(\Mbar_{g,n})$.
\end{proposition} 

\begin{proof}
Consider the codimension 2 stratum associated to the graph $\Gamma$ given by
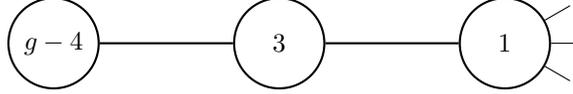
\[
\begin{tikzpicture}[baseline=0pt, vertex/.style={circle,draw,minimum size=1.2cm,thick}]
\node[vertex] (v1) at (0,0) {$g-4$};
\node[vertex] (v2) at (3,0) {$3$};
\node[vertex] (v3) at (6,0) {$1$};

\draw[thick] (v1) -- (v2);
\draw[thick] (v2) -- (v3);

\draw (v3) -- ++(0:1);
\draw (v3) -- ++(30:1);
\draw (v3) -- ++(-30:1);
\end{tikzpicture}
\]
and the class $\gamma = \kappa_1 \otimes 1 \otimes 1 \in \mathsf{R}^\star(\Mbar_\Gamma)$, with $\kappa_1$ on the vertex of genus $g-4$. 
Let $f$ be the homogeneous piecewise polynomial on the star of $\Mbar_{\Gamma}$ given by $\min(\ell_1,\ell_2)$ on $\sigma_{\Gamma} = \RR_{\geq 0}^2$. The minimum function is not linear on this quadrant, as it ``bends'' along the diagonal. {To extend $f$ to a piecewise polynomial on the star, define $f$ on every cone containing the cone $\sigma_{\Gamma}$ by projecting to $\sigma_{\Gamma}$. The extended $f$ automatically vanishes on the boundary of the star since
$f$ vanishes on the boundary of $\sigma_\Gamma$.}.

Let $\widetilde{\mathcal M}_{g,n}\to \Mbar_{g,n}$ be a log blowup associated to a subdivision of $\Sigma_{g,n}$ whose induced subdivision of $\Star_{\sigma_\gamma} \Sigma_{g,n}$ makes $f$ a strict homogeneous piecewise polynomial. On the complement of the codimension $3$ boundary, we are simply blowing up the stratum $\Mcal_{\Gamma} \subseteq \Mbar_{g,n}$. 
We then define
$$\widehat{\gamma} = [\Gamma,f,\gamma] \in \mathsf{CH}^\star(\widetilde{\mathcal M}_{g,n})\,.$$
By definition,  $\widehat{\gamma} \in \mathsf{logR}^\star(\Mbar_{g,n})$. We claim that 
$\widehat{\gamma}$
is \emph{not} contained in $\mathsf{logR}^\star_{\mathrm{sm}}(\Mbar_{g,n})$.

To prove the latter claim, we take the fiber square
\[
\begin{tikzcd}
\widetilde{\mathcal M}_{g,n} \times_{\Mbar_{g,n}} \Mbar_\Gamma \arrow[r,"\pi_1"] \arrow[d, "\pi_2", swap] & \widetilde{\mathcal M}_{g,n} \arrow[d]\\
\Mbar_\Gamma \arrow[r, "\iota_\Gamma"] & \Mbar_{g,n}
\end{tikzcd}\,.
\]
We claim 
\begin{equation}  \label{eqn:Psi_log_f_counterexample}
\Psi^\textup{log}(f) = [\widetilde{\mathcal M}_{g,n} \times_{\Mbar_{g,n}} \Mbar_\Gamma] \in \logCH_\star(\Mbar_{\Gamma})\,
\end{equation}
and that correspondingly
\begin{equation} \label{eqn:widehat_gamma_counterexample}
\widehat{\gamma} = (\pi_1)_\star \pi_2^\star \gamma \in \CH^\star(\widetilde{\mathcal M}_{g,n}) \subseteq \logCH^\star(\Mbar_{g,n}).
\end{equation}
The proof of equation \eqref{eqn:Psi_log_f_counterexample} uses some more machinery, which we have not yet introduced, and is explained in Example \ref{ex:ptwithnr} later.

Restricting to the complement $\mathcal{U}$ of the codimension $3$ boundary, the diagram takes the form
\[
\begin{tikzcd}
E \arrow[r, hook,"\iota_{E}"] \arrow[d, "\pi", swap] & \widehat{\mathcal U} \arrow[d]\\
\cM_{\Gamma} = \cM_{g-4,1} \times \cM_{3,2} \times \cM_{1,n+1} \arrow[r, hook, "\iota_\Gamma"] & \mathcal{U}
\end{tikzcd}
\]
where map $\pi: E \to \cM_{\Gamma}$ is a $\PP^1$-bundle (giving the exceptional divisor of the blowup of $\iota_\Gamma(\cM_\Gamma)$) and the normal bundle of the embedding $\iota_E$ in $\widehat{\mathcal U}$ is $\mathcal{O}_E(-1)$.  The bottom left of the square comes from the standard description of strata associated to dual graphs as products of moduli spaces of curves. The normal bundle computation follows from the explicit description of the blowup. 

By using pull/push compatibility for the cartesian square and the excess intersection formula, we can compute:
\begin{align} \label{eqn:gammahat_actual}
\pi_\star \iota_E^\star \widehat \gamma|_{\widehat{\mathcal U}} &\overset{\eqref{eqn:widehat_gamma_counterexample}}{=} \pi_\star \iota_E^\star \iota_{E\star} \pi^\star (\kappa_1 \otimes 1 \otimes 1) = \pi_\star \left( c_1(\mathcal{O}_E(-1)) \cdot \pi^\star (\kappa_1 \otimes 1 \otimes 1) \right)= - \kappa_1 \otimes 1 \otimes 1 \in \mathsf{R}^1(\cM_\Gamma)\,.
\end{align}
Next, we assume $\widehat \gamma \in \mathsf{logR}^\star_{\mathrm{sm}}(\Mbar_{g,n})$. In other words
\[
\widehat \gamma \in \mathsf{Im}\left(\mathsf{sPP}(\widetilde{\mathcal M}_{g,n}) \otimes \R^\star(\Mbar_{g,n}) \right)\,,
\]
where $\mathsf{sPP}(\widetilde{\mathcal M}_{g,n})$ are the strict piecewise polynomial functions on the cone complex of $\widetilde{\mathcal M}_{g,n}$.\footnote{Here we tacitly use that for any representation of $\widehat \gamma$ using piecewise polynomials on \emph{some} subdivision of the cone complex of $\Mbar_{g,n}$, we can always go to a common refinement and push forward to the complex of $\widetilde{\mathcal M}_{g,n}$ to obtain a strict piecewise polynomial there. The fact that $\widehat \gamma$ was constructed from a representative on $\widetilde{\mathcal M}_{g,n}$ implies that this does not change the class $\widehat \gamma$.} This would imply
\begin{equation} \label{eqn:gammahatsPP}
\pi_\star \iota_E^\star \widehat \gamma|_{\widehat{\mathcal U}} \in \mathsf{Im}\left(\left(\pi_\star \iota_E^\star \mathsf{sPP}(\widetilde{\mathcal M}_{g,n}) \right) \otimes \iota_\Gamma^\star \mathsf{R}^\star(\Mbar_{g,n})\right)\,.
\end{equation}
The classes from strict piecewise polynomials on $\widetilde{\mathcal M}_{g,n}$ are given by fundamental cycles of strata of this space decorated by Chern classes of summands of their normal bundles (see \cite[Theorem 15]{MPS23}). Since the stack $\mathcal{U}$ removes the higher codimension strata of the moduli of curves, the image of the piecewise polynomials simply give the Chern classes of the normal bundle of $\cM_{\Gamma}$:
\[
\pi_\star \iota_E^\star \mathsf{sPP}(\widetilde{\mathcal M}_{g,n}) = \langle (-\psi_1 \otimes 1 - 1 \otimes \psi_1) \otimes 1, 1 \otimes(-\psi_2 \otimes 1 - 1 \otimes \psi_1) \rangle\,.
\]
By the excess intersection formula, this is in fact \emph{contained} in the pullback $\iota_\Gamma^\star \mathsf{R}^\star(\Mbar_{g,n})$, and the additional classes in this pullback are $\iota_\Gamma^\star \kappa_1$ and the $\iota_\Gamma^\star \psi_i$, for $i=1, \ldots, n$. Consider the quotient map $R^1(\cM_\Gamma) \to Q$ by the span of all $\psi$-classes on any of the factors. This tautological group is equal to the full cohomology, and its dimension can be calculated by examining the formulas in~\cite{AC98}. The upshot is that this quotient space $Q$ is $2$-dimensional with basis $[\kappa_1 \otimes 1 \otimes 1]$, $[1 \otimes \kappa_1 \otimes 1]$. By the discussion above, the image of the right-hand side of \eqref{eqn:gammahatsPP} in $Q$ agrees with the image of $\iota_\Gamma^\star \mathsf{R}^\star(\Mbar_{g,n})$ and is spanned by
\[
[\iota_\Gamma^\star \kappa_1] = [\kappa_1 \otimes 1 \otimes 1 + 1 \otimes \kappa_1 \otimes 1]
\]
which does not contain the class \eqref{eqn:gammahat_actual}. This gives the desired contradiction.
\end{proof}

We will take  $\mathsf{logR}^\star(\Mbar_{g,n}) \subset
\mathsf{logCH}^\star(\Mbar_{g,n})$ to be the fundamental definition
of the {\em logarithmic tautological ring} of
the moduli space of Deligne--Mumford
stable curves. In certain
situations, the study of the smaller tautological rings can also
be natural, but our goal is to control all of 
$\mathsf{logR}^\star(\Mbar_{g,n})$.

\subsubsection{Log Gromov-Witten theory}
In the case of standard Gromov-Witten theory, a speculation of \cite{LP09} is that the pushforwards to the moduli of curves of the virtual fundamental classes of
the spaces of genus $g$ stable maps 
to a nonsingular projective
variety $X$ lie in the
tautological ring $\mathsf{RH}^\star(\Mbar_g)$ in cohomology. We can ask
a parallel question here.

\vspace{8pt}
\noindent{\bf Question A.} {\em 
Do the pushforwards to the moduli of curves of the log virtual classes of
the spaces of genus $g$ log stable maps 
to $(X,\Delta)$
lie in $\mathsf{logRH}^\star(\Mbar_g)$?}\smallskip
\vspace{4pt}

For an explanation of the log virtual class and its pushforward to the logarithmic Chow group of $\Mbar_{g,n}$ we refer the reader to \cite[Section 3.2]{RUK22} and the forthcoming survey paper \cite{HMPW_survey}.
The log tautological ring in cohomology can be simply defined as the image
under the cycle map of
$$\mathsf{logR}^\star(\Mbar_g) \rightarrow 
\mathsf{logH}^\star(\Mbar_g)\, .$$
When the target is a toric variety $(Y,\Delta_Y)$ with log structure given by the
full toric boundary, the answer to Question A is positive by the results of \cite{RUK22}. More generally, it should be possible to adapt the methods of~\cite{MR21,R19b,RUK22} and combine them with Janda's results~\cite{Jan17} to give a positive answer to Question A for products of nonsingular curves with logarithmic structure,
but some steps along the path remain to be proven.

\subsection{A presentation of the logarithmic Chow ring in genus 0}
\label{sec:genus0}
\subsubsection{Chow ring of \texorpdfstring{$\Mbar_{0,n}$}{Mbar0n}}
We start with
the standard presentation of the Chow ring
$\mathsf{CH}^\star(\Mbar_{0,n})$ of the space of stable $n$-pointed rational curves. The irreducible components
$\mathsf{D}_{A} \subset \Delta$ of the boundary $\Delta \subset \Mbar_{0,n}$ are
indexed by subsets $A \subset \{ 1,\ldots, n\}$ satisfying
the properties
$$1\in A\, \ \ \ \text{and} \ \ \  2\leq |A|\leq n-2\, .$$
The divisor $\mathsf{D}_{A}$ parameterizes
nodal curves with the markings partitioned by the node
into the
sets $A$ and $A^c$. The Chow classes{\footnote{We use the
same symbol for the divisor and the associated divisor class.}} of the divisors $\mathsf{D}_A$
satisfy two basic sets of relations.

\begin{enumerate}
\item[(i)] $\mathsf{Disjoint\, relations}$: $\mathsf{D}_A\cdot \mathsf{D}_B=0$ when the divisors are set-theoretically disjoint.

\item[(ii)] $\mathsf{WDVV\, relations}$ obtained
from all pullbacks via the forgetting maps
of the boundary relations on $\Mbar_{0,4}$.
\end{enumerate}

\noindent The following fundamental result is due to Keel.
\vspace{8pt}

\begin{theorem}[\cite{Kee92}]
The Chow ring of $\Mbar_{0,n}$ is generated by 
the divisors classes $\{ \mathsf{D}_A\}$, and the ideal of
relations is generated by the ${\mathsf {Disjoint}}$ and ${\mathsf {WDVV}}$ relations:
$$\mathsf{CH}^\star(\Mbar_{0,n}) =
\frac{\mathbb{Q}[\{ \mathsf{D}_A\}]}{(\mathsf{Disjoint}\, ,\, \mathsf{WDVV})}\, .$$
\end{theorem}


\vspace{0pt}

 Since $\mathsf{CH}^\star(\Mbar_{0,n})$
is generated by the classes of the boundary divisors, 
$\mathsf{R}^\star(\Mbar_{0,n})= \mathsf{CH}^\star(\Mbar_{0,n})$. 
A parallel log result holds:
the three  types of logarithmic tautological
rings in genus 0 are all equal to the full logarithmic Chow ring.

\begin{proposition}We have 
$\mathsf{logR}^\star_{\mathrm{pp}}(\Mbar_{0,n})  =
\mathsf{logR}^\star_{\mathrm{sm}}(\Mbar_{0,n}) =
\mathsf{logR}^\star(\Mbar_{0,n}) =
\mathsf{logCH}^\star(\Mbar_{0,n})$. \label{alleq}
\end{proposition}

\begin{proof}
For any stable graph $\Gamma$ associated to a stratum $\Mbar_\Gamma$ of $\Mbar_{0,n}$ we have that $\R^\star(\Mbar_\Gamma) = \CH^\star(\Mbar_\Gamma)$.
Then the equality $\mathsf{logR}^\star(\Mbar_{0,n}) = \mathsf{logCH}^\star(\Mbar_{0,n})$ follows from Corollary \ref{Cor:logCH_generators}. We conclude by showing that any generator $[\Gamma, f, \alpha]$ of $\mathsf{logR}^\star(\Mbar_{0,n})$ lies in $\mathsf{logR}^\star_{\mathrm{pp}}(\Mbar_{0,n})$, where $f \in \pPP_\star(\Mbar_\Gamma^\mathsf{str})$ and $\alpha \in \CH^\star(\Mbar_\Gamma)$. For this just observe that there exists $g \in \sPP^\star(\Mbar_\Gamma^\mathsf{str})$ such that $\alpha = \Phi(g)$ (where $\Phi$ is the pullback map $\sPP^\star(X) \to \CH^\star(X)$, again since the strata of $\Mbar_\Gamma$ generate its Chow ring. But then
\[
[\Gamma, f, \alpha] = [\Gamma, f, \Phi(g)] = [\Gamma, f\cdot g, 1] = (\iota_\Gamma)_\star \Psi(f\cdot g) = \Phi( (\iota_\Gamma)^\textup{trop}_\star f \cdot g) \in \mathsf{logR}^\star_{\mathrm{pp}}(\Mbar_{0,n})\,. \qedhere
\]
\end{proof}



\subsubsection{Calculation of \texorpdfstring{$\mathsf{logCH}^\star(\Mbar_{0,n})$}{logCH*(Mbar0n)}}

Each divisor class $\mathsf{D}_A\in \mathsf{CH}^1(\Mbar_{0.n})$
corresponds canonically to a piecewise polynomial function on the Artin fan of 
$(\Mbar_{0,n},\Delta)$. Therefore, the $\mathsf{WDVV}$ relations can be canonically
lifted from ${\mathbb{Q}[\{ \mathsf{D}_A\}]}$ to the algebra of piecewise polynomials 
$\mathsf{PP}^\star(\Mbar_{0,n}, \Delta)$.
Our first result is the following.

\begin{theorem} \label{Thm:logCHMbar0n}
The logarithmic Chow ring of $\Mbar_{0,n}$ is given by
\[
\mathsf{logCH}^\star(\Mbar_{0,n}) = \mathsf{PP}^\star(\Mbar_{0,n}, \Delta)/\mathsf{WDVV}.
\]
In particular, all log Chow classes  are tautological.
\end{theorem}

The calculation shows that the logarithmic Chow ring of
$\Mbar_{0,n}$ is not only tractable, but has a structure which is as
simple as possible. Assuming this result, we can prove Theorem \ref{Thm:logCHMbar0n-F} from the introduction.

\begin{proof}[Proof of Theorem \ref{Thm:logCHMbar0n-F}]
Our first observation is that, in the notation of Theorem $\mathsf{CH}^\star(\Mbar_{0,n})$ above, the quotient ring
\[
\sPP^\star(\Sigma_{0,n}) = \frac{\mathbb{Q}[\{ \mathsf{D}_A\}]}{(\mathsf{Disjoint})}
\]
gives the Stanley-Reisner presentation \cite[Definition 1.6]{MR2110098} of the strict piecewise polynomials on $\Sigma_{0,n}$ (where the $\mathsf{Disjoint}$ relations are exactly the generators of the face ideal). In particular, Keel's theorem immediately implies
\[
\R^\star(\Mbar_{0,n}) = \CH^\star(\Mbar_{0,n}) = \frac{\sPP^\star(\Sigma_{0,n})}{(\mathsf{WDVV})}\,.
\]
From this presentation, we find the desired isomorphism
\[
\mathsf{R}^\star(\Mbar_{0,n}) \otimes_{\mathsf{sPP}^\star(\Sigma_{0,n})}
\mathsf{PP}^\star(\Sigma_{0,n}) \cong \frac{\sPP^\star(\Sigma_{0,n})}{(\mathsf{WDVV})} \otimes_{\mathsf{sPP}^\star(\Sigma_{0,n})}
\mathsf{PP}^\star(\Sigma_{0,n}) \cong \frac{\pPP^\star(\Sigma_{0,n})}{(\mathsf{WDVV})}\,,
\]
which is equal to $\mathsf{logCH}^\star(\Mbar_{0,n})$ by Theorem \ref{Thm:logCHMbar0n}, finishing the proof.
\end{proof}

\subsubsection{Toric geometry}
Our strategy to prove
Theorem \ref{Thm:logCHMbar0n}
is to move 
the problem from the Chow rings of blowups of $\Mbar_{0,n}$ to the Chow ring of a certain toric variety. We can then take advantage of the fact that the limit Chow rings are known in the toric case. 
The following Proposition
will play a central role.


\begin{proposition}\label{prop: toric-embedding}
There exists a nonsingular quasi-projective toric variety $X_{0,n}$ with dense torus $\mathsf{T}\subset X_{0,n}$
and an embedding $$j: \Mbar_{0,n}\hookrightarrow X_{0,n}$$ which
satisfies the following properties:
\begin{enumerate}[(i)]
\item 
The stack quotient
$[X_{0,n}/\mathsf{T}]$ is canonically identified with the Artin fan of $\Mbar_{0,n}$,
and 
the composition
\[
\Mbar_{0,n}\hookrightarrow X_{0,n}\to [X_{0,n}/\mathsf{T}]
\]
coincides with natural map  
$$\Mbar_{0,n}\rightarrow \mathsf{A}(\Mbar_{0,n},\Delta)$$ to the Artin fan. 
 In particular, the stratification of $X_{0,n}$ by torus orbits pulls back to the stratification of $\Mbar_{0,n}$ by topological type. 
\item If $V$ is a torus orbit closure in $X_{0,n}$ and $W$ is the corresponding stratum of $\Mbar_{0,n}$ obtained by intersection with $V$, then there is an identification of vector bundles bundles on $W$
\[
N_{W/\Mbar_{0,n}} = j|_W^\star N_{V/X_{0,n}},
\]
between the normal bundle of $W$ in $\Mbar_{0,n}$ and the pullback of the normal bundle of $V$ in the ambient toric variety $X_{0,n}$. 
\end{enumerate}
\end{proposition}

Proposition \ref{prop: toric-embedding} is well-known to experts, but since the proof is spread out over many papers in the literature, we recall the appropriate results and explain how to deduce the claims. We start with a theorem of Kapranov \cite[Section 4.1]{Kap93}. 

\begin{theorem}
The moduli space $\Mbar_{0,n}$ is the Chow quotient of the Grassmannian $G(2,n)$ by the action of the $(n-1)$-dimensional dilating torus $\mathsf{H}$.
\end{theorem}



The Grassmannian $G(2,n)$  embeds in $\mathbb P^{\binom{n}{2}-1}$
via the Pl\"ucker map, and the dilating torus $\mathsf{H}$ is a subtorus of the dense torus of the Pl\"ucker projective space $\mathbb P^{\binom{n}{2}-1}$.
We therefore have the following result.

\begin{corollary}
    Let $X'_{0,n}$ be the Chow quotient of $\mathbb P^{\binom{n}{2}-1}$ by the torus $H$. There is a natural embedding:
    \begin{equation}\label{vv44}
    \Mbar_{0,n}\hookrightarrow X'_{0,n}\, .
    \end{equation}
\end{corollary}

The Chow quotient $X'_{0,n}$ is a toric variety, since it is the Chow quotient of a toric variety by a subtorus of its dense torus -- see~\cite{KSZ91} for additional details on toric quotients.
The properties of the embedding \eqref{vv44} have been well-studied, by Gibney--Maclagan \cite{GM07}, Hacking--Keel--Tevelev \cite{HKT}, and Tevelev~\cite{Tev07}. We collect the results that we need here. 

\begin{proposition}
    Let $\Sigma'_{0,n}$ be the fan of the toric variety $X'_{0,n}$, and let 
    $\mathsf{T}$ be its dense torus. The tropicalization of $\cM_{0,n}$ in its embedding into $\mathsf{T}$ is a union of cones in $\Sigma'_{0,n}$. 
\end{proposition}

\begin{proof}
    See~\cite[Theorem~5.7]{GM07} of Gibney and Maclagan and use the geometric interpretation of tropicalization~\cite[Section~2]{HKT}. 
\end{proof}

Let $\Sigma_{0,n}$ be the subfan given by the union of cones in $\Sigma'_{0,n}$ which meet the tropicalization of $\cM_{0,n}$, and let $X_{0,n}$ be the associated non-compact torus invariant open in $X'_{0,n}$. Equivalently, $X_{0,n}\subset X'_{0,n}$ is the complement of the closed strata that are disjoint from $\Mbar_{0,n}$. 

\begin{proposition} \label{prop:smooth_to_artin_fan}
    The morphism
    \[
    \Mbar_{0,n}\to [X_{0,n}/T]
    \]
    is smooth. 
\end{proposition}

\begin{proof}
    See~\cite[Theorem~1.11]{HKT} of Hacking, Keel, and Tevelev.
\end{proof}

\begin{proof}[Proof of Proposition~\ref{prop: toric-embedding}.]
  We have constructed the toric variety $X_{0,n}$ above, as a torus invariant open inside the Chow quotient of $\PP^{\binom{n}{2}-1}$. It follows from Kapranov's description that the toric stratification of $X_{0,n}$ pulls back to the usual stratification of $\Mbar_{0,n}$.  The statement about Artin fans follows.

  The statement about normal bundles follows by using these results together with the following pair of Cartesian diagrams:
\[
\begin{tikzcd}
W\arrow{d}\arrow{r} & \Mbar_{0,n}\arrow{d}\\
V\arrow{d}\arrow{r} & X_{0,n} \arrow{d}\\
{[V/T]} \arrow{r}&\mathsf{A}(X_{0,n}).
\end{tikzcd}
\]
The bottom left of the diagram is the closed stratum in the Artin fan $\mathsf A(V)$ given by the done dual to $V$ in $X_{0,n}$. Since the composite maps from the top row to the bottom are both smooth by Proposition \ref{prop:smooth_to_artin_fan}, and therefore flat, the statement about normal bundles now follows from flat base change for the normal bundle.
\end{proof}

We next turn to the Chow description. The following result is due to de Concini--Procesi~\cite{dCP95}, and is proved in the more general context of wonderful compactifications of hyperplane arrangement complements. See also~\cite{FY04}. 

\begin{proposition}
Let $j: \Mbar_{0,n}\to X_{0,n}$ be the inclusion above. The pullback map
\[
j^\star: \mathsf{CH}^\star(X_{0,n})\to \mathsf{CH}^\star(\Mbar_{0,n})
\]
is an isomorphism.
The same is true for the {pullback map under the} inclusion of a stratum of $\Mbar_{0,n}$ into the corresponding stratum of $X_{0,n}$. 
\end{proposition}

\begin{proof}
It is straightforward to see that the pullback is surjective: Keel's presentation of the Chow ring already shows that the Chow ring of $\Mbar_{0,n}$ is generated as an algebra by boundary divisors. Moreover, the stratification of $\Mbar_{0,n}$ by topological type is the pullback of the toric stratification on $X_{0,n}$. In particular, the boundary divisors pull back to the boundary divisors, which guarantees surjectivity of $j^\star$. 

The injectivity is slightly more subtle, but follows from more general results on wonderful compactifications of hyperplane arrangement complements. Indeed, $\Mbar_{0,n}$ is the wonderful compactification of the braid arrangement complement in $\mathbb C^{n-3}$, in the sense of de Concini--Procesi~\cite[Section~4.3]{dCP95}. The cohomology presentation they give in~\cite[Section~5]{dCP95} is the same as that of Keel's. In this broader context of arrangements, in~\cite{FY04} Feichtner and Yuzvinsky showed that this explicit presentation is precisely the natural Chow presentation for smooth toric varieties\footnote{In terms of invariant divisors with relations given by characters, as in~\cite[Chapter~5]{Ful93}} of the toric variety $X_{0,n}$. The reader can find a summary of this work in~\cite[Section~6.7]{MS14}. These together show that $j^\star$ is an isomorphism. 

We turn to the statement for strata. Fix a boundary stratum in $\Mbar_{0,n}$ with marked dual graph $\Gamma$. This stratum $\Mbar_\Gamma$ is identified with a product of moduli spaces of curves associated to the vertices of $\Gamma$, marked by the flags of incident edges and markings. The stratum $\Mbar_\Gamma$ is naturally embedded in a stratum $X_\Gamma$ of the toric variety $X_{0,n}$. The fan $\Sigma_{0,n}$ of this toric variety is naturally identified with the cone complex $\cM_{0,n}^{\mathsf{trop}}$. The fan of $X_\Gamma$ is equal to the star fan, in $\cM_{0,n}^{\mathsf{trop}}$, of the cone labeled by the type $\Gamma$. This star fan is also naturally a product over vertices in $\Gamma$, of the fans associated to vertices of $\Gamma$, marked as above. 

Summarizing, the stratum $\Mbar_\Gamma$ is a product of $\Mbar_{0,k}$ for various $k<n$ and similarly, the stratum $X_\Gamma$ is a product of $X_{0,k}$ for various $k<n$. The induced embedding $\Mbar_\Gamma\hookrightarrow X_\Gamma$ is compatible with the product decomposition. 

The varieties in question are linear and smooth, and therefore satisfy a K\"unneth theorem in Chow cohomology~\cite{Tot14}. It follows that the embedding $\Mbar_\Gamma\hookrightarrow X_\Gamma$ also induces an isomorphism in Chow under pullback. 
\end{proof}

\subsubsection{Proof of Theorem \ref{Thm:logCHMbar0n}} We will prove the theorem by showing that the directed systems of Chow rings of blowups coincide, and this will be done by induction on the number of blowups. We can then use the easy isomorphism on the $X_{0,n}$ side. Let us spell out the argument in the case of a single blowup, before explaining the general case. We start with an isomorphism on Chow induced by the inclusion
\[
j:\Mbar_{0,n}\to X_{0,n}. 
\]
Let $W$ be a stratum of $\Mbar_{0,n}$. By the proposition above, $W$ is equal to $j^{-1}(V)$. Furthermore, under the pullback $j^\star$, the normal bundle of $V$ becomes that of $W$. By using Keel's blowup formula~\cite{Kee92}, we see that pullback under
\[
j':\mathsf{Bl}_W\Mbar_{0,n}\hookrightarrow\mathsf{Bl}_VX_{0,n}
\]
gives rise to a natural isomorphism of Chow rings. 

The claimed result now follows by induction. Any stratum of $\mathsf{Bl}_W\Mbar_{0,n}$ is either a blowup of a moduli space with smaller numerical data, or formed from projective bundles over such a smaller moduli space. By the arguments above, the map $j'$ still identifies these strata and their Chow rings by pullback, compatibly with normal bundles. It follows that the two directed systems of Chow rings are isomorphic. 

The stated theorem now follows from the fact that the Chow ring of the toric side is given by piecewise polynomials modulo linear relations coming from the characters of the dense torus. Since the latter is preserved by blowups, the result follows. \qed

\subsubsection{Spaces of rational curves}

The calculation of $\mathsf{logCH}^\star(\Mbar_{0,n})$ has a natural
extension to the logarithmic geometry the moduli of maps.
Let $Y$ be a nonsingular and projective toric variety, with its canonical log structure coming from the full toric boundary. Consider the moduli space $\Mbar_\Lambda(Y)$ of logarithmic stable maps of genus 0 curves to $Y$ (with respect to the full toric boundary
$\Delta_Y\subset Y$)
with fixed numerical data $\Lambda$, and assume $n \geq 3$. By results of \cite{R15b,RW19}, the stack
 $\Mbar_\Lambda(Y)$ 
is isomorphic to a logarithmic blowup of $\Mbar_{0,n}\times Y$. As
a consequence, 
$$\mathsf{logCH}^\star(\Mbar_\Lambda(Y)) = \mathsf{logCH}^\star(\Mbar_{0,n}\times Y)\, .$$
The arguments we have used for $\Mbar_{0,n}$ immediately generalize to prove
the following result.

\begin{theorem}
The logarithmic Chow ring of $\Mbar_\Lambda(Y)$ is given by
\[
\mathsf{logCH}^\star(\Mbar_\Lambda(Y)) = \mathsf{PP}^\star(\Mbar_{0,n}\times Y)\, / \, (\mathsf{WDVV,Linear})\, ,
\]
where ${\mathsf{Linear}}$ is the usual space of relations on piecewise
linear functions $\mathsf{PL}^\star(Y)$ obtained from
the divisor linear equivalences of the components of the toric boundary of $Y$. 
\end{theorem}

A parallel calculation without the {\em full toric boundary} condition would
be interesting.
For example, the logarithmic Chow rings of the moduli
spaces of genus 0 logarithmic maps to $\mathbb P^n$ with the
logarithmic structure coming from a subset of the toric boundary
are well-behaved, and 
there is evidence that their Chow rings are entirely tautological (as in the case of $\Mbar_{0,n}$), see~\cite{KHNSZ,Opr06,Pan99}.

\subsection{Results and counterexamples in genus \texorpdfstring{$1$}{1}}
\label{sec:mapsspaces}
For all $g$ and $n$, there is a canonical map
$$\mu_{g,n}:\mathsf{PP}^\star(\Mbar_{g,n}, \Delta) \, \rightarrow\, 
\mathsf{logR}^\star(\Mbar_{g,n})\, .$$
The map $\mu_{0,n}$ is surjective in genus 0, and 
the calculation of 
$\mathsf{logCH}^\star(\Mbar_{0,n})$ can be restated as:
{\em the kernel of $\mu_{0,n}$ is generated by
the canonical lifts of the $\mathsf{WDVV}$ relations}.
A more elegant restatement is as 
an isomorphism:
$$\mu_{0,n}^{\mathsf{R}}:\mathsf{R}^\star(\Mbar_{0,n}) \otimes_{\mathsf{sPP}^\star(\Mbar_{0,n}, \Delta)}
\mathsf{PP}^\star(\Mbar_{0,n}, \Delta)
\ \stackrel{\sim}{\longrightarrow} \
\mathsf{logR}^\star(\Mbar_{0,n})\, .$$
Here, ${\mathsf{sPP}^\star(\Mbar_{0,n})}$
is the algebra of {strict piecewise polynomials}
on the Artin fan of $\Mbar_{0,n}$.


In higher genus, the map $\mu^{\mathsf{R}}_{g,n}$ surjects onto 
$\mathsf{logR}_{\mathrm{sm}}^\star(\Mbar_{g,n})$, 
$$\mu_{g,n}^{\mathsf{R}}:\mathsf{R}^\star(\Mbar_{g,n}) \otimes_{\mathsf{sPP}^\star(\Mbar_{g,n}, \Delta)}
\mathsf{PP}^\star(\Mbar_{g,n}, \Delta)
\ {\twoheadrightarrow} \
\mathsf{logR}_{\mathrm{sm}}^\star(\Mbar_{g,n})\, .$$

\noindent {\bf Question B.} {\em 
Can the 
kernel of $\mu_{g,n}^{\mathsf{R}}$
be understood?}

\vspace{8pt}

A non-trivial kernel of $\mu_{g,n}^{\mathsf{R}}$ can be found
even in genus 1.
\begin{proposition} \label{Prop:mu_1_n_not_injective}
For $g=1, n\geq 3$ the map $\mu_{g,n}^{\mathsf{R}}$ is not injective.
\end{proposition}
\begin{proof}
Below we construct a non-trivial element in the kernel of $\mu_{1,3}^{\mathsf{R}}$. For $n>3$ this element pulls back to a non-zero element in the kernel of $\mu_{1,n}^{\mathsf{R}}$ under the natural forgetful map.

For $n=3$, consider the stable graph $\Gamma_0$ with two genus $0$ vertices, carrying markings $\{1,2\}$ and $\{3\}$, respectively, and connected by two edges. In Figure \ref{fig:star_Mbar_13} we depict the star of $\Gamma_0$ in the tropicalization $\Sigma_{1,3}$ of $\Mbar_{1,3}$. The cone of $\Gamma_0$ corresponds to the vertical line, and the central dot corresponds to the ray $\tau_{\Gamma_0}$ where $\ell_1=\ell_2$. 
The drawn subdivision $\widehat \Gamma$ at $\tau_{\Gamma_0}$  corresponds to the blowup $\widehat{\mathcal{M}} \to \Mbar_{1,3}$ of the stratum associated to $\Gamma_0$. 
We claim that
\begin{equation} \label{eqn:nonzero_element_in_kernel}
0 \neq 1 \otimes ( (x-y) \cdot \min(\ell_1, \ell_2)) \in \ker \mu_{g,n}^{\mathsf{R}}
\end{equation}
is a nonzero element of the kernel.

To see that it maps to zero under $\mu_{g,n}^{\mathsf{R}}$ we note that the piecewise linear function $\min(\ell_1, \ell_2)$ has value $1$ on the generator of the new ray $\tau_{\Gamma_0}$ and $0$ on all other rays of $\widehat \Sigma$. Thus it corresponds to the exceptional divisor $E$ of the blowup $\widehat{\mathcal{M}} \to \Mbar_{1,3}$. Multiplying this function by $x-y$ corresponds to pulling back the WDVV relation under the projection $E \to \Mbar_{\Gamma_0} = \Mbar_{0,4}$, since this relation is given by the piecewise polynomial $x=y$ on $\Sigma_{0,4}$. This shows that $ (x-y) \cdot \min(\ell_1, \ell_2)$ indeed maps to zero in $\mathsf{logR}_{\mathrm{sm}}^\star(\Mbar_{1,3})$.

Finally, to see that \eqref{eqn:nonzero_element_in_kernel} is nonzero, note that the map $\mathsf{sPP}^\star(\Mbar_{1,3}, \Delta) \to \mathsf{R}^\star(\Mbar_{1,3})$ is surjective (since the boundary strata generate the tautological ring in genus $1$) and it is an isomorphism in degree at most $1$ (since the boundary divisors in $\Mbar_{1,3}$ are linearly independent by \cite{AC98}). Thus 
\[
\mathsf{R}^\star(\Mbar_{1,3}) = \mathsf{sPP}^\star(\Mbar_{1,3}, \Delta) / I
\]
with the ideal $I$ generated in the degree at least $2$. It follows that the domain of $\mu_{g,n}^{\mathsf{R}}$ is isomorphic to
\[
\mathsf{PP}^\star(\Mbar_{g,n}, \Delta) / I \cdot \mathsf{PP}^\star(\Mbar_{g,n}, \Delta)\,.
\]
Since the generators of $I$ have degree at least $2$, the degree $2$ part of $I \cdot \mathsf{PP}^\star(\Mbar_{g,n}, \Delta)$ agrees with the degree $2$ part of $I$, and thus consists of strict piecewise polynomial functions on $\Sigma_{1,3}$. Since $(x-y) \cdot \min(\ell_1, \ell_2)$ is not strict piecewise polynomial on $\Sigma_{1,3}$ (only on its subdivision $\widehat \Sigma$), it is not contained in $I \cdot \mathsf{PP}^\star(\Mbar_{g,n}, \Delta)$, and thus nonzero.
\end{proof}

\begin{figure}
    \centering
\begin{tikzpicture}
    \draw (-2,0) --node[above,sloped] {\( \ell_1=0 \)} (0,2) --node[above,sloped] {\( \ell_1=0 \)} (2,0) --node[below,sloped] {\( \ell_2=0 \)} (0,-2) --node[below,sloped] {\( \ell_2=0 \)} (-2,0);
    \draw (0,-2) -- (0,2);
    \draw[red] (-2,0) -- (2,0);
    \draw[red, fill] (0,0)  circle (2pt) node[below right] {$\tau_{\Gamma_0}$};

    \fill (0,-3) circle (0.06);
    \draw[->] (0,-3) --node[above]{y}  (1,-3) ;
    \draw[->] (0,-3) --node[above]{x}  (-1,-3) ;

\begin{scope}[shift={(4,-2)}, scale=0.7]
    
    \coordinate (A) at (0,0);
    \coordinate (B) at (-1,2);
    \coordinate (C) at (1,2);
    
    \coordinate (leg1) at ([xshift=-0.5cm]B);
    \coordinate (leg2) at ([xshift=0.5cm]C);
    \coordinate (leg3) at ([yshift=-0.5cm]A);
    
    \draw (A) -- (B) -- (C) -- cycle;
    
    \draw (B) -- (leg1) node[left] {1};
    \draw (C) -- (leg2) node[right] {2};
    \draw (A) -- (leg3) node[below] {3};
    
    \path (A) -- (B) node[midway,left] {\( \ell_1 \)};
    \path (A) -- (C) node[midway,right] {\( \ell_2 \)};
    \path (B) -- (C) node[midway,above] {y};
    
    \fill (A) circle (0.1);
    \fill (B) circle (0.1);
    \fill (C) circle (0.1);
    
\end{scope}

\begin{scope}[shift={(-4,-2)}, scale=0.7]
    \coordinate (A) at (0,0);
    \coordinate (B) at (0,2);
    \coordinate (C) at (0,3);

    \draw (A) to[bend left] node[left] {\( \ell_1 \)} (B);
    \draw (A) to[bend right] node[right] {\( \ell_2 \)} (B);
    \draw (B) -- node[left] {x} (C);

    \coordinate (leg1) at ([xshift=-0.5cm, yshift=0.5cm]C);
    \coordinate (leg2) at ([xshift=0.5cm, yshift=0.5cm]C);
    \coordinate (leg3) at ([yshift=-0.5cm]A);

    \draw (C) -- (leg1) node[left] {1};
    \draw (C) -- (leg2) node[right] {2};
    \draw (A) -- (leg3) node[right] {3};

    \fill (A) circle (0.1);
    \fill (B) circle (0.1);
    \fill (C) circle (0.1);
\end{scope}

\end{tikzpicture}
    \caption{A cross section through the star of $\Gamma_0$ in the cone stack of $\Mbar_{1,3}$. For better visibility, we draw the double-cover of the actual picture where the two edges with lengths $\ell_1, \ell_2$ are distinguishable. The star of $\Gamma_0$ is the quotient of the figure under reflection along the horizontal axis (the red subdivision defined
    by $\ell_1 = \ell_2$).}
    \label{fig:star_Mbar_13}
\end{figure}
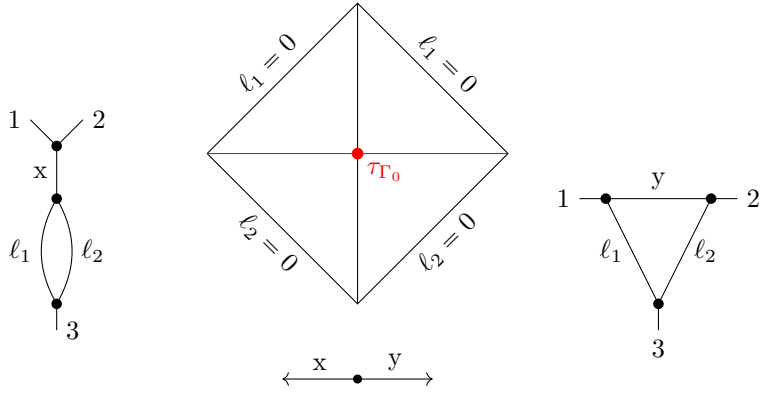

\begin{proposition} \label{Pro:logR_pp_genus_1}
In genus 1, 
$\mathsf{logR}^\star_{\mathrm{pp}}(\Mbar_{1,n})
=\mathsf{logR}^\star(\Mbar_{1,n})$ 
for all $n\geq 1$.
\end{proposition}
\begin{proof}
By Definition \ref{Def:large_taut_ring} it suffices to show that any decorated log stratum $[\Gamma, f, \alpha]$ is contained in $\mathsf{logR}^\star_{\mathrm{pp}}(\Mbar_{1,n})$. For this, we note that $\alpha \in \mathsf{R}^\star(\Mbar_\Gamma)$ is a product of strata classes (decorated by $\kappa$ and $\psi$-classes) on the factors $\Mbar_{g(v), n(v)}$ of $\Mbar_\Gamma$, where clearly all $g(v) \leq 1$. 
But in genus at most $1$, any $\psi$-class (and thus after pushforward any $\kappa$-class) can be expressed in terms of undecorated boundary strata using the divisorial relations from \cite{AC98}. Thus the map $\Phi: \mathsf{sPP}^\star(\Mbar_{g',n'}) \to \mathsf{R}^\star(\Mbar_{g', n'})$ is surjective for all $(g',n')$ with $g'=0,1$.  This means there exists $h_v \in \mathsf{sPP}^\star(\Sigma_{g(v),n(v)})$ such that
\[
\alpha = \prod_{v \in V(\Gamma)} \Phi(h_v) \in \mathsf{R}^\star(\Mbar_\Gamma)\,.
\]
Let $\pi_v : \Sigma_\Gamma \to \Sigma_{g(v), n(v)}$ be the projection to the factor associated to vertex $v$ in \eqref{eqn:Sigma_Gamma}. Then it follows that
\[
[\Gamma, f, \alpha] = [\Gamma, f \cdot \prod_{v \in V(\Gamma)} h_v \circ \pi_v, 1]\,.
\]
Finally we conclude by noting that for any class $[\Gamma,g,1]$ we have
\[
[\Gamma,g,1] = (\iota_\Gamma)_\star \Psi(g) = \Phi((\iota_\Gamma)_\star^\trop g) \in \mathsf{logR}_{\mathrm{pp}}^\star(\Mbar_{1,n}).\qedhere
\]
\end{proof}

\vspace{8pt}
\noindent{\bf Question C.} {\em Find a presentation of 
$\mathsf{logR}^\star(\Mbar_{1,n})$}.

\subsection{Study in higher genus}
\label{sec:highergenus}
While $\mathsf{logR}^\star(\Mbar_{g,n})$ appears larger
and more complicated than the standard tautological ring 
$\mathsf{R}^\star(\Mbar_{g,n})$, we view the study as {\em not
being essentially more difficult}. The calculation of 
$\mathsf{logR}^\star(\Mbar_{0,n})$  is the first evidence
of the tractability of these log Chow rings. We show here how
relations in $\mathsf{R}^\star(\Mbar_{g,n})$ can be used
to determine the structure of $\mathsf{logR}^\star(\Mbar_{g,n})$.

The tautological ring $\mathsf{R}^\star(\Mbar_{g,n})$
admits a surjection from the strata algebra 
$$ \phi_{g,n}:\mathsf{S}^\star_{g,n} \twoheadrightarrow
\mathsf{R}^\star(\Mbar_{g,n})\, ,$$
see \cite[Appendix A]{GP03}.
A full description of the tautological rings of
the moduli spaces of curves is provided by 
presenting a complete set of additive generators of the kernel,
$$\mathcal{P}_{g,n} \subset \text{ker}(\phi_{g,n})\, ,$$
for all stable $g$ and $n$. 
Pixton has conjectured a complete set of generators of  
$\mathcal{P}_{g,n}$
\cite{Pixton_relations}.

Given the domain $\Mbar_\Gamma$ of a gluing map $\iota_\Gamma$, there
is a similar surjection
$$ \phi_{\Gamma}:\mathsf{S}^\star_{\Gamma} = \bigotimes_{v \in V(\Gamma)} \mathsf{S}^\star_{g(v),n(v)} \twoheadrightarrow
\mathsf{R}^\star(\Mbar_{\Gamma}) \subseteq \mathsf{CH}^\star(\Mbar_\Gamma)\, ,$$
whose image agrees with the image of the natural composition
\begin{equation} \label{eqn:Chow_tensor_composition}
    \bigotimes_{v \in V(\Gamma)} \mathsf{R}^\star(\Mbar_{g(v),n(v)}) \hookrightarrow \bigotimes_{v \in V(\Gamma)} \mathsf{CH}^\star(\Mbar_{g(v),n(v)}) \xrightarrow{\mathsf{pr}} \mathsf{CH}^\star(\Mbar_\Gamma)\,.
\end{equation}
It is natural to expect that the kernel $\mathcal{P}_\Gamma$ of $\phi_\Gamma$ is the ideal generated by the tautological relations $\mathcal{P}_{g(v),n(v)}$ on the individual factors $\mathsf{S}^\star_{g(v),n(v)}$ and this would follow if the composition \eqref{eqn:Chow_tensor_composition} is injective. However, in contrast to the case of singular cohomology, the Chow group of a product does not agree with the tensor product of its Chow groups. 

Thus, in order to fully control the Chow groups of all normalizations of strata closures in $\Mbar_{g,n}$ we a priori need the full system of tautological relations $\mathcal{P}_\Gamma$. However, there is a sufficient condition to ensure that the $\mathcal{P}_\Gamma$ are indeed generated by the $\mathcal{P}_{g(v),n(v)}$: assume that for all pairs $g(v), n(v)$, the system $\mathcal{P}_{g(v),n(v)}$ also gives the complete system of tautological relations in cohomology. Then we have $\mathsf{R}^\star(\Mbar_{g(v),n(v)}) \cong \mathsf{RH}^{2\star}(\Mbar_{g(v),n(v)})$ and a commutative diagram
\begin{equation}
\begin{tikzcd}
\bigotimes_{v \in V(\Gamma)} \mathsf{R}^\star(\Mbar_{g(v),n(v)}) \arrow[d, "\cong"]  \arrow[rr] &  & \mathsf{CH}^\star(\Mbar_\Gamma) \arrow[d, "\mathsf{cl}"]\\
\bigotimes_{v \in V(\Gamma)} \mathsf{RH}^{2\star}(\Mbar_{g(v),n(v)}) \arrow[r, hookrightarrow] & \bigotimes_{v \in V(\Gamma)} \mathsf{H}^{2\star}(\Mbar_{g(v),n(v)}) \arrow[r,"\cong"] & \mathsf{H}^{2\star}(\Mbar_\Gamma)
\end{tikzcd}
\end{equation}
The injectivity and isomorphisms along the lower left path of the diagram (where the last isomorphism follows from the K\"unneth formula in cohomology) imply that the upper arrow is injective. As mentioned before, this then implies that $\mathcal{P}_\Gamma$ is the ideal generated by the $\mathcal{P}_{g(v),n(v)}$.

\begin{theorem}  \label{Thm:logR_det_by_Pixton}
The log tautological ring $\mathsf{logR}^\star(\Mbar_{g,n})$
is determined by the set of relations
\[
\mathcal{P}^{\leq}_{g,n}\ =\  \Big\{\,  \mathcal{P}_{\Gamma} \subset \mathsf{S}^\star_\Gamma \ \Big|\  \Gamma \text{ stable graph of genus $g$ with $n$ legs}\,  \Big\}\, .
\]
\end{theorem}



\begin{proof}
By assumption, the above system of relations determines the tautological rings $\mathsf{R}^\star(\Mbar_{\Gamma}) = \mathsf{S}^\star_\Gamma / \mathcal{P}_{\Gamma}$ of all spaces $\Mbar_\Gamma$. But these are precisely the tautological rings of the monodromy torsors $\Mbar_{\Gamma} = P_{\sigma_\Gamma}$ with $\sigma_\Gamma \in \Sigma_{g,n}$ appearing in the standard tautological system $\mathsf{R}_{\Mbar_{g,n}}$ on $\Mbar_{g,n}$. Thus by Theorem \ref{Thm:R_blow_up_determined} this information uniquely determines the tautological system $\pi^\star \mathsf{R}_{\Mbar_{g,n}}$ on any iterated boundary blowup $\pi: \widehat{\mathcal{M}} \to \Mbar_{g,n}$ via Fulton's blowup exact sequence. Since the iterated boundary blowups are cofinal in the system of all log blowups, we obtain a description of the log tautological ring $\mathsf{logR}^\star(\Mbar_{g,n})$.
\end{proof}

While ${\mathsf{logR}}^\star(\Mbar_{g,n})$
is determined by set of 
tautological relations  $\mathcal{P}^{\leq}_{g,n}$, the study via Fulton's blowup sequence is not 
practical. A more useful direction
would be to lift the relations known
among the tautological classes of
$\Mbar_{g,n}$.

\vspace{8pt}
\noindent{\bf Question D.} {\em Are there canonical lifts of Pixton's relations to
$\mathsf{logR}^\star(\Mbar_{g,n})$ ?}

\vspace{8pt}
In the case of Pixton's double ramification cycle relations,
lifts (depending upon a choice of 
stability condition for line bundles)
have been given in \cite{HMPPS} via the
formula for the logarithmic double
ramification cycle.

Another potential source of relations comes from log double ramification cycles of higher rank. The rank $r$ logarithmic double ramification cycle is a virtually log smooth compactification of the space of pointed smooth curves equipped with $r$ principal divisors, each with prescribed zeroes and poles at the marked points~\cite{HS22,MR21}. The space is equipped with logarithmic evaluation maps to a certain toric variety~\cite{RUK22}. The standard toric boundary relations give rise to relations on the higher rank log double ramification cycles, and by pushforward, in the log tautological ring of $\Mbar_{g,n}$. 

\section{Homological piecewise polynomials}
\label{sec:hompp}

\subsection{Conventions and homological Chow groups}
\label{sec:conventions}
Piecewise polynomials have been an important tool in logarithmic intersection theory on moduli spaces. For log smooth schemes/stacks, the theory is developed in \cite{MPS23, HS22, HMPPS}. However, for log stacks that are not log smooth, such as
the strata of $\Mbar_{g,n}$ or the moduli space of log pointed curve $\mathbb{M}_{g,n}^{\text{st}}$ of \cite{Holmes2023LogarithmicCohomologicalFT}, the theory is not yet well developed. 
In singular geometries, Chow {\it homology} classes are more natural to study, but piecewise polynomials are {\it cohomology classes}. There is, as yet, no homological version of piecewise polynomials. 

We propose here a definition of {\em homological piecewise polynomials} for idealised log smooth schemes. We establish basic properties, state the tropical interpretation, and study the question of 
proper pushforwards for
homological piecewise polynomials.
For log smooth schemes, the theory recovers the usual piecewise polynomials. 

In Section \ref{sec:logtautgeneral}, we will use language developed here  
to define the log tautological ring of $\Mbar_{g,n}$ (and, furthermore, to 
describe the log Chow ring of 
any scheme with a divisorial log structure).


\begin{definition}
An \emph{idealised log scheme} is a tuple $(\ul{X}, \alpha: M_X \to \Ocal_X, K_X)$ where $(\ul{X}, \alpha)$ is a log scheme and $K_X \subset \alpha^{-1}(0)$ is a monoid ideal inside $M_X$.
A morphism of idealised log schemes $f: X \to Y$ is a map on the underlying log schemes such that the map $f^\star K_Y \to M_X$ factors through $K_X$.
\end{definition}

A basic example may be helpful. Affine space $\mathbb A^n$ has a natural divisorial log structure coming from its coordinate boundary. Take any {\it monomial subscheme} $Z\subset\mathbb A^n$ and equip it with the pullback log structure. The monomials that vanish on $Z$ give a monoid ideal inside the logarithmic structure of $Z$. See Ogus \cite{Ogu06}, for a detailed treatment of idealised log schemes. We record a few examples that are relevant to our goals.

\begin{example}
Any log scheme with the empty sheaf of ideals forms an idealised log scheme. We call this a log scheme with \emph{trivial idealised structure}.
\end{example}

\begin{example}
\label{ex:schemencdividealised}
Let $X$ be a log scheme with log structure given by a normal crossings divisor $D$. Let $S$ be a stratum closure of $X$
endowed with the strict log structure from the embedding $i: S \to X$. For $U \subseteq X$ open we have \[M_X(U) = \{x \in \Ocal_X(U) : x|_{U \setminus D} \in \Ocal_{X}^\times(U\setminus D)\}.\]
As $S$ is a stratum, it is given by an ideal $K \subset M_X$. Then \[(S,\Mcal_S, K|_S)\] is an idealised log scheme.
In particular, if $S$ is a point and $\dim X = d$, then this idealised log scheme is
\[
(\mathsf{pt},\N^d,\N^d \setminus 0).
\]
\end{example}


The notion of idealised log smoothness is defined, as usual, via the infinitesimal lifting criterion, with respect to an appropriate class of square zero extensions. The appropriate class here is {\it idealised log thickenings}. These are square-zero thickenings $S\hookrightarrow S'$ of idealised log schemes that are strict as well as ideally strict. The latter condition being that the monoidal ideal $K_{S'}$ of $S'$ pulls back to the monoidal ideal $K_S$ associated to $S$. 

\begin{definition}
A map of idealised log stacks $f: X \to Y$ is \emph{idealised log smooth} if $f$ satisfies the infinitesimal lifting criterion for all idealised log thickenings of $Y$. A log scheme $X$ is \emph{idealised log smooth} if $X$ is idealised log smooth over the base point.
\end{definition}

We refer the reader to~\cite[Section~IV]{Ogu06} for further details.  

\begin{example}
All examples of idealised log schemes so far are idealised log smooth. 
Only those with generically trivial sheaf of monoids are log smooth.
\end{example}

For an Artin stack $X$, we let $\CH_\star(X)$ denote Kresch's Chow group \cite{Kresch_cycle} with $\Q$-coefficients. We let $\CH^\star(X)$ denote the operational Chow ring of $X$ with test objects given by
algebraic stacks of finite type stratified by global quotient stacks, as defined in \cite{BaeSchmitt1}. For DM-stacks, these correspond to the usual Chow groups and rings.
When $X$ is smooth, there is a natural Poincar\'e duality isomorphism
\[
\CH^\star(X) \xrightarrow{\sim} \CH_\star(X), \alpha \mapsto \alpha \cap [X]\,.
\]
Thus, when $X$ carries a log structure and is both smooth and log-smooth, we can define the logarithmic Chow ring $\LogCH^\star(X)$ either as a direct limit of $\CH^\star(\widehat X)$ over all log blowups $\widehat X \to X$, or as the direct limit of $\CH_\star(\widehat X)$ for smooth log blowups $\widehat X$, and obtain isomorphic results.

However, when $X$ is only idealized log smooth (as for the strata $S$ from Example \ref{ex:schemencdividealised}), there may be log blowups $\widehat X \to X$ which are no longer smooth (or even equidimensional). For these cases, we will need a definition of homological log Chow groups $\logCH_\star(X)$ -- and the necessary level of generality is to allow $X$ to be an Artin stack.

When $X$ is a log scheme, a definition of 
$\logCH_\star(X)$
was presented in \cite[Definition~2.7]{Barrott2019Logarithmic-Cho}. 
In the remainder of Section \ref{sec:conventions}, we recall the treatment of \cite{Barrott2019Logarithmic-Cho}.

\begin{definition} Log spaces and log stacks:
\begin{enumerate}[(i)]
    \item An \emph{algebraic log space} is an algebraic space with a log structure. 
    \item An \emph{algebraic log stack} is an algebraic stack with a log structure. 
    \item A \emph{log stack} is a stack in groupoids over $\LogSch$, with a log \'etale cover by an algebraic log stack and with diagonal representable by algebraic log spaces.
    \item A log stack $X$ is \emph{dominable} if it has a log blowup $\tilde{X}$ that is an algebraic log stack.  \defendhere
\end{enumerate}
\end{definition}

\begin{remark}
Given any algebraic log stack $X$, the category of log maps $S \to X$ is a log stack. For many purposes, it suffices to study algebraic log stacks: all examples of log stacks that we consider (outside of this Remark) are algebraic log stacks. However, there are important examples of log stacks which are not represented by an algebraic log stack.
\begin{enumerate}[(i)]
\item The log stacks $\G_m^\log$ and $\G_m^\trop$, which represent the functors $X \mapsto M_X(X)$ and $X \mapsto \ghost_X$ respectively. They are dominable by $\P^1$ and $[\P^1/\G_m]$ respectively, see~\cite[Proposition~1]{RW19}
\item The logarithmic Picard group $\mathsf{LogPic}/S$ of a log curve $C/S$, and, in particular, the universal logarithmic Picard group $\mathsf{LogPic}_{g,n}/\Mbar_{g,n}$. The space $\mathsf{LogPic}_{g,n}$ has representable log blowups given by universal compactified Jacobians associated to non-degenerate stability conditions, see  \cite[Section~4]{HMPPS}.
\item The moduli space $\mathsf{LogA}_g$ of log abelian varieties of dimension $g$, first defined in \cite{KKN21}. 
The standard toroidal compactifications of the moduli space $\mathcal{A}_g$ of principally polarized abelian varieties of dimension $g$, such as the perfect cone compactification and the second Voronoi compactification, are log blowups of $\mathsf{LogA}_g$.
\end{enumerate}
\end{remark}


\begin{definition}
A log stack $X$ is \emph{locally free} if every stalk of $\ghost_X$ is isomorphic to $\N^r$ for some $r$.
\end{definition}


\begin{example} The basic cases for us are:
\begin{enumerate}[(i)]
    \item The stack $\Mbar_{g,n}$ with its divisorial log structure is locally free, since the stalk of the ghost sheaf at $(C,p_1, \ldots, p_n)$ is isomorphic to $\mathbb{N}^r$ for $r$ the number of nodes of $C$. 
    \item A toric variety with the toric log structure is locally free if and only if the underlying toric variety is smooth.
\end{enumerate}
\end{example}

\begin{definition} \label{Def:idealized_log_smooth}
The category of idealised log schemes $(\ul{X}, \alpha: M_X \to \Ocal_X, K_X)$ is denoted $\mathsf{IdLogSch}$. The subcategory consisting of
idealised log schemes with maps $f: X \to Y$ satisfying the equality $K_X = f^\star K_Y$ of monoid ideals is denoted $\mathsf{IdLog}$. 
An idealised log stack is a log stack $X/\LogSch$ together with a map to $\mathsf{IdLog}/\LogSch$. 
\end{definition}

\begin{remark}
The category of maps from idealised log schemes to an idealised log stack forms a stack over $\mathsf{IdLogSch}$, and any stack over $\mathsf{IdLogSch}$ with a smooth cover by an idealised log stack is of this form.
\end{remark}

\begin{example}
For the gluing map $\iota_\Gamma : \Mbar_\Gamma \to \Mbar_{g,n}$ associated to a stable graph $\Gamma$, we write $\Mbar_{\Gamma}^{\str}$ for the stack $\Mbar_\Gamma$ endowed with the strict log structure induced from the standard divisorial log structure on $\Mbar_{g,n}$ via $\iota_\Gamma$.
The stack 
    $\Mbar_\Gamma^\str$ is locally free, where the stalk at $(C_v, q_{1,v}, \ldots, q_{n(v),v})_{v \in V(\Gamma)}$ is 
    \[
    \overline{\mathsf M}_{\Mbar_\Gamma^\str, (C_v)_v} = \mathbb{N}^{E(\Gamma)} \oplus \bigoplus_{v \in V(\Gamma)} \overline{\mathsf M}_{\Mbar_{g(v), n(v)}, (C_v, q_{1,v}, \ldots, q_{n(v),v})}\,.\]
To give $\Mbar_\Gamma^{\str}$ an idealised log structure, we must  give an idealised log structure on $S$ for every map $S \to \Mbar_{\Gamma}^{\str}$. We let the log ideal $K_S \subset M_S$ be the ideal generated by the lengths of the edges of $\Gamma$.
\end{example}


\begin{definition}
\label{def:virtreldim0pullback}
Let $X$ be a locally free algebraic log stack, and let $\pi: \tilde{X} \to X$ be a log blowup of locally free log stacks. We let $\pi^!$ denote the virtual relative dimension $0$ pullback $\CH_\star(X) \to \CH_\star(\tilde{X})$ constructed in \cite[Construction~2.4]{Barrott2019Logarithmic-Cho}.
\end{definition}


\begin{example}
Let $X$ be the point with characteristic monoid $\N^2$, and let $\tilde{X}$ be the $\P^1$ obtained by blowing up in the log ideal $(x_1,x_2)$. Using the excess intersection formula, we obtain 
\[
\pi^! [X] = -[H] \in \CH_\star(\tilde{X}),
\]
where $H$ is the hyperplane class on $\P^1$.
\end{example}
\begin{example}
Let $\tilde{X} \to X$ be a log blowup of locally free log smooth log stacks. Then, the Gysin pullback map $\pi^!$ is equal to the cohomological $\pi^\star$, applied via Poincar\'e duality.
\end{example}

\begin{definition}
\label{Def:homologicallogCH}
Let $X$ be a dominable log stack. We define the homological log Chow group $\logCH_\star(X)$ to be the colimit  
\[\logCH_\star(X) = \colim_{\widetilde{X} \to X} \CH_\star(\widetilde{X})\,,\]
where the colimit runs over all locally free algebraic log blowups of $X$, and the transition maps for a log blowup $\pi: \tilde{X}_1 \to \tilde{X}_2$ is the Gysin pullback $\pi^!$ from Definition \ref{def:virtreldim0pullback}.

We define the cohomological log Chow group $\LogCH^\star(X)$ to be the operational Chow ring for $\logCH_\star$, consisting of bivariant classes acting on $\logCH_\star(T)$ for maps $T \to X$, and commuting with saturated proper pushforward, log flat pullback and all strict Gysin maps (see \cite[Definition~2.20]{Barrott2019Logarithmic-Cho}).
\end{definition}

\begin{remark}
For a log smooth log stack $X$, we have by \cite[Corollary~2.23]{Barrott2019Logarithmic-Cho} a natural isomorphism
\[
\LogCH^\star(X) \to \logCH_\star(X)
\]
given by acting on the fundamental class $[X]$.
\end{remark}

\begin{remark}
There are several different definitions of $\LogCH^\star$ in the literature. For example, in \cite{HS22}, the definition
\[
\LogCH_{\mathsf{HSc}}^\star(X) =\colim_{\tilde{X} \to X} \CH_{\op}(\tilde{X})
\]
is used for log smooth stacks $X$.  In \cite{Holmes2023LogarithmicCohomologicalFT}, the definition
\[
\LogCH_{\mathsf{HSp}}^\star(X) = \colim_{\tilde{X} \to X} \CH_{\sf{OP}}(\tilde{X})
\]
is used (where $\CH_{\sf{OP}}$ is the operational Chow ring defined in \cite{BaeSchmitt1}).

There is a map $\LogCH_{\mathsf{HSp}}^\star(X) \to \LogCH_{\mathsf{HSc}}^\star(X)$. In \cite[Appendix~A]{Holmes2023LogarithmicCohomologicalFT}, a map $\LogCH_{\mathsf{HSp}}^\star(X) \to \LogCH^\star(X)$ is
constructed. Many constructions in the literature lie in $\LogCH_{\mathsf{HSp}}^\star(X)$ or can be canonically lifted to $\LogCH_{\mathsf{HSp}}^\star(X)$, and can then be mapped to $\LogCH^\star(X)$.
\end{remark}

\begin{remark}
When $X$ is not log smooth, many nice properties, such as the Poincar\'e duality isomorphism, fail. In  \cite{Barrott2019Logarithmic-Cho}, the notion of a \emph{compatible fundamental class} $[X] \in \logCH_\star(X)$,  represented by the fundamental class of any locally free log blowup of $X$ of maximal dimension, is introduced. However, it is then no longer the case that every class in $\logCH_\star(X)$ can be obtained from $[X]$ by the action of a logarithmic operational class. 

For example, for $X = (\mathsf{pt}, \mathbb{N}^2)$, the group $\logCH_\star(\mathsf{pt}, \mathbb{N}^2)$ is supported in degree $0$ and $1$ for dimension reasons. Moreover, it follows from Proposition~\ref{Pro:LogCH_pt_Nr} below that 
\[
\LogCH_1(\mathsf{pt}, \mathbb{N}^2) \cong \left\{f \in \pPP^1(\mathbb{R}_{\geq 0} ^2) : f(x,0)=f(0,y) = 0 \text{ for all }x,y \in \mathbb{R}_{\geq 0}\right\}
\]
is an infinite-dimensional $\QQ$-vector space (all functions $f=\min(ax, by)$ for $a,b \in \mathbb{Z}_{\geq 1}$ with $\gcd(a,b)=1$ are linearly independent). On the other hand, since $\LogCH^0(\mathsf{pt}, \mathbb{N}^2) = \mathbb{Q}$,  the map 
\[
\logCH^\star(\mathsf{pt}, \mathbb{N}^2) \to \logCH_{1-\star}(\mathsf{pt}, \mathbb{N}^2), \alpha \mapsto \alpha \cap [X]\]
only covers a 1-dimensional subspace of $\LogCH_1(\mathsf{pt}, \mathbb{N}^2)$.
An important motivation for introducing homological piecewise polynomials in Section \ref{subsec:definitionhomPP} is
to describe significantly larger parts of $\logCH_\star(X)$.

\end{remark}

\subsection{Cone stacks with boundary}
Recall from \cite{CCUW}  that a (combinatorial) cone stack is a category fibered in groupoids 
\begin{equation}
    C : \Sigma \to \mathbf{RPC}^f
\end{equation}
over the category $\mathbf{RPC}^f$ of rational polyhedral cones (with morphisms being face inclusions). For the precise technical conditions, see \cite[Definition 2.15]{CCUW}. We denote the objects of $\Sigma$ by $\sigma$ and the morphisms by $\sigma \to \sigma'$. Where there is no risk of confusion,
we identify the morphisms with the cones and face morphisms that they correspond to under $C$.

The category of Artin fans is usually defined as the category of Artin stacks that are log \'etale over $\Spec k$ and admit a strict \'etale cover by Artin cones, which are Artin stacks of the form $[\Spec k[M]/\Spec k[M^{\gp}]]$.
The category of cone stacks $\Sigma$ is naturally isomorphic to the category of Artin fans $\mathcal{A}_\Sigma$, see \cite[Theorem 6.11]{CCUW} for more details. 

While the Artin fan $\mathcal{A}_\Sigma$ itself is log smooth, a closed subset  $\mathcal{B} \subseteq \mathcal{A}_\Sigma$ with its reduced stack structure and induced log structure will in general only be idealized log smooth. Conversely, as explained in Section \ref{subsubs:bartinfan}, any idealized log smooth stack admits a strict and smooth map to such a set $\mathcal{B}$. 
In $\Sigma_X$, the set $\mathcal{B}$ is described by a collection of cones $\sigma_0$ in $\Sigma_X$ that are closed under taking face maps $\sigma_0 \to \sigma$. 
These ideas motivate the following definition.
\begin{definition}
\label{def:conestack}
A \emph{cone stack with boundary} $(\Sigma, \Sigma^0, \Delta)$ is a cone stack $\Sigma$ together with a full subcategory $\Sigma^0 \subseteq \Sigma$ which is forward-closed: for $\sigma_0 \in \Sigma^0$ and a morphism $\sigma_0 \to \sigma$ in $\Sigma$, we also have $\sigma \in \Sigma^0$. We call $\Sigma^0$ the \emph{interior} of $\Sigma$ and the complementary full subcategory $\Delta = \Sigma \setminus \Sigma^0$ its \emph{boundary}.
A morphism $$(\Sigma_1, \Sigma_1^0, \Delta_1) \to (\Sigma_2, \Sigma_2^0, \Delta_2)$$ of idealized cone stacks is a morphism $\Sigma_1 \to \Sigma_2$ of cone stacks sending $\Sigma_1^0$ to $\Sigma_2^0$. 
\end{definition}




For a cone stack with boundary $(\Sigma, \Sigma^0, \Delta)$, the subcategory $\Delta$ is a cone stack itself, and the map $\Delta \to \Sigma$ corresponds to an open embedding $\mathcal{A}_\Delta \to \mathcal{A}_\Sigma$ of the corresponding Artin fans. We denote by 
$$\mathcal{B}_{\Sigma^0} = \mathcal{A}_\Sigma \setminus \mathcal{A}_\Delta \subseteq \mathcal{A}_\Sigma$$
the complement of $\mathcal{A}_\Delta$ with its reduced substack structure. We can extract a colimit presentation of the stack $\mathcal{B}_{\Sigma^0}$ from the combinatorial data. Indeed, recall the classical presentation of the Artin fan
\[
\mathcal{A}_\Sigma = \colim_{\sigma \in \Sigma}\ \underbrace{\left[ \operatorname{Spec} k[S_{\sigma}] / \operatorname{Spec} k[S_{\sigma}^{\mathrm{gp}}] \right]}_{= \Acal_\sigma}\,,
\]
with $S_{\sigma} \subseteq \sigma^\vee$ the semigroup associated to $\sigma$. 

To find the corresponding presentation of the closed substack $\mathcal{B}_{\Sigma^0} \subseteq \mathcal{A}_\Sigma$, consider the \'etale cover $\coprod_\sigma \Acal_\sigma \to \Acal_\Sigma$ of the Artin fan $\Acal_\Sigma$ by Artin cones $\Acal_\sigma$. We then identify the closed substack of $\Acal_\sigma$ obtained as the preimage of the union of strata $\Bcal_{\Sigma^0}$ insider $\Acal_\Sigma$. This preimage is cut out by a toric ideal $J_{\sigma, \Sigma^0} \subseteq k[S_{\sigma}]$.

To calculate this ideal, note that in $\Acal_\Sigma$, the strata of $\mathcal{B}_{\Sigma^0}$ contained in the image of $\Acal_\sigma$ correspond to morphisms $\sigma' \to \sigma$ with $\sigma' \in \Sigma^0$. Since $\Sigma^0$ is forward-closed, this forces $\sigma \in \Sigma^0$. Given such a morphism, we see by an argument similar to \cite[Exercise 3.2.6]{CLS11} that the functions in $k[S_\sigma]$ vanishing on the preimage of the stratum associated to $\sigma'$ 
are given by the ideal\footnote{Note that equation (3.2.7) in \cite{CLS11} has a very unfortunate typo and should read $$I = \langle \chi^m | m \notin \tau^\perp \cap (\sigma')^\vee \cap M \rangle \subseteq \mathbb{C}[(\sigma')^\vee \cap M] = \mathbb{C}[S_{\sigma'}]\, .$$}
\[
J_{\sigma' \to \sigma} = (\chi^s : s \in S_\sigma \text{ with } s|_{\sigma'} \neq 0 \in (\sigma')^\vee) \subseteq k[S_\sigma]\,.
\]
To give an example: when $\sigma' = \sigma = \mathbb{R}_{\geq 0}^n$ then $S_\sigma = \mathbb{N}^n$ so that $k[S_\sigma] = k[x_1, \ldots, x_n]$. Then  among all elements $s \in S_\sigma$ only $s=0$ restricts to zero on the dual cone $(\sigma')^\vee = (\mathbb{R}_{\geq 0}^n)^\vee$. Thus
\[
J_{\sigma' \to \sigma} = (\chi^s : s \in S_\sigma \setminus \{0\}) = (x_1, \ldots, x_n) \subseteq k[x_1, \ldots, x_n]\,.
\]
And indeed, in the toric variety $\mathbb{A}^n = \mathrm{Spec} k[x_1, \ldots, x_n]$ associated to $\sigma$, the cone $\sigma' = \sigma$ itself corresponds to the origin in $\mathbb{A}^n$. Correspondingly, we have that $(x_1, \ldots, x_n)$ is the toric ideal of functions vanishing at the origin.

Returning to the general setting, we define the ideal
\[
J_{\sigma, \Sigma^0} = \bigcap_{\substack{\sigma' \to \sigma\\\sigma' \in \Sigma^0}} J_{\sigma' \to \sigma} \subseteq k[S_\sigma]
\]
of all functions in $k[S_\sigma]$ vanishing on some preimage of a stratum in $\Bcal_{\Sigma^0}$. Then we have a presentation
\begin{equation}
\label{eqn:Bartin_fan_as_colimit}
\Bcal_{\Sigma^0} = \colim_{\sigma \in \Sigma^0} \left[\Spec\left( k[S_{\sigma}]/J_{\sigma, \Sigma^0} \right)/\Spec\left( k[S_{\sigma}]^\gp \right) \right]\,.
\end{equation}
Compared to the presentation of $\Acal_\Sigma$, we could restrict to those $\sigma$ contained in $\Sigma^0$, since the others satisfy $J_{\sigma, \Sigma^0} = (1)$ and thus their contribution to the colimit \eqref{eqn:Bartin_fan_as_colimit} would be empty. In particular, the presentation \eqref{eqn:Bartin_fan_as_colimit} shows that the stack $\Bcal_{\Sigma^0}$ only depends on the category $\Sigma^0$ and the functor from this category to the category of rational polyhedral cones. 

\begin{remark}
The construction above gives an equivalence between the category of cone stacks with boundary and the category of embeddings $\Bcal \to \Acal$ of a reduced closed substack, where the maps are commutative squares. In general, the more natural category might be the category of (reduced) closed substacks of Artin fans, where the maps are maps of {algebraic log} stacks. For example, Proposition~\ref{prop:bartinfantocan} is not true for embedded idealised Artin fans, per Remark~\ref{rem:artinfansnomaps}.

Combinatorially, this category of reduces closed substacks of Artin fans is equivalent to the category of cone stacks with boundary localised at maps that are an isomorphism on the interior.
\end{remark}

\subsection{Scheme theoretic images inside Artin fans}
\label{subsubs:bartinfan}

We will use the language of cone stacks (with boundary) and Artin fans to describe the tropicalization of a log stack $X$. 

\begin{definition} \label{Def:an_Artin_fan}
For an algebraic log stack $X$,  \emph{an Artin fan of $X$} is the data of a strict morphism $X \to \mathcal{A}$  to an Artin fan $\mathcal{A}$ with geometrically connected fibers (where the empty set is connected) and non-empty fibers over minimal strata of $\Acal$.\footnote{We do not require smoothness of the map $X \to \mathcal{A}$, as this does not hold in general for idealised log smooth stacks. We will see in Lemma \ref{lem:xsmoothoverbx} that any idealised log smooth stack is smooth over its image in the Artin fan.}
Given an Artin fan $\Acal$ of $X$, the cone stack $\Sigma$ associated to $\mathcal{A}$ is \emph{a tropicalization} of $X$.
\end{definition}

\begin{remark}
\label{rem:theartinfananartinfan}
Every algebraic log stack $X$ satisfying some mild hypotheses admits a canonical Artin fan $X \to \mathcal{A}_X^\textup{can}$ as constructed in~\cite[Section~3.2]{ACMW} building on~\cite{AW}. This canonical Artin fan deserves to be called {\it the Artin fan}, but following this path leads to various well-known issues:
\begin{enumerate}[(i)]
    \item The canonical Artin fan of a log stack is not functorial as a map $X \to Y$ of log stacks does not necessarily induce a map $\mathcal{A}_X^\textup{can} \to \mathcal{A}_Y^\textup{can}$ which makes the natural diagram commute (see \cite[Section~5.4]{ACMUW} for a discussion). 
    \item The stack $X = B \mathbb{Z}/2\mathbb{Z}$ with trivial log structure is an Artin fan, but $\mathcal{A}_X^\textup{can}$ is point
    and is therefore not identical to $X$. In other words, $X$ is an Artin fan which is not its own (canonical) Artin fan.
    \item For $X = \Mbar_{g,n}$ and $\Sigma_{g,n}$ the moduli space of tropical curves (Definition \ref{Def:Mgntrop}), the map from $X$ to its (canonical) Artin fan factors as
    \begin{equation}
        \Mbar_{g,n} \xrightarrow{\afanmap} \mathcal{A}_{\Sigma_{g,n}} \xrightarrow{\afanmap'} \mathcal{A}_{\Mbar_{g,n}}^\textup{can}
    \end{equation}
    but the map $\afanmap'$ is not in general an isomorphism. The issue arises since the cone $\sigma_\Gamma \in \mathcal{A}_{\Sigma_{g,n}}$ associated to a stable graph $\Gamma$ can have non-trivial automorphisms acting trivially on the set of edges (by flipping two half-edges forming a loop). This automorphism acts trivially on the branches of the boundary divisor cutting out the stratum associated to $\Gamma$ and is therefore not seen in the canonical Artin stack $\mathcal{A}_{\Mbar_{g,n}}^\textup{can}$ (see \cite[Example 4.10]{U16} for a related discussion).
\end{enumerate}
All these issues are resolved by the added flexibility of being able to choose \emph{an} Artin fan:
\begin{enumerate}[(i)]
\item For any morphism $X \to Y$ of log stacks, we  have a commutative diagram
\[
\begin{tikzcd}
  X  \arrow[r]\arrow[d] & Y \arrow[d]\\
  \Acal_X \arrow[r] & \Acal_Y
\end{tikzcd}
\]
where $\Acal_X$ and $\Acal_Y$ are Artin fans for $X$ and $Y$ respectively, see~\cite[Section~2.5]{AW}.
\item The identity $B \mathbb{Z}/2\mathbb{Z} \to B \mathbb{Z}/2\mathbb{Z}$ is an Artin fan.
\item The map $\afanmap: \Mbar_{g,n} \to \mathcal{A}_{\Sigma_{g,n}}$ is a choice of an Artin fan, as proven in \cite[Theorem 4]{CCUW}.
\end{enumerate}

Whenever we refer to the Artin fan of a log stack $X$, we mean a fixed choice of an Artin fan, which we denote by $X \to \mathcal{A}_X$ with associated cone stack $\Sigma_X$. 
\end{remark}

\begin{remark}
A log stack with an Artin fan is automatically algebraic.  In Section \ref{sec:hompp}, we often assume the log stack has an Artin fan, to ease the exposition. For a dominable non-algebraic log stack, we can either blowup to make it admit an Artin fan, or use the {\em non-algebraic Artin fan} (a log stack that admits a log blowup by an Artin fan).
\end{remark}

To define homological piecewise polynomials, we first need one more ingredient lying inside the Artin fan.
\begin{definition}
\label{def:idealisedArtinfan}

Let $X$ be a log stack, and $X \to \Acal_X$ the map to an Artin fan of $X$. We define an \emph{idealised Artin fan} $\Bcal_X$ of $X$ as the scheme theoretic image of $X$ inside its Artin fan. We denote the scheme theoretic image of $X$ inside its canonical Artin fan $X \to \Acal_X^\textup{can}$ by $\Bcal_X^\textup{can}$.
\end{definition}
\begin{example}
Let $X$ be the point with log structure $\N^r$. Then we can choose $\Acal_X = \Acal_{\A^r} = [\A^r/\G_m^r]$, and $\Bcal_X = B\G_m^r$.
\end{example}
Despite the name, the definition of the idealized Artin fan (and other definitions in Section \ref{sec:hompp}) a priori make sense for any log stack $X$. However, the only context in which we know them to be well behaved is for idealised log smooth log stacks, the setting to which we restrict in the following. It is reasonable to expect that much of the theory we develop below is well behaved for any log stack that is log flat over its image in its Artin fan.

We will now show that any idealised Artin fan admits an \'etale map to the canonical idealised Artin fan, and that idealised Artin fans are functorial (in the sense of Remark~\ref{rem:theartinfananartinfan}).

\begin{proposition}
\label{prop:bartinfantocan}
    Let $X$ be an idealised log stack, with idealised Artin fan $X \to \Bcal_X$ and canonical idealised Artin fan $X \to \Bcal_X^\textup{can}$. There is a strict \'etale map $f: \Bcal_X \to \Bcal_X^\textup{can}$ making the diagram
    \[
    \begin{tikzcd}
        X \arrow[d] \arrow[dr] & \\
        \Bcal_X \arrow[r,"f"] & \Bcal_X^\textup{can}
    \end{tikzcd}
    \]
    commute.
\end{proposition}
\begin{proof}
First note that as the map $X \to \Bcal_X$ is strict with geometrically connected fibers we have 
\begin{equation} \label{eqn:monoid_equality_technical}
H^0(X,\Mbar_X) = H^0(\Bcal_X,\Mbar_{\Bcal_X}).
\end{equation}
We say $X$ is \emph{small} if the Artin fan $\Acal_X^\textup{can}$ of $X$ is an Artin cone.
We first prove the existence of $f$ for small $X$, with associated monoid $P$. Then $\Hom(Y,\Acal_X^\textup{can})$ is in bijection with $\Hom(P, H^0(Y,\Mbar_Y))$. By the equality \eqref{eqn:monoid_equality_technical} the map $X \to \Acal_X^\textup{can}$ now descends to a map $\Bcal_X \to \Acal_X^\textup{can}$. By functoriality of schematic image we obtain the map $f: \Bcal_X \to \Bcal_X^\textup{can}$ fitting in the diagram.

Now we show that if $X$ is small, then $f$ is strict \'etale. For this, we can work \'etale locally on $\Acal_X$ and on $X$. Then we can assume $\Acal_X$ is itself an Artin cone $[\Spec k[Q]/\Spec k[Q^\gp]]$ for some sharp fs monoid $Q$, and $X$ is still small. Then the argument above shows that $\Acal_X \to \Acal_X^\textup{can}$ is an isomorphism, and hence the map of idealised Artin fans is also an isomorphism. All in all, $f$ is \'etale locally an isomorphism, so $f$ is \'etale.

Now we will show this proposition holds for general $X$. We take a groupoid representation $\Bcal_{V} \rightrightarrows \Bcal_{U} \to \Bcal_X$ where the $\Bcal_U$ and $\Bcal_V$ are disjoint unions of small idealised Artin fans. Let $V \rightrightarrows U \to X$ be the representation obtained by pulling back $\Bcal_{V} \rightrightarrows \Bcal_{U} \to \Bcal_X$. In particular, $V \to \Bcal_V$ and $U \to \Bcal_V$ are idealised Artin fans, as they are strict maps with geometrically connected fibers to idealised Artin fans. We get the diagram
\[
\begin{tikzcd}
& V \arrow[ld] \arrow[rd] \arrow[d,shift left = 1.2] \arrow[d,shift right = 1.2]& \\
\Bcal_V\arrow[d,shift left = 1.2] \arrow[d,shift right = 1.2] & U \arrow[ld] \arrow[rd] \arrow[d] & \Bcal_V^\textup{can} \arrow[d,shift left = 1.2] \arrow[d,shift right = 1.2]\\
\Bcal_U \arrow[d] & X \arrow[ld] \arrow[rd]  & \Bcal_U^\textup{can}\arrow[d]\\
\Bcal_X & & \Bcal_X^\textup{can}\\
\end{tikzcd}
\]

Because $U$ and $V$ are themselves small, we obtain \'etale maps $f_U: \Bcal_U \to \Bcal_U^\textup{can}$ and $f_V: \Bcal_V \to \Bcal_V^\textup{can}$ fitting in the diagram, by the first part of the argument.

Recall that for a log stack $Z$, the idealised Artin fan $\Bcal_Z^\textup{can}$ is the colimit of $\Bcal_C^\textup{can} \rightrightarrows \Bcal_D^\textup{can}$ where $C \rightrightarrows D \to Z$ is a colimit diagram of strict maps and $C,D$ are disjoint unions of small log schemes. Hence $\Bcal_-^\textup{can}$ commutes with colimits of strict maps, and in particular $\Bcal_X^\textup{can}$ is the colimit of $\Bcal_V^\textup{can} \rightrightarrows \Bcal_U^\textup{can}$.
Then $f_U, f_V$ descend to a strict \'etale map $\Bcal_X \to \Bcal_X^\textup{can}$.
\end{proof}

\begin{remark}
\label{rem:artinfansnomaps}
Proposition~\ref{prop:bartinfantocan} is false for Artin fans. For example, for $X = (\mathsf{pt}, \mathbb{N}) \sqcup (\mathsf{pt}, \mathbb{N})$, the canonical Artin fan is $[\mathbb{A}^1 / \mathbb{G}_m] \sqcup [\mathbb{A}^1 / \mathbb{G}_m]$. If we pick $\Acal_X = [\P^1/\mathbb{G}_m]$, then there is no commutative diagram
\[
\begin{tikzcd}
X  \arrow[r, "\mathfrak{t}"] \arrow[rd, "\mathfrak{t}^\textup{can}",swap] & \Acal_X \arrow[d]\\
 & \mathcal{A}_X^\textup{can}
\end{tikzcd}
\]
\end{remark}

\begin{lemma}
\label{lem:funcartinfuncbartin}
Let $f\colon X \to Y$ is a map of log stacks, and assume functoriality of the Artin fan for $f$, i.e. that there is a $\Acal_f\colon \Acal_X \to \Acal_Y$ fitting in the obvious diagram. Then there is a map $\Bcal_f\colon \Bcal_X \to \Bcal_Y$ fitting in the obvious diagram.
\end{lemma}
\begin{proof}
For this, note that there is (by composition) a map $X \to \Bcal_Y$, and so the pullback of $\Bcal_Y$ to $\Acal_X$ is a closed substack of $A_X$  via which $X$ factors. But then (by definition of the scheme theoretic image) we find that $\Bcal_X$ is a subscheme of the pullback of $\Bcal_Y$, so we get a map $\Bcal_X \to \Bcal_Y$.
\end{proof}

\begin{definition}
Let $X$ be a log scheme and $d \in \Z$. We say $X$ has \emph{pure log dimension $d$} if for any log stratum\footnote{Here with a log stratum, we mean a log stratum on a strict \'etale cover.} $Z$ with generic point $\zeta$ the local log dimension $\mathsf{rk } \ghost_{\zeta} + \dim Z$ is equal to $d$. 
\end{definition}

\begin{lemma}
\label{lem:xsmoothoverbx}
Let $X$ be an idealized log smooth stack over $k$. Then the morphism $X \to \Bcal_X$ is smooth. If $X$ is pure of dimension $e$ and pure of log dimension $d$, then $\Bcal_X$ has dimension $e - d$ and the map $X \to \Bcal_X$ has relative dimension $d$. 
\end{lemma}
\begin{proof}
We will prove this locally on $X$. 
By definition, the map $X \to \mathrm{Spec}(k)$ to the point with trivial log structure is idealised log smooth.
Using \cite[Variant IV.3.3.5]{Ogu06}, such an  idealised log smooth map has charts, which in our case means that \'etale locally on $X$ there is a monoid $Q$, an ideal $K$ of $Q$ and a strict smooth\footnote{\cite[Variant IV.3.3.5]{Ogu06} claims this map is \'etale, but is not true (for example, it fails for $X$ a smooth scheme of positive dimension with trivial log structure). It should say ``smooth (resp. \'etale)''.} map 
\begin{equation}
b\colon X \to \Spec k[Q]/(K) =:Y. 
\end{equation}
Since the map $b$ is strict, there exists a strict and \'etale map $\Acal_X^\textup{can} \to \Acal_Y^\textup{can}$ of their canonical Artin fans such that the rectangular diagram on the right below commutes:
\[
\begin{tikzcd}
  & X  \arrow[r]\arrow[d] \arrow[dl] & Y \arrow[d]\\
  \Bcal_X \arrow[r, dashed, "\text{Pro.}~\ref{prop:bartinfantocan}", swap]  & \Bcal_X^\textup{can} \arrow[d, hook] \arrow[r, dashed] & \Bcal_Y^\textup{can} \arrow[d, hook]\\
  & \Acal_X^\textup{can} \arrow[r] & \Acal_X^\textup{can}
\end{tikzcd}
\]
In the diagram above, the surjectivity of $X \to \Bcal_X^\textup{can}$ implies that we obtain a map $\Bcal_X^\textup{can} \to \Bcal_Y^\textup{can}$ as indicated, and in fact it then follows that the lower square diagram is Cartesian, so that this map is also strict and \'etale. Secondly, from Proposition \ref{prop:bartinfantocan} it follows that we get a strict \'etale map $\Bcal_X \to \Bcal_X^\textup{can}$ as indicated.

Summarizing, we have obtained a commutative diagram of maps
\[
\begin{tikzcd}
  X  \arrow[r]\arrow[d] & Y \arrow[d]\\
  \Bcal_X \arrow[r] & \Bcal_Y^\textup{can}
\end{tikzcd}
\]
with $X \to Y$ smooth and $\Bcal_X \to \Bcal_Y^\textup{can}$ \'etale. Then if we show that $Y \to \Bcal_Y^\textup{can}$ is smooth, it follows that $X \to \Bcal_X$ is also smooth (since smoothness can be checked after composition with the \'etale morphism $\Bcal_X \to \Bcal_Y^\textup{can}$). 

All in all we are reduced to proving the claim for log schemes of the form $Y = \Spec k[Q]/(K)$, and for the canonical idealised Artin fan. The Artin fan of $\Spec k[Q]/(K)$ is $$\mathcal{A}(Q) = [\Spec k[Q]/\Spec k[Q^\gp]],$$ and the idealised Artin fan is the substack cut out by the ideal generated by $K$. Indeed, the map from $\Spec k[Q]$ to its Artin fan $\mathcal{A}(Q)$ is smooth, and the pullback of the substack cut out by $(K)$ is the subscheme cut out by $(K)$. In particular, the induced map $\Spec k[Q]/(K) \to \mathcal{A}(Q)$ is smooth onto its scheme theoretic image.

The claim on the relative dimension also follows immediately from the same claim for the charts.
\end{proof}

Given an idealised log smooth log stack, we will now use its idealised Artin fan to create a cone stack with boundary. We recall that there is an equivalence of categories between Artin fans and cone stacks.
\begin{definition}
Let $X$ be an idealised log smooth log stack and $X \to \mathcal{A}_X$ an Artin fan. Let $\Sigma_X$ be the cone stack corresponding to $\Acal_X$, and $\Delta_X \to \Acal_X$ be the cone stack corresponding to the open embedding $\Acal_X \setminus \Bcal_X \to \Acal_X$. Then we define $\Sigma_X^\circ$ to be $\Sigma_X \setminus \Delta_X$, and the \emph{tropicalization} of $X$ to be the cone stack with boundary $(\Sigma_X, \Sigma_X^\circ, \Delta_X)$.
\end{definition}
\begin{example}  
If $X$ is log smooth, then $\Acal_X = \Bcal_X$, and we get the empty boundary.
\end{example}
\begin{example}\label{Exa:Faces}
Given a rational polyhedral cone $\sigma$, consider the cone stack $\mathsf{Faces}(\sigma)$ whose objects are the faces $\sigma_0 \prec \sigma$ and whose morphisms are inclusions of faces induced by the identity of $\sigma$. If $\sigma$ is smooth, the stack $\mathsf{Faces}(\sigma)$ has $2^{\dim \sigma}$ many pairwise non-isomorphic objects. Declare its interior $\sigma^0$ to be the full subcategory consisting only of $\sigma$ itself. Then we obtain a cone stack with boundary $(\mathsf{Faces}(\sigma), \sigma^0, \Delta_\sigma)$.
If $X$ is the point with characteristic monoid $\N^r$ and sheaf of ideals $\N^{r} \setminus 0$, then its cone stack (with boundary) is given by  $\Sigma_X = \mathsf{Faces}(\mathbb{R}_{\geq 0}^r)$.
\end{example}

In Section \ref{subsec:definitionhomPP}, we will be interested in the (log) Chow groups of these Artin fans and their substacks, and their relations with the Chow group of $X$. We have the following result on the virtual relative dimension $0$ pullback used in Definition~\ref{Def:homologicallogCH}, the definition of the log Chow group.
\begin{lemma}
\label{lem:virtreldim0pullbackeqgysinpullback}
Let $\pi\colon \tilde{\Acal}_X \to \Acal_X$ be a log blowup. Consider the associated diagram of cartesian squares
\[
\begin{tikzcd}
 \tilde{X} \arrow[r] \arrow[d, "\pi_X"] & \tilde{\Bcal}_X \arrow[d,"\pi_{\Bcal}"] \arrow{r} & \tilde{\Acal}_X \arrow[d,"\pi"] \\
 X \arrow[r] & \Bcal_X \arrow[r] & \Acal_X
\end{tikzcd}
\]
Then the virtual relative dimension $0$ pullbacks
\begin{align*}
\CH_\star(X) \xrightarrow{\pi_X^!} \CH_\star(\tilde{X}), \ \ \CH_\star(\Bcal_X) \xrightarrow{\pi_\Bcal^!} \CH_\star(\tilde{\Bcal}_X), \ \ \CH_\star(\Acal_X) \xrightarrow{\pi^!} \CH_\star(\tilde{\Acal}_X) 
\end{align*}
in Definition~\ref{def:virtreldim0pullback} are all equal to the Gysin pullback of the log blowup $\pi$.
\end{lemma}
\begin{proof}
This follows immediately from Construction 2.4 of \cite{Barrott2019Logarithmic-Cho}.
\end{proof}

\subsection{Definitions}
\label{subsec:definitionhomPP}
For a log stack $X$ with Artin fan $\mathcal{A}_X$ and associated cone stack $\Sigma_X$, we denote the ring of strict piecewise polynomials on $\Sigma_X$ by $\sPP^\star(X)$, and similarly the ring of piecewise polynomials on $\Sigma_X$ by $\pPP^\star(X)$. There are natural isomorphisms 
\begin{equation}
    \Phi: \sPP^\star(X) \to \CH^\star(\Acal_X) \text{ and }\Phi^\mathrm{log}: \pPP^\star(X) \to \LogCH^\star(\Acal_X)\,,
\end{equation}
of graded algebras (see \cite[Section 3.4]{HS22} for the construction and \cite[Theorem 14]{MPS23} for the proof that they give isomorphisms). Next we introduce homological versions of these polynomials.



\begin{definition}
\label{def:sppandpp}
Let $(\Sigma, \Sigma^\circ, \Delta)$ be a cone stack with boundary. Then we define
\[
\sPP_\star(\Sigma, \Delta) \subseteq \sPP^\star(\Sigma) \text{ and }\pPP_\star(\Sigma, \Delta) \subseteq \pPP^\star(\Sigma)
\]
to be the set of (strict) piecewise polynomials on $\Sigma$ vanishing on all cones of $\Delta$. These $\QQ$-vector spaces carry a grading with values $k \in \mathbb{Z}_{\leq 0}$ given by the \emph{negative degree} of the piecewise polynomial, i.e. $\sPP_{k} \subseteq \sPP^{-k}$. They also have a natural module structure over $\sPP^\star(\Sigma)$ (respectively, $\pPP^\star(\Sigma)$).

For $X$ an idealised log smooth log stack of pure log dimension $d$ and $(\Sigma, \Sigma^\circ, \Delta) = (\Sigma_X, \Sigma^\circ_X, \Delta_X)$ the associated tropicalization of $X$, we also write
\[
\sPP_\star(X) = \sPP_\star(\Sigma_X, \Delta_X) \subseteq \sPP^\star(X) \text{ and }\pPP_\star(X) = \pPP_\star(\Sigma_X, \Delta_X) \subseteq \pPP^\star(X)
\]
for the sets above.
\end{definition}

\begin{remark}
A priori $\sPP_\star(\Sigma,\Delta)$ depends on the full data of $(\Sigma,\Sigma^\circ, \Delta)$. A posteriori however, it only depends on $\Sigma^\circ$, as we can also define it as a system of polynomials $(f_\sigma)_{\sigma}$ for $\sigma \in \Sigma^\circ$ that is compatible with face-pullbacks for morphisms in $\Sigma^\circ$ and such that for a face $\tau \preceq \sigma$ that is \emph{not} the image of some object in $\Sigma^\circ$, we have $f_{\sigma}|_{\tau} = 0$.
\end{remark}

\begin{theorem} \label{Thm:sPP_lowerstar_isom}
Let $(\Sigma, \Sigma^\circ, \Delta)$ be a cone stack with boundary, with $\Sigma$ smooth and denote by $i: \mathcal{B}_{\Sigma^\circ} \to \mathcal{A}_\Sigma$ the inclusion of the associated closed substack. Then there is a commutative diagram
\begin{equation}
\begin{tikzcd}
\sPP_\star(\Sigma, \Delta) \arrow[r] \arrow[d, "\Psi"] & \sPP^\star(\Sigma) \arrow[d, "\Phi"]\\
    \CH_\star(\Bcal_{\Sigma^\circ}) \arrow[r,"i_\star"]  & \CH_\star(\Acal_\Sigma) = \CH^\star(\Acal_\Sigma) 
\end{tikzcd}
\end{equation}
where both vertical maps are isomorphisms of graded $\QQ$-vector spaces. Similarly we have a diagram
\begin{equation} \label{eqn:PP_log_CH_Artin_fan_diagram}
\begin{tikzcd}
\pPP_\star(\Sigma, \Delta) \arrow[r] \arrow[d, "\Psi^\mathrm{log}"] & \pPP^\star(\Sigma) \arrow[d, "\Phi^\mathrm{log}"]\\
    \logCH_\star(\Bcal_{\Sigma^\circ}) \arrow[r,"i_\star"]  & \logCH_\star(\Acal_\Sigma) = \logCH^\star(\Acal_\Sigma) 
\end{tikzcd}
\end{equation}
with both vertical maps being isomorphisms, even for $\Sigma$ not necessarily smooth.
\end{theorem}
In order to prove the above result, we want to use the theory of higher Chow groups of Artin stacks, applied to the various Artin fans and their open or closed substacks. This theory was originally developed by Bloch \cite{Bloch} for schemes, with the purpose of extending the excision exact sequence of Chow groups on the left. It was generalized to the setting of Artin stacks by Kresch \cite{Kresch_cycle}, and there are now modern approaches via \'etale motivic Borel-Moore homology \cite{Khan}. For recent applications of these higher Chow groups in intersection theory of moduli spaces see \cite{BaeSchmitt2, Eric_Larson_higher_Chow, Hannah_Larson_higher_Chow, Bishop_integral}. Below, we use several formal properties of these higher Chow groups which were proven in \cite{Khan}, such as their functoriality under proper pushforwards and flat pullbacks, the extension of the excision sequence \cite[Theorem 2.18]{Khan} and certain (cap) products \cite[Section 2.2.6]{Khan}. 
The forthcoming paper \cite{BaePark_comparison} by Bae and Park will give a comparison result of those (higher) Chow groups with the groups defined by Kresch (and used in the remaining paper). The comparison for the (zeroth) Chow groups is established in \cite[Example 2.10]{Khan}, and since these are what we ultimately care about in the proof of Theorem~\ref{Thm:sPP_lowerstar_isom}, we can use Khan's formalism in the technical arguments and specialize to the standard Chow groups in the end.

To set up notation, let $\mathcal{X}$ be an algebraic stack of finite type over $k$, stratified by quotient stacks. Then there exist the first higher Chow groups $\CH_\star(\mathcal{X}, 1)$ with $\mathbb{Q}$-coefficients. If $\mathcal{U} \subseteq \mathcal{X}$ is open with complement $\mathcal{Z}$, then there is an exact sequence
\begin{equation} \label{eqn:long_excision}
    \CH_\star(\mathcal{Z}, 1) \to \CH_\star(\mathcal{X},1) \to \CH_\star(\mathcal{U},1) \xrightarrow{\partial} \CH_\star(Z) \to \CH_\star(X) \to \CH_\star(U) \to 0\,.
\end{equation}
In general, the groups $\CH_\star(\mathcal{X}, 1)$ can be quite large, even for very simple stacks $\mathcal{X}$. One reason for this is that they always contain a piece coming from the higher Chow group $\CH^1(k, 1)$ of the base field $k$, which is non-trivial in general. Denote by
\begin{equation}
    \overline{\CH}_\star(\mathcal{X}, 1) = \CH_\star(\mathcal{X},1) / \mathsf{im}(\CH_{\star+1}(\mathcal{X})  \otimes \CH^1(k, 1))
\end{equation}
the \emph{indecomposable part} of the higher Chow group (see \cite[Section 2.3]{BaeSchmitt2}). The boundary map $\partial$ in the excision sequence \eqref{eqn:long_excision} factors through $\overline{\CH}_\star(\mathcal{U}, 1)$ by \cite[Remark 2.18]{BaeSchmitt2}. In particular, when this latter group vanishes, the pushforward of Chow groups under the inclusion $\mathcal{Z} \hookrightarrow \mathcal{X}$ is injective. When $\mathcal{U} = B(\mathbb{G}_m^n \rtimes G)$ for $G$ finite, this vanishing $\overline{\CH}_\star(\mathcal{U}, 1) = 0$ follows from \cite[Proposition 2.14, Remark 2.21]{BaeSchmitt2}. The following result allows to extend this vanishing to a broader range of Artin stacks.
\begin{proposition} 
Let $\mathcal{U}$ be an algebraic stack of finite type over $k$, stratified by quotient stacks, such that it has a locally closed stratification by stacks $\mathcal{U}_i$ with $\overline{\CH}_\star(\mathcal{U}_i, 1) = 0$. Then $\overline{\CH}_\star(\mathcal{U}, 1) = 0$. In particular, this vanishing holds if $\mathcal{U}$ is a smooth Artin fan.
\end{proposition}
\begin{proof}
We do induction on the number of strata $\mathcal{U}_i$, with the base case $\mathcal{U}=\mathcal{U}_i$ being trivial. In the general case, let $\mathcal{V}=\mathcal{U}_i$ be an open stratum inside $\mathcal{U}$, with complement $\mathcal{Z}$. Then letting $K = \CH^1(k, 1)$ and using the excision sequence \eqref{eqn:long_excision}, we have a diagram
\begin{equation}
\begin{tikzcd}
\CH_{\star+1}(\mathcal{Z}) \otimes K \arrow[r] \arrow[d] & \CH_{\star+1}(\mathcal{U}) \otimes K \arrow[r] \arrow[d] & \CH_{\star+1}(\mathcal{V}) \otimes K \arrow[r] \arrow[d] & 0  \arrow[d]\\
\CH_\star(\mathcal{Z},1) \arrow[r] \arrow[d] & \CH_\star(\mathcal{U},1) \arrow[r] \arrow[d] & \CH_\star(\mathcal{V},1) \arrow[r] \arrow[d] & \CH_\star(Z)\\
\underbrace{\overline \CH_\star(\mathcal{Z},1)}_{=0\text{ by induction}} & \overline \CH_\star(\mathcal{U},1) & \underbrace{\overline \CH_\star(\mathcal{V},1)}_{=0\text{ by assumption}} & 
\end{tikzcd}
\end{equation}
with exact rows and columns. Here the vertical arrows from the first to the second row are the cap products mentioned before, which are compatible with the proper pushforward under $\mathcal{Z} \hookrightarrow \mathcal{X}$ and the open restriction under $\mathcal{V} \subseteq \mathcal{U}$. 
This shows that the two left squares are commutative, and the commutativity of the right square follows from \cite[Remark 2.18]{BaeSchmitt2}. Then the four-lemma implies the vanishing of $\overline{\CH}_\star(\mathcal{U}, 1)$. 
Finally, any smooth Artin fan has a locally closed stratification with strata $B(\mathbb{G}_m^n \rtimes G)$ (corresponding to an object $\sigma$ in the associated cone stack mapping to $\mathbb{R}_{\geq 0}^n$ and having automorphism group $G$). The vanishing of the indecomposable part of the higher Chow group for these pieces was discussed before, so the proposition applies.
\end{proof}

\begin{proof}[Proof of Theorem \ref{Thm:sPP_lowerstar_isom}]
Let $\mathcal{A}_{\Delta} \subseteq \mathcal{A}_\Sigma$ be the open complement of $\mathcal{B}_{\Sigma^\circ}$. Then applying the excision sequence to that open substack, we have a commutative diagram of solid arrows
\[
\begin{tikzcd}
\CH_\star(\mathcal{A}_{\Delta}, 1) \arrow[r, "\partial"]  & \CH_\star(\mathcal{B}_{\Sigma^\circ}) \arrow[r] &  \CH_\star(\mathcal{A}_{\Sigma})\arrow[r] & \CH_\star(\mathcal{A}_{\Delta}) \arrow[r] & 0\\
0\arrow[r] & \sPP_\star(\Sigma, \Delta) \arrow[r] \arrow[u, dashed, "\Psi"]  & \sPP^\star(\Sigma) \arrow[r] \arrow[u, "\sim"] & \sPP^\star(\Delta)\arrow[r] \arrow[u, "\sim"] & 0
\end{tikzcd}
\]
with exact rows (the upper row by excision, the lower by definition of $\sPP_\star(\Sigma, \Delta)$. Since the indecomposable part $\overline{\CH}_\star(\mathcal{A}_\Delta, 1)$ of $\CH_\star(\mathcal{A}_{\Delta}, 1)$ vanishes and since $\partial$ factors through it, we have that $\partial=0$. Then it follows that a unique dashed arrow $\Psi$ as above exists and is an isomorphism.

To show the statement for the logarithmic Chow groups, first note that all spaces involved in the diagram \eqref{eqn:PP_log_CH_Artin_fan_diagram} are invariant under replacing $\Sigma$ by a refinement and $\Bcal_{\Sigma^\circ}$ and $\Acal_\Sigma$ by the corresponding log blowups. Choosing a suitable refinement we can reduce to the case that $\Sigma$ is smooth. In this setting, note that any log blowup $q: \widehat \Bcal \to \Bcal_{\Sigma^\circ}$ is obtained as a total transform of a log blowup $p: \Acal_{\widehat \Sigma} \to \Acal_\Sigma$ for a subdivision $\widehat \Sigma$ of $\Sigma$. Moreover, for the choice of such a subdivision, the diagram
\[
\begin{tikzcd}
\CH_\star(\widehat \Bcal) \arrow[r, hook] & \CH_\star(\Acal_{\widehat \Sigma})\\
\CH_\star(\mathcal{B}_{\Sigma^\circ}) \arrow[r, hook] \arrow[u, "q^!"]  &  \CH_\star(\mathcal{A}_{\Sigma}) \arrow[u, "p^\star"]
\end{tikzcd}
\]
commutes by Lemma~\ref{lem:virtreldim0pullbackeqgysinpullback}. Under the identifications with spaces of strict piecewise polynomials in the first part, this shows that the virtual Gysin pullback $q^!$ is given by the standard restriction map
\[
\sPP_\star(\Sigma, \Delta) \to \sPP_\star(\widehat \Sigma, \widehat \Delta)
\]
to the subdivision $\widehat \Sigma$ of $\Sigma$. Taking the colimit over such subdivisions, we obtain the diagram \eqref{eqn:PP_log_CH_Artin_fan_diagram}
where the isomorphism in the lower right corner was proven in \cite[Corollary~2.23]{Barrott2019Logarithmic-Cho}.
\end{proof}

\begin{corollary}
\label{cor:vanishingpp}
Let $X$ be an idealised log smooth log stack of pure log dimension $d$. Let $i: \Bcal_X \to \Acal_X$ denote the inclusion. Then if $X$ is locally free, there is a commutative diagram
\[
\begin{tikzcd}
    \sPP_\star(X) \arrow[r,"i_\star"] \arrow[d, "\Psi"] &  \arrow[d, "\Phi"] \sPP^\star(X)\\
    \CH_\star(\Bcal_X) \arrow[r] & \CH_\star(\Acal_X) = \CH^\star(\Acal_X)
\end{tikzcd}
\]
where both vertical maps are isomorphisms of graded groups.

The same diagram holds verbatim, replacing $\sPP$ by $\pPP$ and $\CH$ by $\LogCH$ (see \eqref{eqn:PP_log_CH_Artin_fan_diagram}), even for $X$ not necessarily locally free.
\end{corollary}
\begin{proof}
Let $(\Sigma_X, \Sigma^\circ_X, \Delta_X)$ be the tropicalization of $X$. The corollary follows by applying Theorem \ref{Thm:sPP_lowerstar_isom} to the cone stack $(\Sigma, \Sigma^\circ, \Delta) = (\Sigma_X, \sigma_X^\circ, \Delta_X)$.
\end{proof}
Assume $X$ is idealized log smooth of pure log dimension $d$. Composing $\Psi$ and $\Psi^\mathrm{log}$ with the pullback under the map $X \to \Bcal_X$, which is strict and smooth by Lemma \ref{lem:xsmoothoverbx}, we obtain maps
\begin{equation}
\Psi : \sPP_\star(X) \to \CH_{\star+d}(X) \text{ and }\Psi^\mathrm{log} : \pPP_\star(X) \to \logCH_{\star+d}(X)\,,
\end{equation}
which we denote again by $\Psi, \Psi^\mathrm{log}$ if there is no risk of confusion (and by $\Psi_X, \Psi_X^\mathrm{log}$ otherwise).

\begin{example}
\label{ex:ptwithnr}
Consider the log scheme $X = (\mathsf{pt}, \mathbb{N}^r)$ with cone stack $\Sigma_X = (\mathbb{R}_{\geq 0})^r$, which is idealized log smooth of pure log dimension $r$. Then $\sPP_\star(X)$ is the free module over $\sPP^\star(X)=\mathbb{Q}[x_1, \ldots, x_r]$ generated by $x_1x_2 \dots x_r$ of degree $-r$. 
We have piecewise polynomial functions
\[
x_1x_2 \dots x_r \in \sPP_{-r}(X) \text{ and }\min(x_1,\dots,x_r) \in \pPP_{-1}(X)\,.
\]
If $\P^{r-1} \to \mathsf{pt}$ denotes the log blowup obtained by  blowing up in the monoid ideal $(x_1,\dots,x_r)$, then we have
\[
\Psi(x_1x_2 \dots x_r) = [\mathsf{pt}] \in \CH_0(\mathsf{pt}) \text{ and }\Psi^\textup{log}(\min(x_1,\dots,x_r)) = [\P^{r-1}] \in \logCH_{r-1}(\mathsf{pt})
\]
give the fundamental class of $X$ and the fundamental class of its log blowup $\P^{r-1}$, respectively. To see these formulas, consider the Artin fan $\mathcal{A}_X = [\mathbb{A}^n / \mathbb{G}_m^n]$ and its blowup $\widetilde{\mathcal{A}} \to \mathcal{A}_X$ at the origin $B \mathbb{G}_m^n$ with exceptional divisor $\mathcal{E} \subseteq \widetilde{\mathcal{A}}$. We obtain 
\[
\Phi(x_1 x_2 \dots x_r) = [B \mathbb{G}_m^n] \in \CH_{-r}(\mathcal{A}_X) \text{ and } \Phi^\textup{log}(\min(x_1,\dots,x_r)) = [\mathcal{E}] \in \CH_{-1}(\widetilde{\mathcal{A}})
\]
using the results of \cite[Section 6.2]{HMPPS}. Together with the injectivity of the horizontal maps in Corollary \ref{cor:vanishingpp}  this implies the claimed formulas.
\end{example}

In fact, with some more work we can fully understand the logarithmic Chow group of the log scheme $X$ above.
\begin{proposition}  \label{Pro:LogCH_pt_Nr}
For the log scheme $X = (\mathsf{pt}, \mathbb{N}^r)$ with cone stack $\Sigma_X = (\mathbb{R}_{\geq 0})^r$ we have that the map
\[
\Phi: \mathsf{PP}_\star(X) \to \logCH_\star(\mathsf{pt}, \mathbb{N}^r)
\]
is surjective. Its kernel is given by the submodule
\[
K = \{c_1 x_1 + \ldots + c_r x_r : c_1, \ldots, c_k \in \mathsf{PP}_\star(X)\} \subseteq \mathsf{PP}_\star(X)\,.
\]
obtained from the ideal $I = (x_k: k=1, \ldots, r) \subseteq \mathsf{PP}^\star(X)$ as $K = I \cdot \mathsf{PP}_\star(X)$.
\end{proposition}
\begin{proof}
Consider the stack $\mathcal{B}_X = (B \mathbb{G}_m^r, \mathbb{N}^r)$ associated to $X$ and the universal vector bundle $\pi: V = [\mathbb{A}^r / \mathbb{G}_m^r] \to \mathcal{B}_X$ with its induced strict log structure from the target. Then we have a diagram
\[
X \xrightarrow{i} V \xrightarrow{\pi} \mathcal{B}_X
\]
representing $X \cong [\mathbb{G}_m^r / \mathbb{G}_m^r]$ as an open substack (via the map $i$) in the vector bundle $V$. Since the map $X \to \mathcal{B}_X$ induces an isomorphism on Artin fans, any log blowup $\widetilde X$ of $X$ fits into a fiber diagram
\begin{equation} \label{eqn:exact_sequence_point_NNr}
\begin{tikzcd}
\widetilde X \arrow[r, "\widetilde i"] \arrow[d] & \widetilde V \arrow[r, "\widetilde \pi"] \arrow[d] & \widetilde{\mathcal{B}} \arrow[d]\\
X \arrow[r, "i"] & V \arrow[r, "\pi"] & \mathcal{B}_X
\end{tikzcd}
\end{equation}
where $\widetilde \pi$ is still a vector bundle of rank $r$, and $\widetilde i$ is an open embedding, representing $\widetilde X$ as the open complement of the union $\widetilde Z$ of $r$ hyperplane sections $\widetilde Z_1, \ldots, \widetilde Z_r \subseteq \widetilde V$. Let $j_k : \widetilde Z_k \to \widetilde V$ be the inclusion, then the map
\[
\bigsqcup_{k=1}^r \widetilde Z_k \xrightarrow{\sqcup j_k} \widetilde Z
\]
is proper, representable and surjective and thus induces a surjection of Chow groups. Then the excision sequence for Chow groups implies that we have an exact sequence
\begin{equation}
    \bigoplus_{k=1}^r \CH_\star(\widetilde Z_k) \xrightarrow{\sum_k (j_k)_\star} \CH_\star(\widetilde V) \xrightarrow{\widetilde i^\star} \widetilde X \to 0\,.
\end{equation}
Both $\widetilde V$ and all $\widetilde Z_k$ are vector bundles over $\widetilde{\Bcal}$, and thus the composition
\[
\sPP_\star(\widetilde{\mathcal B}) \xrightarrow{\Psi} \CH_\star(\widetilde{\mathcal B}) \xrightarrow{\pi^\star} \CH_{\star+r}(\widetilde V)
\]
is an isomorphism (the first map is an isomorphism by Corollary \ref{cor:vanishingpp}, the second by \cite[Theorem 2.1.12 (vi)]{Kresch_cycle}). Similarly, we have an isomorphism
\[
\sPP_\star(\widetilde{\mathcal B}) \xrightarrow{\Psi} \CH_\star(\widetilde{\mathcal B}) \xrightarrow{\pi^\star} \CH_{\star+r-1}(\widetilde Z_k)\,.
\]
Using these identifications in \eqref{eqn:exact_sequence_point_NNr}, the map $(j_k)_\star : \sPP_\star(\widetilde{\mathcal B}) \to \sPP_\star(\widetilde{\mathcal B})$ is given by multiplication with the piecewise linear function $x_i$. 

This shows that $\sPP_\star(\widetilde{\mathcal B}) \to \CH_\star(\widetilde {\mathcal B})$ is surjective with kernel $(x_k: k=1, \ldots, r)$. Taking the direct limit over all log blowups $\widetilde{\mathcal B} \to {\mathcal B}_X$ we obtain the desired statement (where similar to the proof of Theorem \ref{Thm:sPP_lowerstar_isom} we use Lemma \ref{lem:virtreldim0pullbackeqgysinpullback} to show that the refined Gysin pullbacks between the groups $\CH_\star(\widetilde {\mathcal B})$ correspond to restrictions of piecewise polynomials to subdivisions).
\end{proof}

\begin{example} \label{Exa:triple_axis}
Let $X$ be given by the union of the three axes inside $\A^3$, with log structure the pullback of the toric log structure on $\A^3$. Then $\Bcal_X = [X/\G_m^3]$ and $\Acal_X = [\A^3/\G_m^3]$. We have $\Sigma_X = \RR_{\geq 0}^3$, and the boundary $\Delta_X$ consists of the three rays generated by the three basis vectors. The strict piecewise polynomials vanishing on $\Delta_X$ are, as a $\Q[x,y,z]$-module, generated by $yz,xz$ and $xy$. These three generators map to the classes of the three axes in $\CH_\star(X)$ under the map $\Psi$.
\end{example}

\begin{example}
Let $X$ be a log smooth scheme pure of dimension $d$ (which hence is pure of log dimension $d$). Then we have a commutative diagram
\[
\begin{tikzcd}
\pPP^\star(X) \arrow[d] \arrow[r, "\sim"] & \pPP_{-\star}(X) \arrow[d] \\
\logCH^\star(X) \arrow[r, "\sim"] & \LogCH_{d-\star}(X)
\end{tikzcd}
\]
where the horizontal maps are given by acting on the constant function $1$ (top) and the fundamental class of $X$ (bottom) respectively.
Note that the isomorphism $\pPP_{-\star}(X) \to \pPP^\star(X)$ also has a natural interpretation, as the inclusion of the piecewise polynomials vanishing on the empty boundary $\Delta_X = \emptyset$ of $\Sigma_X$ inside the set of all piecewise polynomials.
\end{example}


\subsection{Pushing forward piecewise polynomials} \label{Sect:push_forward_PP}
We define, for a certain class of maps, a pushforward on the level of homological piecewise polynomials, show the projection formula, and verify compatibility with the usual pushforward on (log) Chow groups. 

Our notion of pushforward works for maps of so-called relative log dimension $0$, as per the following definition.
\begin{definition} \label{Def:relative_log_dimension_0}
Let $X$ and $Y$ be two idealised log smooth log stacks of pure log dimension $d$. Let $f: X \to Y$ be a proper saturated morphism of log stacks which descends to a map on Artin fans. 

We say $f$ is \emph{of relative log dimension $0$} if the map on tropicalisations $\Sigma_X \to \Sigma_Y$ is a relative dimension $0$ map between the dimension $d$ cone stacks, i.e. any cone $\sigma \in \Sigma_X$ is mapped to a cone of dimension $\dim(\sigma)$ in $\Sigma_Y$.
\end{definition}

\begin{proposition}  \label{Prop:sPP_pushforward}
Let $X$ and $Y$ be two idealised log smooth log stacks and $f: X \to Y$ be a proper map of relative log dimension $0$ and of relative Deligne--Mumford type, for which there is a cartesian square 
\begin{equation} \label{eqn:X_B_fiber_diagram}
\begin{tikzcd}
  X  \arrow[r,"f"]\arrow[d] & Y \arrow[d] \\
  \Bcal_X \arrow[r,"\Bcal_f"] & \Bcal_Y
\end{tikzcd}
\end{equation}
Then the map $\Bcal_f$ is proper and saturated. 
Let $f^\textup{trop} : \Sigma_X \to \Sigma_Y$ be the map associated to $\Acal_f$. Then we have a map
\[
(f^\textup{trop})_\star : \sPP_\star(X) \overset{\Psi_X}{\cong}\CH_\star(\Bcal_X) \xrightarrow{(\Bcal_{f})_\star} \CH_\star(\Bcal_Y) \overset{\Psi_Y}{\cong} \sPP_\star(Y)
\]
of $\sPP^\star(Y)$ modules such that the diagram
\begin{equation} \label{eqn:trop_pushforward_commutativity}
\begin{tikzcd}
\sPP_\star(X) \arrow[r, "(f^\textup{trop})_\star"] \arrow[d, "\Psi_X"] & \sPP_\star(Y) \arrow[d, "\Psi_Y"]\\
\CH_\star(X) \arrow[r, "f_\star"] & \CH_\star(Y)
\end{tikzcd}
\end{equation}
commutes. The same statement holds verbatim when replacing $\sPP$ by $\pPP$ and $\CH$ by $\LogCH$.
\end{proposition}

Here the $\sPP^\star(Y)$-module structure on $\sPP_\star(X)$ is induced from the $\sPP^\star(X)$-structure on $\sPP_\star(X)$ via the pullback map $\sPP^\star(Y) \to \sPP^\star(X)$.
\begin{proof}
Properness is local on the base in the fpqc topology, hence $\Bcal_f$ is proper. And similarly, being saturated is local in the log smooth topology. Then for strict piecewise polynomials and usual Chow groups, the commutativity of \eqref{eqn:trop_pushforward_commutativity} follows from the compatibility of proper pushforwards and flat pullbacks in the diagram \eqref{eqn:X_B_fiber_diagram}. The existence of a proper pushforward $(\Bcal_{f})_\star$ of logarithmic Chow groups and the corresponding commutativity follow from \cite[Construction~2.15, Theorem~2.19]{Barrott2019Logarithmic-Cho} applied to the map $\Bcal_f$.
\end{proof}

\begin{remark}
The class of morphisms $f$ as in Proposition \ref{Prop:sPP_pushforward} contains log blowups, inclusions of strata closures, and finite $G$-torsors, and is closed under compositions.
\end{remark}

\begin{proposition} \label{Pro:proper_pushforward_surjective}
In the situation of Proposition \ref{Prop:sPP_pushforward}, if $f$ is additionally surjective and of relative DM-type, then the map
\[
(\mathcal{B}_f)_\star : \sPP_\star(X) \to \sPP_\star(Y)
\]
is surjective.
\end{proposition}
\begin{proof}
Again we can use that surjectivity and being of relative DM-type is local on the base in the fpqc topology, so these properties descend from $f$ to $\mathcal{B}_f$. Then the surjectivity follows from \cite[Proposition B.19]{BaeSchmitt1}. 
\end{proof}


\begin{proposition}
Let $f: X \to Y$ be as in Proposition \ref{Prop:sPP_pushforward}, with both $X,Y$ of pure log dimension $d$. We get a commutative diagram
\[
\begin{tikzcd}
  \LogCH_i(X)  \arrow[r] & \LogCH_{i}(Y) \\
  \LogCH_{i-d}(\Bcal_X) \arrow[r]\arrow[u] & \LogCH_{i-d}(\Bcal_Y)\arrow[u]
\end{tikzcd}
\]
where the vertical maps are pullbacks and the horizontal maps are pushforwards.
\end{proposition}
\begin{proof}
This follows from \cite[Theorem~2.19]{Barrott2019Logarithmic-Cho}.
\end{proof}


Now we apply this theory to the gluing maps.
\begin{definition}
Let $\Gamma$ be a genus-decorated graph of genus $g$ with $n$ markings. Let $\Mbar_\Gamma^{\str}$ be the stack $\Mbar_\Gamma$ with log structure a pullback along the gluing map $\gl: \Mbar_\Gamma \to \Mbar_{g,n}$.
\end{definition}
\begin{lemma}
The map $\gl^\str: \Mbar_\Gamma^{\str} \to \Mbar_{g,n}$ is of relative log dimension $0$ with both spaces having log dimension $3g - 3 + n$. 
\end{lemma}

\begin{corollary}
There is a map $\sPP_\star(\Mbar_\Gamma^\str) \to \sPP_\star(\Mbar_{g,n})$ that lies over the pushforward map $\logCH_\star(\Mbar_\Gamma^\str) \to \logCH_\star(\Mbar_{g,n})$. 
\end{corollary}

While we don't need the following description, the pushforward of piecewise polynomials can be carried out explicitly within the language of polynomials on combinatorial cone stacks, using results from Brion.
\begin{definition}
Let $f:\Sigma' \to \Sigma$ be a map of cone stack, and let $\sigma \in \Sigma$ be a cone. We let $\chi_{\sigma,f}$ be the groupoid consisting of elements
\[
(\sigma' \to \sigma \isom \sigma)
\]
where $\sigma' \in \Sigma'$, and $\sigma' \to \sigma$ is a minimal factorisation of $f|_{\sigma'} \to \Sigma'$. The isomorphisms in this groupoid are given by diagrams
\[
\begin{tikzcd}
    \sigma'  \arrow[r] \arrow[d,"\sim"] & \sigma  \arrow[r,"\sim"] \arrow[d,"\sim"] & \sigma \arrow[d,equal] \\
    \sigma'' \arrow[r] & \sigma \arrow[r,"\sim"] & \sigma
\end{tikzcd}
\]
\end{definition}
\begin{remark}
If $f$ is a map of cone complexes, then $\chi_{\sigma,f}$ is the set of $\sigma' \in \Sigma'$ such that $\sigma$ is the smallest cone containing $f(\sigma')$.
\end{remark}

\begin{proposition}
\label{prop:pushforwardppstack}
Let $\psi\colon \Sigma_X \to \Sigma_Y$ be a relative dimension $0$, proper map between smooth cone stacks. Then $\phi_\star: \sPP^\star(X) \to \sPP^\star(Y)$ is uniquely determined by the fact that for any maximal cone $\sigma \in \Sigma_Y$ we have
\[
(\pi_\star f)_{\sigma} = \phi_{\sigma} \cdot \sum_{(\sigma' \to \sigma \isom \sigma) \in \chi_{\sigma,f}} \phi_{\sigma'}^{-1} \cdot f_{\sigma'}.
\]
Here $\sum_{(\sigma' \to \sigma \isom \sigma) \in \chi_{\sigma,f}}$ refers to the groupoid sum, where each term $(\sigma' \to \sigma \isom \sigma)$ is counted with weight $\#\Aut((\sigma' \to \sigma \isom \sigma))^{-1}$.
\end{proposition}
\begin{proof}
For a map of cone complexes, this is simply Brion's formula \cite[Theorem 2.3.(iii)]{Brion}. The general case follows by taking finite covers by cone complexes.
\end{proof}

We present three examples that fall outside of the scope of Brion's original formula \cite[Theorem 2.3.(iii)]{Brion}.

\begin{example}
We take $f: X \to Y$ to be $\A^2 \to [\A^2/(\Z/2\Z)]$. Let $\sigma'$ denote the unique maximal cone in $\Sigma_X$, and $\sigma$ the unique maximal cone in $\Sigma_Y$. We fix an isomorphism $\sigma' \to \sigma$. Note that $\sigma$ has a non-trivial isomorphism $\tau$. Then $\chi_{\sigma,f}$ is the set \[(\sigma' \to \sigma \xrightarrow{\id} \sigma),(\sigma' \to \sigma \xrightarrow{\tau} \sigma)\].

Then by Proposition \ref{prop:pushforwardppstack} the linear functions $x$ and $y$ on $\Sigma_X$ both get sent to $x+y$, and the quadratic function $xy$ gets sent to $2xy$.
\end{example}

\begin{example}
We take $\Sigma$ to be the cone shown in Figure \ref{fig:a2z2}, and $\Sigma' \to \Sigma$ to be the barycentric subdivision in the maximal cone. Again, with $\sigma$ the unique maximal cone in $\Sigma$, the groupoid $\chi_{\sigma,f}$ is a set of size $2$. The formula from Proposition~\ref{prop:pushforwardppstack} then gives the equality
\[
\frac{1}{xy} = \frac{1}{x(y-x)} + \frac{1}{y(x-y)}.
\]
\end{example}

\begin{example}
We take $f: \Sigma' \to \Sigma$ to be $B(\Z/2\Z) \to \{*\}$. The groupoid $\chi_{*,f}$ is $B(\Z/2\Z)$, and the constant function $1$ pushes forward to the constant function $\frac12$.
\end{example}

The following lemma shows that pushing forward homological piecewise polynomials is the same as pushing forward the corresponding piecewise polynomials.
\begin{lemma}
Let $f: \Bcal_X \to \Bcal_Y$ be a proper, saturated map of idealised Artin fans, extending to a diagram
\[
\begin{tikzcd}
    \Bcal_X \arrow[r] \arrow[d] & \Bcal_Y \arrow[d] \\ 
    \Acal_X \arrow[r] & \Acal_Y
\end{tikzcd}
\]
where the map $\Acal_X \to \Acal_Y$ is a proper, saturated map of Artin fans.
Under the identification of $\logCH_\star(\Bcal_X)$ with piecewise polynomials on $\Sigma_X$ vanishing on the boundary $\Delta_X$ as per Corollary~\ref{cor:vanishingpp}, the pushforward 
\[
\LogCH_{i-d}(\Bcal_X) \to \LogCH_{i-d}(\Bcal_Y)
\]
is given by pushing forward piecewise polynomial functions as in Proposition~\ref{prop:pushforwardppstack}.
\end{lemma}
\begin{proof}
This immediately follows from the commutative diagram
\[
\begin{tikzcd}
  \LogCH_i(\Bcal_X)  \arrow[r]\arrow[d] & \LogCH_{i}(\Bcal_Y)\arrow[d] \\
  \LogCH_{i}(\Acal_X) \arrow[r] & \LogCH_{i}(\Acal_Y)
\end{tikzcd}
\]
where all the maps are pushforwards, using the injectivity of the vertical maps and the identifications with spaces of piecewise polynomials from Corollary \ref{cor:vanishingpp}.
\end{proof}

\begin{example}
Let $X$ be the point with log structure $\N^2$, let $Y$ be $\P^2$ with toric log structure, and let $f: X \to Y$ be the strict proper map sending $X$ to the origin. Then we get the commutative diagram
\[
\begin{tikzcd}
  X  \arrow[r]\arrow[d] & \P^2 \arrow[d] \\
  \B \G_m^2 \arrow[r] & {[\P^2/\G_m^2]}
\end{tikzcd}
\]

and this induces the commutative diagram
\[
\begin{tikzcd}
  \LogCH_i(X)  \arrow[r] & \LogCH_{i}(Y) \\
  \LogCH_{i-2}(\B \G_m^2) \arrow[r]\arrow[u] & \LogCH_{i-2}[\P^2/\G_m^2]\arrow[u]
\end{tikzcd}
\]
With the identification from Corollary~\ref{cor:vanishingpp}, we have that $\sPP_\star(X)$ consists of PP functions on $\RR_{\geq 0}^2$ vanishing on the boundary, and $\sPP_\star(X)$ of PL functions on the fan $\Sigma$ of $\P^2$. The pushforward then sends a PP function $f$ on $\RR_{\geq 0}^2$ vanishing on the boundary to the PP function on $\Sigma$ that is $f$ on $\RR_{\geq 0}^2$ and $0$ everywhere else.
\end{example}


\section{A general treatment of log tautological rings}
\label{sec:logtautgeneral}
Let $(X,D)$ be a smooth Deligne-Mumford stack over a characteristic $0$ field $k$. For simplicity assume $X$ to be connected. We consider $X$ as a smooth log smooth log stack with the divisorial log structure induced by $D$. Let $X \to \mathcal{A}_X$ be an Artin fan as in Definition \ref{Def:an_Artin_fan}, such that the generic point of $\mathcal{A}_X$ has trivial automorphism group.\footnote{As an example, the canonical Artin fan $\mathcal{A}_X^\textup{can}$ always has this property.} Denote  by $\Sigma_X$ the associated cone stack. The assumptions above guarantee that it has a unique object $0 \in \Sigma_X$ (up to isomorphism)  mapping to the zero cone, and its automorphism group $\Aut(0)=\{\mathrm{id}\}$ in $\Sigma_X$ is trivial. 

Our goal here is to define notions of tautological classes, both in the Chow ring of $X$ itself, as well as for log blowups of $X$ (leading to a notion of log tautological classes). These classes are constructed by combining information from piecewise polynomials together with decorations by Chow classes defined on strata closures of $X$. Modeling our definition on the decorated strata classes of $\Mbar_{g,n}$, we require an analogue of the gluing maps $\iota_\Gamma : \Mbar_\Gamma \to \Mbar_{g,n}$ parameterizing these strata closures.

For this purpose, let $\sigma \in \Sigma_X$ be a cone and denote by $S_\sigma \subseteq X$ the associated locally closed stratum, and by $\overline S_\sigma$ its closure. The normalization $\widetilde S_\sigma \to \overline S_\sigma$ is smooth. For $G_\sigma$ the group of automorphisms of $\sigma$ in $\Sigma_X$ we claim that there is a universal principal $G_\sigma$-bundle $p_\sigma: P_\sigma \to \widetilde S_\sigma$ such that the normal bundle of the composition $P_\sigma \to \widetilde S_\sigma \to X$ splits as a sum of line bundles. It is referred to as the {\it monodromy torsor}. For $X = \Mbar_{g,n}$ and $\sigma=\sigma_\Gamma$ the cone associated to a stable graph, we have $G_{\sigma_\Gamma}=\Aut(\Gamma)$ and we exactly recover that $P_{\sigma_\Gamma} = \Mbar_\Gamma$ is the domain of the gluing map $\iota_\Gamma$.

The spaces $P_\sigma$ were explained in \cite[Section 5.1]{MPS23} and \cite[Section 6.2.1]{HMPPS} when $\Acal_X$ is chosen as the canonical Artin fan. In this case, the $G_\sigma$ can be seen as the monodromy group, acting on the branches of the divisor $D$ cutting out $S_\sigma$. Analytically locally, the normal bundle of the map $\widetilde S_\sigma \to X$ has one summand for each such branch of $D$, and the cover $P_\sigma \to \widetilde S_\sigma$ precisely ensures that the pullback of this bundle to $P_\sigma$ decomposes as a sum of line bundles.
For an arbitrary choice of Artin fan $X \to \Acal_X$ we construct $P_\sigma$ in Section \ref{Sect:Star_cone_stacks} below and verify several properties used in later proofs (see in particular Lemma \ref{Lem:Sigma_sigma_construction} and Corollary \ref{Cor:Psigma_cartesian_diagram}).





To provide some intuition and illustrate our constructions, we have the following running example  throughout Section \ref{sec:logtautgeneral}.
\begin{example}
\label{ex:a2z2}
Let $\Sigma_X$ be the combinatorial cone stack from \cite[Example~2.21]{CCUW}, given by\footnote{Note that the diagram below omits the identity morphisms of each object, and the unique morphism $0 \to \sigma$.}
\begin{equation}
\label{eqn:cone_stack_a2z2}
    \begin{tikzcd}
    0 \arrow[r] & \rho \arrow[r, shift left = 1] \arrow[r, shift right = 1] & \sigma \arrow[loop right]{r}
    \end{tikzcd}
\end{equation}
where $0$ is the cone point, $\rho$ is a ray, $\sigma$ is $\RR_{\geq 0}^2$. The two maps $\rho \to \sigma$ are the two inclusions of $\rho$ as a ray of $\sigma$, and the non-trivial automorphism on $\sigma$ is the swap $(x,y) \mapsto (y,x)$.
A more geometric visualization can be found in Figure~\ref{fig:a2z2}.
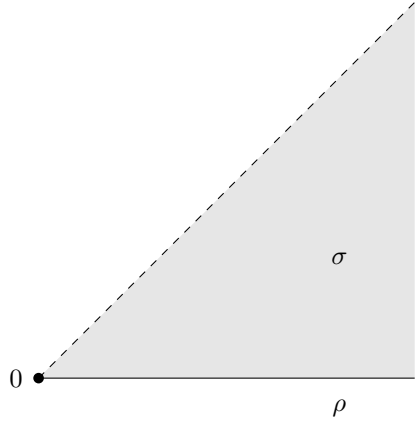
\begin{figure}
\begin{center}
\begin{tikzpicture}
    \fill [gray!20] (0,0) -- (5,5) -- (5,0) -- (0,0);
    \draw (0,0) -- (5,0);
    \fill (0,0) node {} circle (2pt);
    \node at (4,1.6) {$\sigma$};
    \node at (4,-.4) {$\rho$};
    \node at (-.3,0) {$0$};
    \draw[dashed] (0,0) -- (5,5);
\end{tikzpicture}
\end{center}
\caption{The fan $\Sigma_X$ defined in Example~\ref{ex:a2z2}}
\label{fig:a2z2}
\end{figure}

This is similar to the $\Z/2\Z$-quotient of $\RR_{\geq 0}^2$ except with trivial stabilizer on the cone point. There are multiple stacks which have $\Sigma_X$ as cone stack. A trivial example is the Artin fan $\Acal_{\Sigma_X}$ itself. For an example in  algebraic log spaces, we can take the colimit of the diagram
\[
\begin{tikzcd}
    \A^1 \times \G_m \arrow[r, shift left = 1] \arrow[r, shift right = 1] & \A^2 \arrow[loop right]{r}
    \end{tikzcd}
\]
where the map $\A^2 \to \A^2$ sends $(x,y)$ to $(y,x)$, and the two maps $\A^1 \times \G_m \to \A^2$ are $(x,y) \mapsto (x,y)$ and $(x,y) \mapsto (y,x)$. Here with the colimit we mean the universal algebraic log space $Y$ with maps $\A^1 \times \G_m \to Y$ and $\A^2 \to Y$ such that the diagram
\[
\begin{tikzcd}
    \A^1 \times \G_m \arrow[rr, shift left = 1] \arrow[dr] \arrow[rr, shift right = 1] & & \A^2 \arrow[ld] \arrow[loop right]{r} \\
    & Y &
    \end{tikzcd}
\]
commutes.
A final example which we will study further is to take $X = \A^2 \times \G_m$, with divisorial log structure induced by the divisor 
$$D = \{(x,y,z) \in \A^2 \times \G_m : x^2 - y^2z = 0\} \subset \A^2 \times \G_m\,.$$
This example is also known as the \emph{punctured Whitney umbrella}, and is also treated in detail in \cite[Example 3.3.1]{ACMW} and \cite[Example 5.4.1]{ACMUW}.
The singular locus of $D$ is the $z$-axis $\{(0,0)\} \times \G_m \subset X$, which we denote by $D^{(2)}$. We have $\ol{S}_{\sigma} = \tilde{S}_{\sigma} = D^{(2)}$, and $P_{\sigma}$ = $\G_m$ with the map to $\ol{S}_\sigma$ given by
\[
P_{\sigma} = \G_m \to \G_m = \ol{S}_{\sigma}, u \mapsto z=u^2\,,
\]
corresponding to the automorphism group of $\sigma$ having order $2$. For the ray $\rho$ we have $\ol{S}_{\rho} = D$, and 
\[
P_\rho = \tilde{S}_{\rho} \cong \{(x,y,(u:t)) \in \A^2 \times \P^1 : xt = yu\}
\]
given by the normalization of $D$ (which is the strict transform of $D$ in the blowup of $X$ at $D^{(2)}$). The scheme $P_S$ has generic log structure of rank $1$ and log structure of rank $2$ along $0 \times 0\times \P^1$.

\end{example}

\subsection{Preliminaries on cone stacks and monodromy torsors}
In the following subsections we start with a purely combinatorial construction of star cone stacks, giving the tropical analogue of the monodromy torsors above (Section \ref{Sect:Star_cone_stacks}). Via the equivalence of cone stacks and Artin fans, we use this to define the monodromy torsors (Section \ref{Sect:Monodromy_torsors}). Finally we discuss star subdivisions, which are needed when analyzing the effect of log blowups on the Chow groups of the monodromy torsors above (Section \ref{Sect:Star_subdivisions}).

\subsubsection{Star cone stacks} \label{Sect:Star_cone_stacks}
\begin{definition}
\label{def:starconestack}
Given a cone stack $\Sigma$ and $\sigma \in \Sigma$, the \emph{star cone stack} $$C : \mathsf{Star}_\sigma(\Sigma) \to \mathbf{RPC}^f$$ 
has as objects diagrams $(\sigma \xrightarrow{j'} \sigma' \xleftarrow{j''}\sigma'')$ of morphisms in $\Sigma$ such that $\sigma'$ is minimal among objects in $\Sigma$ receiving maps $j', j''$ from $\sigma, \sigma''$.
The functor $C$ sends this object to
\begin{equation}
    C(\sigma \xrightarrow{j'} \sigma' \xleftarrow{j''}\sigma'') = \sigma'' \in \mathbf{RPC}^f\,.
\end{equation}
The morphisms in $\mathsf{Star}_\sigma(\Sigma)$ are defined by commuting diagrams
\begin{equation} \label{eqn:Star_morphism}
    \mathsf{Mor}(\sigma \xrightarrow{j_1'} \sigma_1' \xleftarrow{j_1''}\sigma_1'', \sigma \xrightarrow{j_2'} \sigma_2' \xleftarrow{j_2''}\sigma_2'') = \ \ 
    \begin{tikzcd}
    & \sigma_1' \arrow[dd,"\varphi'"] & \sigma_1'' \arrow[l,"j_1''",swap] \arrow[dd,"\varphi''"]\\
    \sigma \arrow[ru, "j_1'"] \arrow[rd, "j_2'", swap] & & \\
    & \sigma_2' & \sigma_2'' \arrow[l,"j_2''"]
    \end{tikzcd}
\end{equation}
The functor $C$ sends the morphism \eqref{eqn:Star_morphism} to $\varphi'': \sigma_1'' \to \sigma_2''$.

We define the interior $\mathsf{Star}_\sigma(\Sigma)^0$ to be the full subcategory of objects $(\sigma \xrightarrow{j'} \sigma' \xleftarrow{j''} \sigma'') \in \mathsf{Star}_\sigma(\Sigma)$ such that the map $j''$ is an isomorphism, and we let $\Delta_{\sigma, \Sigma}$ be its complement.
\end{definition}
\begin{lemma}
The tuple $(\mathsf{Star}_\sigma(\Sigma), \mathsf{Star}_\sigma(\Sigma)^0, \Delta_{\sigma, \Sigma})$ is a cone stack with boundary.
\end{lemma}
\begin{proof}
The only thing to check is that $\mathsf{Star}_\sigma(\Sigma)^0$ is forward-closed. Thus assume that we have a morphism
\[
\varphi: (\sigma \xrightarrow{j_1'} \sigma_1' \xleftarrow{j_1''}\sigma_1'') \to (\sigma \xrightarrow{j_2'} \sigma_2' \xleftarrow{j_2''}\sigma_2'')
\]
in $\mathsf{Star}_\sigma(\Sigma)$ such that $j_1''$ is an isomorphism. This gives a solid diagram as follows:
\begin{equation}
\begin{tikzcd}
    & \sigma_1' \arrow[d,"\varphi'"] & \sigma_1'' \arrow[l,"j_1''",swap] \arrow[d,"\varphi''"]\\
    \sigma \arrow[ru, "j_1'"] \arrow[r, "j_2'"] \arrow[rd, dashed, "\varphi'' \circ (j_1'')^{-1} \circ j_1'", swap]& \sigma_2' & \sigma_2'' \arrow[l,"j_2''", swap] \arrow[dl, dashed, "\mathsf{id}"]\\
    & \sigma_2'' \arrow[u, dashed, "j_2''", swap] & 
\end{tikzcd}
\end{equation}
But we immediately check that the dashed diagram commutes with the maps in the solid diagram. Therefore, the assumption that $\sigma_2'$ is minimal among objects receiving maps from $\sigma, \sigma_2''$ implies that $j_2''$ is an isomorphism and hence $(\sigma \xrightarrow{} \sigma_2' \xleftarrow{}\sigma_2'') \in \mathsf{Star}_\sigma(\Sigma)^0$ as desired.
\end{proof}

\begin{example}[continues Example \ref{ex:a2z2}]
We compute that $\mathsf{Star}_\rho(\Sigma)$ is given by the following diagram.
\[
\begin{tikzcd}
& (\rho \xrightarrow{\mathsf{id}} \rho \xleftarrow{\mathsf{id}} \rho) \arrow[rd, "{(\iota_1, \iota_1)}"] & \\
(\rho \xrightarrow{\mathsf{id}} \rho \leftarrow \bullet) \arrow[ru, "{(\mathsf{id}, \bullet \to \rho)}"] \arrow[rd, "{(\iota_1, \bullet \to \rho)}",swap] & & (\rho \xrightarrow{\iota_1} \sigma \xleftarrow{\mathsf{id}} \sigma)\\
& (\rho \xrightarrow{\iota_1} \sigma \xleftarrow{\iota_2} \rho) \arrow[ru, "{(\mathsf{id}, \iota_2)}",swap] & 
\end{tikzcd}
\]
Here the boundary is given by the two objects on the bottom left. Similarly, the star $\mathsf{Star}_\sigma(\Sigma)$ is given by  \[
\begin{tikzcd}
& (\sigma \xrightarrow{\mathsf{id}} \sigma \xleftarrow{\iota_1} \rho) \arrow[rd, "{(\id, \iota_1)}"] & \\
(\sigma \xrightarrow{\mathsf{id}} \sigma \leftarrow \bullet) \arrow[ru, "{(\mathsf{id}, \bullet \to \rho)}"] \arrow[rd, "{(\id, \bullet \to \rho)}",swap] & & (\sigma \xrightarrow{\mathsf{id}} \sigma \xleftarrow{\mathsf{id}} \sigma)\\
& (\sigma \xrightarrow{\mathsf{id}} \sigma \xleftarrow{\iota_2} \rho) \arrow[ru, "{(\mathsf{id}, \iota_2)}",swap] & 
\end{tikzcd}
\]
with the boundary being everything except $(\sigma \xrightarrow{\mathsf{id}} \sigma \xleftarrow{\mathsf{id}} \sigma)$. Looking at this diagram, we see that $\mathsf{Star}_\sigma(\Sigma) \cong \mathsf{Faces}(\sigma) = \mathsf{Faces}(\RR_{\geq 0}^2)$.

\end{example}

If the unique object $0 \in \Sigma$ mapping to the origin of the cones in $\Sigma$ has no non-trivial automorphisms, we have a natural identification $\mathsf{Star}_0(\Sigma) = \mathsf{Star}_0(\Sigma)^0 \cong \Sigma$. 
See also \cite[Section 3.3.1]{MR21} for more applications of the star construction.

Given a morphism $\psi: \tau \to \sigma$ in $\Sigma$ there is an induced map
\begin{align}
\mathsf{Star}_\psi : \mathsf{Star}_\sigma(\Sigma) &\to \mathsf{Star}_\tau(\Sigma)\\
(\sigma \xrightarrow{j'} \sigma' \xleftarrow{j''} \sigma'') & \mapsto (\tau \xrightarrow{j' \circ \psi} \widetilde{\sigma}' \xleftarrow{\widetilde{j}''} \sigma'')\,, \nonumber
\end{align}
where $\widetilde{\sigma}' \to \sigma'$ is the unique minimal face morphism in $\Sigma$  through which the maps $j' \circ \psi$ and $j''$ factor. Its existence and uniqueness follow the properties of cone stacks in \cite[Definition 2.15]{CCUW}. We check that $\mathsf{Star}_\psi$ sends $\mathsf{Star}_\sigma(\Sigma)^0$ to $\mathsf{Star}_\tau(\Sigma)^0$ and thus defines a map of cone stacks with boundary.

This construction is functorial (in the sense that $\mathsf{Star}_\varphi \circ \mathsf{Star}_\psi = \mathsf{Star}_{\psi \circ \varphi}$), and hence when $\psi: \sigma \to \sigma$ runs through the automorphisms of $\sigma$, we obtain a compatible system of automorphisms $\mathsf{Star}_\psi$ of $\mathsf{Star}_\sigma(\Sigma)$ and a quotient cone stack\footnote{The \emph{quotient cone stack} $\Sigma/G$ for a finite group $G$ acting on $\Sigma$ is the cone stack with  objects given by the cones of $\Sigma$, and morphisms given by the compositions $g \circ \psi$ of morphisms $\psi$ in $\Sigma$ and elements $g \in G$.} (with boundary) 
\begin{equation} \label{eqn:q_sigma}
q_\sigma: \mathsf{Star}_\sigma(\Sigma) \to \mathsf{Star}_\sigma(\Sigma) / \mathsf{Aut}(\sigma)\,.
\end{equation}


Another important functoriality is given by a map from $\mathsf{Star}_\sigma(\Sigma)$ to $\sigma$ itself in the case when $\Sigma$ is smooth. To set this up, recall the cone stack $\mathsf{Faces}(\sigma)$ from Example \ref{Exa:Faces} obtained from the rational polyhedral cone $\sigma$ and its faces $\sigma_0 \prec \sigma$ with interior given by the full subcategory consisting only of $\sigma$ itself.

To construct a map $\mathsf{Star}_\sigma(\Sigma) \to \mathsf{Faces}(\sigma)$ we observe that for each morphism $j:\sigma \to \sigma'$ in $\Sigma$ there exists a unique map of the underlying cones $\pi_j : \sigma' \to \sigma$ which is the projection from $\sigma'$ onto its face $\sigma$.\footnote{Here we stress that $\pi_j$ is in general \emph{not} a morphism of the original cone stack $\Sigma$; it is just a map of rational polyhedral cones!} Clearly we have $\pi_j \circ j = \mathrm{id}_\sigma$.
\begin{proposition} \label{Pro:pi_sigma}
There is a morphism of cone stacks with boundary 
\begin{equation}
\pi_\sigma^\textup{trop} : (\mathsf{Star}_\sigma(\Sigma), \mathsf{Star}_\sigma(\Sigma)^0, \Delta_{\sigma, \Sigma}) \to (\mathsf{Faces}(\sigma), \sigma^0, \Delta_\sigma)
\end{equation}
sending an object $(\sigma \xrightarrow{j'} \sigma' \xleftarrow{j''} \sigma'')$ to the face inclusion $(\pi_{j'} \circ j'')(\sigma'') \preceq \sigma$, whose underlying map 
\[
C(\sigma \xrightarrow{j'} \sigma' \xleftarrow{j''} \sigma'') = \sigma'' \xrightarrow{\pi_{j'} \circ j''} (\pi_{j'} \circ j'')(\sigma'')
\]
on cones is likewise given by $\pi_{j'} \circ j''$. Moreover, the morphism satisfies
\begin{equation} \label{eqn:pi_sigma_trop_preimage}
(\pi_\sigma^\textup{trop})^{-1}(\sigma^0) = \mathsf{Star}_\sigma(\Sigma)^0\,.
\end{equation}
\end{proposition}
\begin{proof}
The necessary compatibility checks to show that $\pi_\sigma^\textup{trop}$ is a morphism of cone stacks are a short chase in the diagram \eqref{eqn:Star_morphism}. This uses that adding the projection morphisms $\pi_{j_1'}$ and $\pi_{j_2'}$ to the diagram, we still have the relevant maps in the diagram commute.

For proving \eqref{eqn:pi_sigma_trop_preimage} we have to show that given an object $(\sigma \xrightarrow{j'} \sigma' \xleftarrow{j''} \sigma'')$  of $\mathsf{Star}_\sigma(\Sigma)$ with $\pi_j : \sigma' \to \sigma$ the projection onto the face given by $j'$, we have that the composition $\pi_{j'} \circ j''$ surjects onto $\sigma$ if and only if $j''$ is an isomorphism (the two respective criteria of being in the interior). Clearly it is true that $j''$ being an isomorphism implies the surjectivity since the projection $\pi_{j'}$ is surjective. 
Conversely, if $j''$ is a strict face inclusion, then by the minimality of $\sigma'$ in the diagram $(\sigma \xrightarrow{j'} \sigma' \xleftarrow{j''} \sigma'')$ there must be a ray of $\sigma$ whose image under $j'$ is not contained in the image of $j''$. But then this ray cannot be contained in the image of the composition $\pi_{j'} \circ j''$ and thus this composition is not surjective.
\end{proof}

A final result that is necessary later is a compatibility of the star cone stack construction with morphisms of the ambient cone stacks.

\begin{lemma}  \label{Lem:iota_sigmahat_section}
Let $\varphi: \widehat \Sigma \to \Sigma$ be a morphism of cone stacks, and for a choice of $\widehat \sigma \in \widehat \Sigma$ denote by $\sigma = \varphi(\widehat \sigma)$ the image cone of $\widehat \sigma$ in $\Sigma$. Then there exists a well-defined map of cone-stacks
\begin{equation} \label{eqn:t_section_definition}
    t_{\widehat \sigma \to \sigma} : \mathsf{Star}_{\widehat \sigma}(\widehat \Sigma) \to \mathsf{Star}_{\sigma}(\Sigma), (\widehat \sigma \to \widehat \sigma' \leftarrow \widehat \sigma'') \mapsto (\sigma \to \sigma' = \varphi( \widehat \sigma') \leftarrow \sigma'' = \varphi( \widehat \sigma'') )
\end{equation}
commuting with the maps $\mathsf{Star}_{\sigma}(\Sigma) \to \Sigma$ and $\mathsf{Star}_{\widehat \sigma}(\widehat \Sigma) \to \widehat \Sigma \xrightarrow{\varphi} \Sigma$. Moreover, the construction is functorial: for a morphism $\widehat \sigma \to \widehat \sigma'$ in $\widehat \Sigma$ mapping to $\sigma \to \sigma'$ under $\varphi$, we have 
\begin{equation} \label{eqn:t_functoriality}
\mathsf{Star}_{\sigma \to \sigma'} \circ t_{\widehat \sigma' \to \sigma'} = t_{\widehat \sigma \to \sigma} \circ \mathsf{Star}_{\widehat \sigma \to \widehat \sigma'}\,.
\end{equation}
\end{lemma}
\begin{proof}
The central point in checking that $t_{\widehat \sigma \to \sigma}$ is well-defined is to verify that $(\sigma \to \sigma' \leftarrow \sigma'')$ is an object of $\mathsf{Star}_{\sigma}(\Sigma)$, i.e. that in this diagram, the cone $\sigma'$ is minimal among cones of $\Sigma$ receiving maps from $\sigma, \sigma''$. This follows very easily from the corresponding minimality property of $\widehat \sigma'$ and the fact that $\widehat \sigma' \xrightarrow{\varphi} \sigma'$ does not factor through a proper face of $\sigma'$ (by definition of $\varphi$ being a morphism of cone stacks). Thus equation \eqref{eqn:t_section_definition} gives a well-defined map on objects, whose associated map on cones is just $\widehat \sigma'' \xrightarrow{\varphi} \sigma''$. Filling in the remaining data and verifying the compatibilities is straightforward, just like the fact that $t_{\widehat \sigma \to \sigma}$ commutes with the maps of its domain and target to $\Sigma$. 
Finally, we check the functoriality \eqref{eqn:t_functoriality} by observing that both maps act on the cones of $\widehat \Sigma_{\widehat \sigma'}$ as
\[
(\widehat \sigma' \to \widehat \sigma'' \leftarrow \widehat \sigma''') \mapsto (\sigma \to \varphi(\widehat \sigma'') \leftarrow \varphi(\widehat \sigma'''))\,. \qedhere
\]
\end{proof}

\subsubsection{Monodromy torsors} \label{Sect:Monodromy_torsors}
In the following we use the star cone stacks above to define the monodromy torsors mentioned in the introduction of Section \ref{sec:logtautgeneral}. For this recall that by \cite[Theorem 3]{CCUW} the category of cone stacks is equivalent to the category of Artin fans. Given a cone stack $\Sigma$ with associated Artin fan $\mathcal{A}_\Sigma$ and a cone $\sigma \in \Sigma$, we write
\begin{itemize}
    \item $\mathcal{S}_\sigma \subseteq \mathcal{A}_\Sigma$ for the associated locally closed stratum,
    \item $\overline{\mathcal{S}}_\sigma$ for its closure in $\mathcal{A}_\Sigma$, and
    \item $\widetilde{\mathcal{S}}_\sigma$ for the normalization of $\overline{\mathcal{S}}_\sigma$.
\end{itemize}
Let $\Sigma$ be a smooth cone such that $0 \in \Sigma$ has trivial automorphism group and let $\sigma \in \Sigma$ be an object of $\Sigma$. Then the natural map $\mathsf{Star}_{0 \to \sigma} : \mathsf{Star}_\sigma(\Sigma) \to \mathsf{Star}_0(\Sigma) \cong \Sigma$ factors through the quotient \eqref{eqn:q_sigma}. For the morphisms
\[
\mathsf{Star}_\sigma(\Sigma) \to \mathsf{Star}_\sigma(\Sigma) / \mathsf{Aut}(\sigma) \to \Sigma
\]
of cone stacks with boundary, we obtain associated morphisms of the closed substacks $\mathcal{B}$ of their Artin fans
\begin{equation} \label{eqn:BStarsigma_identification}
\mathcal{B}_{\mathsf{Star}_\sigma(\Sigma)^0} \to \mathcal{B}_{\mathsf{Star}_\sigma(\Sigma)^0/\mathsf{Aut}(\sigma)} \to \mathcal{B}_\Sigma = \mathcal{A}_\Sigma\,.
\end{equation}
Moreover, since $\Aut(\sigma)$ acts on $\mathsf{Star}_\sigma(\Sigma)^0$, we obtain an induced action of $\Aut(\sigma)$ on $\mathcal{B}_{\mathsf{Star}_\sigma(\Sigma)^0}$ by functoriality.
\begin{lemma} \label{Lem:Sigma_sigma_construction}
For a smooth cone stack $\Sigma$ with trivial automorphism group of the zero cone $0 \in \Sigma$ and for a choice of $\sigma \in \Sigma$, there is a canonical isomorphism $$\mathcal{B}_{\mathsf{Star}_\sigma(\Sigma)^0/\mathsf{Aut}(\sigma)} \cong \widetilde{\mathcal{S}}_\sigma\,.$$
Moreover, the action of $\Aut(\sigma)$ on $\mathcal{B}_{\mathsf{Star}_\sigma(\Sigma)^0}$ makes it a principal $\Aut(\sigma)$-bundle over $\widetilde{\mathcal{S}}_\sigma$.
\end{lemma}

\begin{definition} \label{Def:stacky_monodromy_torsor}
We call $\mathcal{P}_\sigma \coloneqq \mathcal{B}_{\mathsf{Star}_\sigma(\Sigma)^0} \to \widetilde{\mathcal{S}}_\sigma$ the \emph{stacky monodromy torsor} associated to $\sigma \in \Sigma$ and write $\iota_\sigma : \mathcal{P}_\sigma \to \mathcal{A}_\Sigma$ for the induced map to the Artin fan of $\Sigma$.
\end{definition}
\begin{example}[continues Example \ref{ex:a2z2}]
The cone stack $\Sigma_X$ from \eqref{eqn:cone_stack_a2z2} has an Artin fan with a locally closed decomposition
\[
\mathcal{A}_{\Sigma_X} = \underbrace{\mathsf{pt}}_{\mathcal S_0} \sqcup \underbrace{B \mathbb{G}_m}_{\mathcal S_\rho} \sqcup \underbrace{B \mathbb{G}_m^2 \rtimes (\mathbb{Z}/2\mathbb{Z})}_{\mathcal S_\sigma}\,.
\]
On the other hand, we saw that $\mathsf{Star}_\sigma(\Sigma)^0 \cong \mathsf{Faces}(\sigma)$ so that 
\[
\Bcal_{\mathsf{Star}_\sigma(\Sigma)^0} = B \mathbb{G}_m^2 \subseteq \Acal_{\mathsf{Star}_\sigma(\Sigma)^0} = [\mathbb{A}^2 / \mathbb{G}_m^2]\,.
\]
Correspondingly, the stacky monodromy torsor $\mathcal{P}_\sigma \to \widetilde{\mathcal{S}}_\sigma$ is given by $B \mathbb{G}_m^2 \to B \mathbb{G}_m^2 \rtimes (\mathbb{Z}/2\mathbb{Z})$.
\end{example}

The proof of Lemma \ref{Lem:Sigma_sigma_construction} proceeds by reduction to the case of toric varieties. We state and prove that case separately.
\begin{lemma} \label{Lem:Sigma_sigma_construction_toric}
Let $\Sigma$ be a smooth fan with toric variety $X_\Sigma$, $\sigma \in \Sigma$ and consider the strata closure $\overline{S}_\sigma \subseteq X_\Sigma$. Then the map of cone stacks with boundary associated to $\overline{S}_\sigma \subseteq X_\Sigma$ is given by $\mathsf{Star}_\sigma(\Sigma) \to \Sigma$. In particular, there is a fiber diagram
\[
\begin{tikzcd}
\overline{S}_\sigma \arrow[r,hook] \arrow[d] & X_\Sigma \arrow[d]\\
\mathfrak{B}_{\mathsf{Star}_\sigma(\Sigma)^0} \arrow[r, hook] & \mathcal{A}_\Sigma
\end{tikzcd}\,.
\]
\end{lemma}
\begin{proof}
It is immediate to check that the map $\mathsf{Star}_\sigma(\Sigma)^0 \to \Sigma$ is a fully faithful embedding, using that the cone stack (associated to) $\Sigma$ has at most one morphism between each of its objects. This implies that $\mathfrak{B}_{\mathsf{Star}_\sigma(\Sigma)^0}$ is the unique closed reduced substack of $\mathcal{A}_\Sigma$ whose points correspond to the cones of $\Sigma$ containing $\sigma$ as a face. By the toric Orbit-Cone correspondence (\cite[Theorem 3.2.6]{CLS11}) the preimage of that substack in $X$ is precisely the orbit closure $\overline{S}_\sigma$.
\end{proof}

\begin{proof}[Proof of Lemma \ref{Lem:Sigma_sigma_construction}]
By construction, the Artin fan $\mathcal{A}_\Sigma$ is a colimit
\begin{equation}
    \mathcal{A}_\Sigma = \varinjlim_{\sigma_0 \in \Sigma} [\underbrace{\mathrm{Spec}\ k[\sigma_0^\vee]}_{V_{\sigma_0}} / \underbrace{\mathrm{Spec}\ k[(\sigma_0^\vee)^\textup{gp}]}_{T_{\sigma_0}}]
\end{equation}
of Artin cones $[V_{\sigma_0}/T_{\sigma_0}]$ where $V_{\sigma_0}$ is the affine toric variety with torus $T_{\sigma_0}$ associated to the cone $\sigma_0$. Thus we obtain a strict smooth cover $\coprod_{\sigma_0} [V_{\sigma_0}/T_{\sigma_0}] \to \mathcal{A}_\Sigma$ by quotient stacks of affine toric varieties. We prove the claimed isomorphisms after pullback to one of the varieties $[V_{\sigma_0}/T_{\sigma_0}]$ of the cover.

To calculate this pullback, note that for the cone stack $\mathsf{Faces}(\sigma_0)$ associated to the cone $\sigma_0$ we have an isomorphism $[V_{\sigma_0}/T_{\sigma_0}] \cong \mathcal{A}_{\mathsf{Faces}(\sigma_0)}$. 
By the equivalence of $2$-categories between cone stacks and Artin fans (\cite[Theorem 6.11]{CCUW}), we can calculate the fiber product as the Artin fan
\begin{equation} \label{eqn:AStar_fiber_product}
    \mathcal{A}_{\mathsf{Star}_\sigma(\Sigma)} \times_{\mathcal{A}_\Sigma} \mathcal{A}_{\mathsf{Faces}(\sigma_0)} = \mathcal{A}_{\mathsf{Star}_\sigma(\Sigma) \times_\Sigma \mathsf{Faces}(\sigma_0) }
\end{equation}
associated to the fiber product of $\mathsf{Star}_\sigma(\Sigma)$ and $\mathsf{Faces}(\sigma_0)$ over $\Sigma$. 
Then the pullbacks of the stacks $\mathcal{B}_{\mathsf{Star}_\sigma(\Sigma)^0}$ and $\mathcal{B}_{\mathsf{Star}_\sigma(\Sigma)^0/\mathsf{Aut}(\sigma)}$ in \eqref{eqn:BStarsigma_identification} under $\mathcal{A}_{\mathsf{Faces}(\sigma_0)} \to \mathcal{A}_\Sigma$ are closed substacks of \eqref{eqn:AStar_fiber_product} associated to the pullback of the interior $\mathsf{Star}_\sigma(\Sigma)^0$  of $\mathsf{Star}_\sigma(\Sigma)$. We conclude by showing that they are canonically isomorphic to the pullbacks of $\mathcal{P}_\sigma$ and $\widetilde{\mathcal{S}}_\sigma$, respectively.

Spelling out the definition of the fiber product $\mathsf{Star}_\sigma(\Sigma)^0 \times_\Sigma \mathsf{Faces}(\sigma_0)$ of cone stacks, its objects are given by triples
\[
\left((\sigma \xrightarrow{j'} \sigma' \xleftarrow{j''} \sigma''), \tau \preceq \sigma_0, \sigma'' \cong \tau \right)
\]
of an object in $\mathsf{Star}_\sigma(\Sigma)^0$ (which thus satisfies that $j''$ is an isomorphism), a face $\tau$ of $\sigma_0$ and an isomorphism of their associated objects in $\Sigma$. Using $j''$ to identify $\sigma'$ and $\sigma''$ and applying the provided isomorphism of $\sigma''$ and $\tau$, this is just equivalent (up to unique isomorphism) to the data $(\sigma \to \sigma' \preceq \sigma_0)$ of a morphism $\sigma \to \sigma'$ in $\Sigma$ and an identification of $\sigma'$ as a face of $\sigma_0$. 

Before returning to geometry, we need to classify these objects a bit further. Note that we can separate them into a disjoint union depending on the image $\mathsf{im}(\sigma) = \tau_\sigma \subseteq \sigma_0$ of the morphism $\sigma \to \sigma_0$. Any two morphisms $\sigma \to \sigma_0$ with the same image $\tau_\sigma$ are uniquely related by an automorphism of $\sigma$, and thus $\mathsf{Aut}(\sigma)$ acts freely on the objects of $\mathsf{Star}_\sigma(\Sigma)^0 \times_\Sigma \mathsf{Faces}(\sigma_0)$. 
Dividing out by $\mathsf{Aut}(\sigma)$, once we fix a choice of face $\tau_\sigma$ that can be the image of a morphism $\sigma \to \sigma_0$, the intermediate face $\sigma'$ is just any element of the (classical) star of $\tau_\sigma$ in $\sigma_0$. Thus we have
\[
\mathsf{Star}_\sigma(\Sigma)^0 \times_\Sigma \sigma_0 / \mathsf{Aut}(\sigma) = \bigsqcup_{\tau_\sigma \subseteq \sigma_0} \mathsf{Star}_{\tau_\sigma}(\sigma_0)^0
\]
To finish the proof, note that by Lemma \ref{Lem:Sigma_sigma_construction_toric} the toric variety associated to $\mathsf{Star}_{\tau_\sigma}(\sigma_0)^0$ is precisely the (normalization of the) strata closure of the stratum of $V_{\sigma_0}$ associated to $\tau_\sigma \subseteq \sigma_0$. Since taking the normalization is compatible under smooth base change, this descends to the quotient stack $\mathcal{A}_{\mathsf{Faces}(\sigma_0)}$, and we verify that the stacks $\mathcal{B}_{\mathsf{Star}_{\tau_\sigma}(\sigma_0)}$ indeed glue together to the pullback of the strata closure normalization $\widetilde{\mathcal{S}}_\sigma$. 

Finally, since $\mathsf{Star}_\sigma(\Sigma) \to \mathsf{Star}_\sigma(\Sigma)/\Aut(\sigma)$ is clearly an $\Aut(\sigma)$-quotient, the equivalence of categories between cone stacks and Artin fans implies that the associated map $\mathcal{P}_\sigma \to \widetilde{\mathcal{S}}_\sigma$ is a principal $\Aut(\sigma)$-bundle, as desired.
\end{proof}
An important property of the monodromy torsors considered in \cite{MPS23, HMPPS} was that their normal bundles (for the natural map to the ambient stack) decomposes into a sum of line bundles. Using the setup above, we can make this property very explicit for the stacks $\mathcal{P}_\sigma$, using the projection morphism $\pi_\sigma^\textup{trop}$ from Proposition \ref{Pro:pi_sigma}. For this observe that there is a natural isomorphism
\[
\mathcal{B}_{\mathsf{Faces}(\sigma)^\circ} = (B \mathbb{G}_m)^{\sigma(1)} = \prod_{\rho \in \sigma(1)} B \mathbb{G}_m\,.
\]
For $\rho \in \sigma(1)$ let $\mathcal{L}_\rho$ be the universal line bundle on $\mathcal{B}_{\mathsf{Faces}(\sigma)^0}$ associated to the factor $B \mathbb{G}_m$ corresponding to $\rho$. 
Then the normal bundle of the inclusion
\[
i : \mathcal{B}_{\mathsf{Faces}(\sigma)^\circ} \to \mathcal{A}_{\mathsf{Faces}(\sigma)}
\]
is given by the direct sum of $\mathcal{L}_\rho$ for $\rho \in \sigma(1)$.
\begin{proposition}  \label{Prop:B_Faces}
For a smooth cone stack $\Sigma$ such that $0 \in \Sigma$ has trivial automorphism group and $\sigma \in \Sigma$ consider the diagram
\begin{equation}
\begin{tikzcd}
~&  \mathcal{B}_{\mathsf{Star}_\sigma(\Sigma)^0} \arrow[r, "i'"] \arrow[rr, bend left, "\iota_\sigma"] \arrow[l, phantom, "\mathcal{P}_\sigma ="] \arrow[d, "\pi_\sigma"] &  \mathcal{A}_{\mathsf{Star}_\sigma(\Sigma)} \arrow[d, "\widetilde \pi_\sigma"] \arrow[r, "q_\sigma"] & \mathcal{A}_\Sigma\\
& \mathcal{B}_{\mathsf{Faces}(\sigma)^\circ} \arrow[r, "i"] & \mathcal{A}_{\mathsf{Faces}(\sigma)} & 
\end{tikzcd}
\end{equation}
where 
\begin{itemize}
    \item $i, i'$ are the closed embeddings associated to the respective cone stacks with boundary,
    \item $q_\sigma$ and the composition $\iota_\sigma = q_\sigma \circ i'$ are induced from the map $\mathsf{Star}_{0 \to \sigma}$,
    \item $\pi_\sigma, \widetilde \pi_\sigma$ are induced from $\pi_\sigma^\textup{trop}$.
\end{itemize}
Then $q_\sigma$ and $\widetilde{\pi_{\sigma}}$ are \'etale, 
and the square diagram is cartesian. In particular, the normal bundle of $\iota_\sigma$ splits as
\begin{equation}
    \mathcal{N}_{\mathcal{P}_\sigma/\mathcal{A}_\Sigma} = \bigoplus_{\rho \in \sigma(1)} \pi_\sigma^\star \mathcal{L}_\rho\,,
\end{equation}
and $c_1(\pi_\sigma^\star \mathcal{L}_\rho) = \Phi((\pi_\sigma^\textup{trop})^\star x_\rho)$, with $x_\rho$ the (piecewise) linear function on $\mathsf{Faces}(\sigma)$ associated to the ray $\rho$ of $\sigma$.

Moreover, for $F_\sigma = \prod_{\rho \in \sigma(1)} (\pi_\sigma^\textup{trop})^\star x_\rho$ we have $F_\sigma \in \sPP_\star(\mathcal{P}_\sigma)$ and $\Psi(F_\sigma) = [\mathcal{P}_\sigma] \in \CH_{-\dim(\sigma)}(\mathcal{P}_\sigma)$. In fact 
\begin{equation} \label{eqn:sPP_hom_free_module}
    \sPP_\star(\mathcal{P}_\sigma) = \sPP^\star(\mathcal{P}_\sigma) \cdot F_\sigma
\end{equation}
is a free $\sPP^\star(\mathcal{P}_\sigma)$-module with generator $F_\sigma$. 
\end{proposition}
\begin{proof}
To see that $q_\sigma$ is \'etale, note that this property can be checked \'etale locally on the source (see \cite[Tag \href{https://stacks.math.columbia.edu/tag/036W}{036W}]{stacks-project}). The domain $\mathcal{A}_{\mathsf{Star}_\sigma(\Sigma)}$ of $q_\sigma$ has an \'etale cover by Artin cones $\mathcal{A}_{\sigma''}$ for $(\sigma \to \sigma' \leftarrow \sigma'') \in \mathsf{Star}_\sigma(\Sigma)$. But the composition $\Acal_{\sigma''} \to \mathcal{A}_{\mathsf{Star}_\sigma(\Sigma)} \xrightarrow{q_\sigma} \Acal_\Sigma$ is just part of the \'etale cover of $\Acal_\Sigma$ by its own Artin cones. Hence we conclude that $q_\sigma$ is \'etale.

Similarly, to see that $\widetilde \pi_\sigma$ is smooth, note that $\pi_\sigma^\textup{trop}$ acts on cones by projection to one of their faces. This implies, that the domain and target of $\widetilde \pi_\sigma$ have compatible covers by Artin cones, such that \'etale locally $\widetilde \pi_\sigma$ has the form
\[
[\mathbb{A}^n / \mathbb{G}_m^n] \to [\mathbb{A}^{n'} / \mathbb{G}_m^{n'}]\,,
\]
where $n \geq n'$ and the map is induced by the projection to the first $n'$ coordinates. It suffices to show that this map is smooth, which we can check on the smooth cover $\mathbb{A}^n \to [\mathbb{A}^n / \mathbb{G}_m^n]$ of its domain. But there it is just given by the composition of the two smooth morphisms
\[
\mathbb{A}^n \to \mathbb{A}^{n'} \to [\mathbb{A}^{n'} / \mathbb{G}_m^{n'}]\,,
\]
first performing a coordinate projection, followed by a principal $\mathbb{G}_m^{n'}$-bundle.

To see that the square diagram is cartesian, we only have to prove on the cone stack side that given an object $(\sigma \xrightarrow{j'} \sigma' \xleftarrow{j''} \sigma'')$ of $\mathsf{Star}_\sigma(\Sigma)$, its image under $\pi_\sigma^\textup{trop}$ is contained in the interior of $\mathsf{Faces}(\sigma)$, and is thus equal to $\sigma \prec \sigma$, if and only if the object is contained in the interior of $\mathsf{Star}_\sigma(\Sigma)$. This was proven in Proposition \ref{Pro:pi_sigma}.

Since $q_\sigma$ is \'etale and $\widetilde \pi_\sigma$ is smooth, we have an isomorphism 
\[
\mathcal{N}_{\mathcal{P}_\sigma/\mathcal{A}_\Sigma} = \mathcal{N}_{\mathcal{P}_\sigma/\mathcal{A}_{\mathsf{Star}_\sigma(\Sigma)}} = \pi_\sigma^\star\  \underbrace{\mathcal{N}_{\Bcal_{\mathsf{Faces}(\sigma)^\circ}/\Acal_{\mathsf{Faces}(\sigma)}}}_{=\bigoplus_\rho \mathcal{L}_\rho}\,.
\]
The formula for $c_1(\pi_\sigma^\star \mathcal L_\rho)$ follows since $c_1(\mathcal{L}_\rho) = \Phi(x_\rho) \in \CH^1(\Acal_{\mathsf{Faces}(\sigma)})$, essentially by definition of $\Phi$.

The claim that $F_\sigma$ vanishes on the boundary of $\mathsf{Star}_\sigma(\Sigma)$ follows since that boundary is precisely the locus where one of the coordinates $x_\rho$ vanishes. As for the statement about fundamental classes, it follows immediately from \cite[Lemma~40]{HMPPS}, which shows that $$[\mathcal{P}_\sigma] = \Phi\left(\prod_{\rho \in \sigma(1)} (\pi_\sigma^\textup{trop})^\star x_\rho \right) \in \CH^\star(\Acal_{\mathsf{Star}_\sigma(\Sigma)})\,,$$ and the injectivity of the pushforward $\CH_\star(\mathcal{P}_\sigma) \to \CH_\star(\Acal_{\mathsf{Star}_\sigma(\Sigma)})$ from Corollary~\ref{cor:vanishingpp}. 
Finally, for \eqref{eqn:sPP_hom_free_module} we simply observe that the condition of a function $f \in \sPP_\star(\mathcal{P}_\sigma)$ vanishing on the boundary, is equivalent to it being divisible by $\prod_{\rho \in \sigma(1)} x_\rho$  on any of its cones. 
\end{proof}

We now return to the geometric situation of a smooth Deligne--Mumford stack $X$ with normal crossings divisor $D$, associated cone stack  $\Sigma_X$ and Artin fan $X \to \mathcal{A}_X = \mathcal{A}_{\Sigma_X}$ as described at the beginning of Section \ref{sec:logtautgeneral}. Recall that the assumptions there ensure that the zero cone $0 \in \Sigma_X$ has trivial automorphism group. 
\begin{definition} \label{Def:monodromy_torsor}
For $\sigma \in \Sigma_X$ we call $P_\sigma = \mathcal{P}_\sigma \times_{\Acal_X} X$ the \emph{monodromy torsor} associated to $\sigma$.
\end{definition}
The action of $\Aut(\sigma)$ on $\mathcal{P}_\sigma$ induces an action on $P_\sigma$ and the natural map $\iota_\sigma : P_\sigma \to X$ is invariant under this action. Using the fact that $X \to \mathcal{A}_X$ is smooth, the content of Lemma \ref{Lem:Sigma_sigma_construction}  immediately implies the following.
\begin{corollary} \label{Cor:Psigma_cartesian_diagram}
Given a cone $\sigma \in \Sigma_X$, we have a fiber diagram
\begin{equation}
\label{eqn:cartesiandiagramstratanormalisation}
\begin{tikzcd}
{P}_\sigma \arrow[r, "p_\sigma"] \arrow[rr, bend left, "\iota_\sigma"] \arrow[d] &  \widetilde{{S}}_\sigma \arrow[r] \arrow[d] & X \arrow[d]\\
\mathcal{P}_\sigma \arrow[r] &  \widetilde{\mathcal{S}}_\sigma \arrow[r] & \mathcal{A}_{X}
\end{tikzcd}\,
\end{equation}
with $p_\sigma$ an $\Aut(\sigma)$-principal bundle.
In particular, the monodromy torsor $P_\sigma$ endowed with the strict topology from $\iota_\sigma$ has tropicalization $\mathsf{Star}_\sigma(\Sigma_X)$ and the map of cone stacks associated to $\iota_\sigma$ is given by
\[
\iota_\sigma^\textup{trop} = \mathsf{Star}_{0 \to \sigma}\,.
\]
Moreover, given a morphism $\psi: \sigma \to \sigma'$ in $\Sigma_X$ we obtain an associated map
\begin{equation} 
\label{eqn:Psigma_fiber_product}
    \iota_{\sigma \to \sigma'} : P_{\sigma'} \to P_\sigma
\end{equation}
of principal bundles as the fiber product of the map $\mathcal{P}_{\sigma'} \to \mathcal{P}_\sigma$ induced by $\mathsf{Star}_\psi$.
\end{corollary}
\begin{proof}
First recall that $\widetilde{\mathcal{S}}_\sigma$ is the normalization of the closure of the stratum $\mathcal{S}_\sigma$ in $\mathcal{A}_{\Sigma_X}$ by Lemma \ref{Lem:Sigma_sigma_construction}. Since the map $X \to \mathcal{A}_{\Sigma_X}$ is smooth and surjective, it follows that the fiber product is indeed the normalization of the closure of $S_\sigma$ in $X$. This shows that the right diagram is a fiber square. The left square is Cartesian by the definition of $P_\sigma$. This immediately implies that the tropicalization of $\iota_{\sigma}$ is the map $\mathsf{Star}_{0 \to \sigma}$ associated to $\mathcal{P}_\sigma \to \Acal_X$. Finally, for  $\psi: \sigma \to \sigma'$ in $\Sigma_X$ as above we first obtain a map $\mathsf{Star}_{\sigma'}({\Sigma_X}) \xrightarrow{\mathsf{Star}_\psi} \mathsf{Star}_{\sigma}({\Sigma_X}) \xrightarrow{\mathsf{Star}_{0 \to \sigma}} {\Sigma_X}$ of cone stacks with boundary. Again by Lemma \ref{Lem:Sigma_sigma_construction} this induces a morphism $\mathcal{P}_{\sigma'} \to \mathcal{P}_\sigma \to \mathcal{A}_X$ of their associated Artin stacks. Taking the fiber product with $X \to \mathcal{A}_X$ gives the desired map $P_{\sigma'} \xrightarrow{\iota_{\sigma \to \sigma'}} P_\sigma \xrightarrow{\iota_{\sigma}} X$.
\end{proof}

\begin{notation} \label{Not:Sigma_sigma}
To ease the notation slightly in the following, we write $\Sigma_\sigma = \mathsf{Star}_\sigma(\Sigma_X)$ for the cone stack with boundary associated to $P_\sigma$. Given $\sigma \to \sigma'$ as above, we also write $\iota_{\sigma \to \sigma'}^\textup{trop} : \Sigma_{\sigma'} \to \Sigma_{\sigma}$ for the map $\mathsf{Star}_{\sigma \to \sigma'}$ associated to $\iota_{\sigma \to \sigma'} : P_{\sigma'} \to P_{\sigma}$.
\end{notation}


\begin{example}[continues Example \ref{ex:a2z2}]
Let $\sigma \in \Sigma_X$ be the two-dimensional cone, then the map $P_\sigma \to X$ factors as
\[
P_\sigma = \mathbb{G}_m \xrightarrow{u \mapsto (0,0,u)} Y = \A^2 \times \G_m \xrightarrow{(x,y,u) \mapsto (x,y,u^2)} X = \A^2 \times \G_m\,.
\]
As the log structure on $X$ is induced by the divisor $D \subseteq X$, the induced log structure on $Y$ comes from the pre-image $E \subseteq Y$ of $D$, cut out by the equation $x^2 - y^2u^2=(x-yu)(x+yu)$. We see that $E$ splits in $E_\pm$ cut out by $x = \pm yu$, with $E_+ \cap E_- = P_{\sigma}$. In particular, as predicted by Proposition \ref{Prop:B_Faces}, the normal bundle of the map $P_\sigma \to X$ splits canonically as
\[
\mathcal{N}_{P_\sigma/X} = \Ocal(E_{+})|_{P_{\sigma}} \oplus \Ocal(E_{-})|_{P_{\sigma}}\,.
\]
\end{example}
\begin{remark}
The Cartesian diagram \eqref{eqn:cartesiandiagramstratanormalisation} is also mentioned in \cite[Section 3.7]{Mol22}. Using the notation above, that paper studies the \emph{strata homology classes} in the Chow groups of $P_\sigma$ pulled back from the stack $\mathcal{P}_\sigma$. As observed in \cite[Lemma 3.9]{Mol22}, pushing forward such classes to $X$ just recovers classes on $X$ defined via strict piecewise polynomials on $\Sigma_X$. In Section \ref{Sect:TautSystems} below we will see how to combine strata homology classes with explicit systems of decorations on $P_\sigma$ itself to obtain potentially larger tautological rings of $X$.
\end{remark}
We conclude Section \ref{Sect:Monodromy_torsors} with some further results on the monodromy torsors which are needed in later proofs. The first one concerns the behavior of the spaces $P_\sigma$ under log blowups of $\widehat X \to X$. Recall that such a log blowup corresponds to a subdivision $\varphi: \widehat \Sigma \to \Sigma_X$. On the cone stack side, we saw in Lemma \ref{Lem:iota_sigmahat_section} that for $\widehat \sigma \in \widehat \Sigma$ mapping to $\sigma \in \Sigma_X$ we obtain a map
\[
t_{\widehat \sigma \to \sigma} : \mathsf{Star}_{\widehat \sigma}(\widehat \Sigma) \to \mathsf{Star}_{\sigma}(\Sigma)\,.
\]
Taking first the geometric realization of \eqref{eqn:t_section_definition} on Artin fans and the fiber product with $X \to \mathcal{A}_X$, we obtain a map
\begin{equation} \label{eqn:s_widehat_sigma_to_sigma}
s_{\widehat \sigma \to \sigma} : P_{\widehat \sigma} \to P_\sigma\,.
\end{equation}
To gain some geometric intuition: in the proof of Theorem \ref{Thm:R_blow_up_determined} we are going to see that if $\widehat \Sigma \to \Sigma$ corresponds to a blowup $\widehat X \to X$ of a smooth stratum closure in $X$, then the map $s_{\widehat \sigma \to \sigma}$ is given by either
\begin{itemize}
    \item a projective bundle, in case that $P_\sigma \to X$ has image inside the blowup center, or
    \item a blowup at a disjoint union of smooth strata closures, representing $P_{\widehat \sigma}$ as the strict transform of the map $P_\sigma \to X$, otherwise.
\end{itemize}

\begin{example}[continues Example \ref{ex:a2z2}]
We take $\tilde{\Sigma}$ to be the subdivision in the ray $\rho' = \langle(1,1)\rangle \subset \sigma$. Then $s_{\rho' \to \sigma}$ is a $\P^1$-bundle, giving the map from the exceptional divisor to the center of the blowup.
\end{example}

An important construction used below is the fiber product of $s_{\widehat \sigma \to \sigma}$ with a map $\iota_{\sigma \to \sigma'}$, which corresponds to restricting $s_{\widehat \sigma \to \sigma}$ to (the parametrization of) a smaller stratum closure inside $P_\sigma$. 
The following proposition shows that this fiber product can be covered by a union of monodromy torsors $P_{\widehat \sigma'}$ for $\widehat \sigma' \in \widehat \Sigma$.
\begin{proposition} \label{Pro:P_sigmahat_sigmaprime_cover}
Let $\pi: \widehat X \to X$ be a log blowup with associated map $\varphi: \widehat \Sigma \to \Sigma_X$ of cone stacks. Let $\widehat \sigma \in \widehat \Sigma$ be a cone mapping to $\sigma \in \Sigma_X$ and choose a morphism $\sigma \to \sigma'$ in $\Sigma_X$. Define $P_{\widehat \sigma \to \sigma'}$ as the fiber product
\begin{equation} \label{eqn:P_widehat_sigma_sigmaprime}
\begin{tikzcd}
P_{\widehat \sigma \to \sigma'} \arrow[d, "\pi_{\sigma'}",swap] \arrow[r, "\iota_{\widehat \sigma \to \sigma'}"] & P_{\widehat \sigma} \arrow[d, "s_{\widehat \sigma \to \sigma}"]\\
P_{\sigma'} \arrow[r, "\iota_{\sigma \to \sigma'}"] & P_{\sigma}
\end{tikzcd}
\end{equation}
whose associated cone stack $\widehat \Sigma_{\widehat \sigma} \times_{\Sigma_\sigma} \Sigma_{\sigma'}$ has interior indexed by commuting diagrams
\begin{equation} \label{eqn:H_widehat_sigma_sigmaprime_diagrams}
\begin{tikzcd}
\widehat \sigma \arrow[d, mapsto, "\varphi"] \arrow[rrr] & & & \widehat \sigma' \arrow[d, mapsto, "\varphi"]\\
\sigma \arrow[r] & \sigma' \arrow[r] & \sigma'' \arrow[r, "\sim"] & \sigma''
\end{tikzcd}\,.
\end{equation}
Let $\mathfrak{H}_{\widehat \sigma \to \sigma'} = \{(\widehat \sigma \to \widehat \sigma', \sigma' \to \sigma'')\}$ be a set of representatives of diagrams \eqref{eqn:H_widehat_sigma_sigmaprime_diagrams} associated to the minimal cones of $(\widehat \Sigma_{\widehat \sigma} \times_{\Sigma_\sigma} \Sigma_{\sigma'})^0$. Then for the induced diagram
\begin{equation} \label{eqn:P_widehat_sigma_sigmaprime_cover}
\begin{tikzcd}
\bigsqcup P_{\widehat \sigma'} \arrow[r, "J"] \arrow[rr, bend left, "\iota_{\widehat \sigma \to \widehat \sigma'}"] \arrow[d, "\sqcup s_{\widehat \sigma' \to \sigma''}", swap] & P_{\widehat \sigma \to \sigma'} \arrow[d, "\pi_{\sigma'}"] \arrow[r, "\iota_{\widehat \sigma \to \sigma'}"] & P_{\widehat \sigma} \arrow[d, "s_{\widehat \sigma \to \sigma}"]\\
\bigsqcup P_{\sigma''} \arrow[r, "\sqcup \iota_{\sigma' \to \sigma''}"] & P_{\sigma'} \arrow[r, "\iota_{\sigma \to \sigma'}"] & P_{\sigma}
\end{tikzcd}
\end{equation}
where the disjoint unions go over elements of $\mathfrak{H}_{\widehat \sigma \to \sigma'}$, 
we have that the map $J$ is proper, representable and surjective.
\end{proposition}
\begin{proof}
The description of the interior of $\widehat \Sigma_{\widehat \sigma} \times_{\Sigma_\sigma} \Sigma_{\sigma'}$ follows directly from the definition of the fiber product of cone stacks. The minimal cones in this interior correspond to the irreducible components of the fiber product $P_{\widehat \sigma \to \sigma'}$. Then the properness, representability and surjectivity of $J$ can first be checked on the level of idealised Artin fans and is preserved under taking the fiber product with $X \to \mathcal{A}_{\Sigma}$. On the level of idealised Artin fans, surjectivity follows from surjectivity of the corresponding map of cone stacks, representability follows as the map on cones is injective on automorphism groups. Finally, properness can be checked by a valuative property, which on the level of cone stacks amounts to the following statement: consider any inclusion of cones 
\begin{equation} \label{eqn:morph_in_Sigmahat_sigmahat_sigmaprime}
    (\widehat \sigma \to \widehat \sigma'_1, \sigma' \to \sigma''_1) \to (\widehat \sigma \to \widehat \sigma'_2, \sigma' \to \sigma''_2)
\end{equation}
in the cone stack of $P_{\widehat \sigma \to \sigma'}$. This  corresponds to a morphism between diagrams of the form \eqref{eqn:H_widehat_sigma_sigmaprime_diagrams}. Now for each cone $(\widehat{\sigma}' \to \widehat{\sigma}'_1)$ of one of the $P_{\widehat \sigma'}$ mapping to the left-hand side of \eqref{eqn:morph_in_Sigmahat_sigmahat_sigmaprime}, we claim that we can find a unique inclusion in a cone mapping to the right-hand side of \eqref{eqn:morph_in_Sigmahat_sigmahat_sigmaprime}. Indeed this is the case, since \eqref{eqn:morph_in_Sigmahat_sigmahat_sigmaprime} contains the data of an inclusion $\widehat \sigma_1' \to \widehat \sigma_2'$ and thus the morphism 
\[(\widehat{\sigma}' \to \widehat{\sigma}'_1) \to (\widehat{\sigma}' \to \widehat{\sigma}'_1 \to \widehat{\sigma}'_2)\]
in $\widehat{\Sigma}_{\widehat \sigma'}$
given by composing with this face map is the unique solution we are looking for.
\end{proof}

\begin{lemma} \label{Lem:strata_closure_smooth}
The strata closure $\overline{S}_\sigma$ associated to $\sigma \in \Sigma_X$ is smooth if and only if the map
\begin{equation} \label{eqn:StarmodAut_subcategory}
    \mathsf{Star}_\sigma(\Sigma)^0 / \mathsf{Aut}(\sigma) \to \Sigma
\end{equation}
is a fully faithful embedding.
\end{lemma}
\begin{proof}
We have that $\overline{S}_\sigma$ is smooth if and only if the normalization map $\widetilde{S}_\sigma \to \overline{S}_\sigma$ is an isomorphism. Reformulating, this is the case if and only if the map $\widetilde{S}_\sigma \to X$ is a closed embedding. This can equivalently be checked for the map $\widetilde{\mathcal{S}}_\sigma \to \mathcal{A}_X$ on the Artin fan side. By Lemma \ref{Lem:Sigma_sigma_construction} the map $\widetilde{\mathcal{S}}_\sigma \to \mathcal{A}_X$ corresponds to the map \eqref{eqn:StarmodAut_subcategory} of cone stacks with boundary. It is thus a closed embedding if and only if this underlying map of cone stacks is a fully faithful embedding.
\end{proof}

Let $\Sigma$ be a cone stack and $\Sigma^0$ be a forward-closed subcategory. We say that $\Sigma^0$ is \emph{connected} if the graph whose vertices are isomorphism classes of objects of $\Sigma^0$ and whose edges correspond to morphisms in $\Sigma^0$ is connected. It is straightforward to see that $\Sigma^0$ is connected if and only if the associated stack $\mathcal{B}_{\Sigma^0}$ is connected.

\begin{lemma}
\label{Lem:smooth_eq_star_quotient}
Let $\Sigma_X^0$ be a connected forward-closed subcategory in the cone stack $\Sigma_X$ with complement $\Delta$. Then $\Bcal_{\Sigma_X^0}$ is smooth if and only if there is a stratum $\sigma \in \Sigma_X^0$ such that the embedding $\Sigma_X^0 \to \Sigma_X$ is isomorphic to $\Star_{\sigma}(\Sigma_X)^0/\Aut(\sigma) \to \Sigma_X$.
\end{lemma}
\begin{proof}
First assume $\Bcal_{\Sigma_X^0}$ is smooth. Then since it is also connected, it must be irreducible and hence has a generic point corresponding to some minimal stratum $\sigma$ in $\Sigma_X$. Then the equivalence follows from Lemma~\ref{Lem:strata_closure_smooth}. The implication the other way is also direct from Lemma~\ref{Lem:strata_closure_smooth}.
\end{proof}


\begin{example}[continues Example \ref{ex:a2z2}]
We know the stratum $D$ is not smooth, as it has a self-intersection. The corresponding substack with boundary \[\Sigma_X^0 = \left(\begin{tikzcd} \rho \arrow[r, "i_1", shift left = 1] \arrow[r, swap,"i_2", shift right = 1] & \sigma \arrow[loop right]{r}\end{tikzcd}\right)\] does have a unique minimal object, but is is not of the form stated in Lemma~\ref{Lem:smooth_eq_star_quotient}, as the map $\mathsf{Star}_\rho(\Sigma_X)^0/\Aut(\rho) \to \Sigma_X$ is not isomorphic to $\Sigma_X^0 \to \Sigma_X$ (indeed, the map from the star is not even an embedding).

The stratum $D^{(2)}$ is smooth, and does trivially satisfy the condition from Lemma~\ref{Lem:smooth_eq_star_quotient}.
\end{example}

\subsubsection{Star subdivisions} \label{Sect:Star_subdivisions}

Given a cone $\sigma \in \Sigma_X$ as in Lemma \ref{Lem:strata_closure_smooth}, the blowup
\begin{equation}
    \widetilde X = \mathsf{Bl}_{\overline{S}_\sigma} X \to X
\end{equation}
is a logarithmic modification, corresponding to the \emph{star subdivision} $\mathsf{ssd}_\sigma(\Sigma_X) \to \Sigma_X$ of the cone stack $\Sigma_X$. Since it will be needed later, we give an explicit construction of this subdivision below.
\begin{construction} \label{Const:ssd}
Consider the decomposition $\Sigma_X =  \mathsf{Star}_\sigma(\Sigma_X)^0 / \mathsf{Aut}(\sigma) \sqcup \mathcal{R}$ of the objects of $\Sigma_X$ into subcategories induced by the embedding \eqref{eqn:StarmodAut_subcategory}. Objects of $\mathsf{Star}_\sigma(\Sigma_X)^0 / \mathsf{Aut}(\sigma)$ are indexed by morphisms $(\sigma \to \sigma')$ in $\Sigma_X$ and we denote $\widetilde \sigma \in \mathcal{R}$ the  objects of the residual category $\mathcal{R}$.

Then the objects of the star subdivision $\mathsf{ssd}_\sigma(\Sigma_X)$ decompose as
\[
\mathsf{ssd}_\sigma(\Sigma_X) = \mathsf{ssd}(\mathsf{Star}_\sigma(\Sigma_X)^0) / \mathsf{Aut}(\sigma) \sqcup \mathcal{R}\,,
\]
where elements of $\mathsf{ssd}(\mathsf{Star}_\sigma(\Sigma_X)^0) / \mathsf{Aut}(\sigma)$ are indexed by diagrams $(\tau \xrightarrow{h} \sigma \xrightarrow{h'} \sigma')$ in $\Sigma_X$ such that $h$ is a \emph{proper} face morphism (i.e. a morphism in $\Sigma$ whose induced map on cones is the inclusion of a proper face of $\sigma$). Denoting by $b_\sigma = \sum_{\rho \in \sigma(1)} u_\rho$ the barycenter of $\sigma$, the associated cone to this object is the subcone of $\sigma'$ given by the convex hull
\[
C(\tau \xrightarrow{h} \sigma \xrightarrow{h'} \sigma') = \mathsf{cone}(\{h'(b_\sigma)\} \cup (h' \circ h)(\tau) \cup (\sigma'(1) \setminus h'(\sigma(1)))) \subseteq \sigma'
\]
of the barycenter $b_\sigma$ (or rather its image under $h'$), the image of $\tau$ in $\sigma'$ and all the rays $\rho \in \sigma'(1)$ which are not in the image of any of the rays of $\sigma$. On the other hand, the cone $C(\widetilde \sigma)$ associated to $\widetilde \sigma \in \mathcal{R}$ is just the same cone as for $\Sigma$. See Figure \ref{fig:star subdivision} for an example of a star subdivision, showing the cones $C(\tau \xrightarrow{h} \sigma \xrightarrow{h'} \sigma')$ in red.

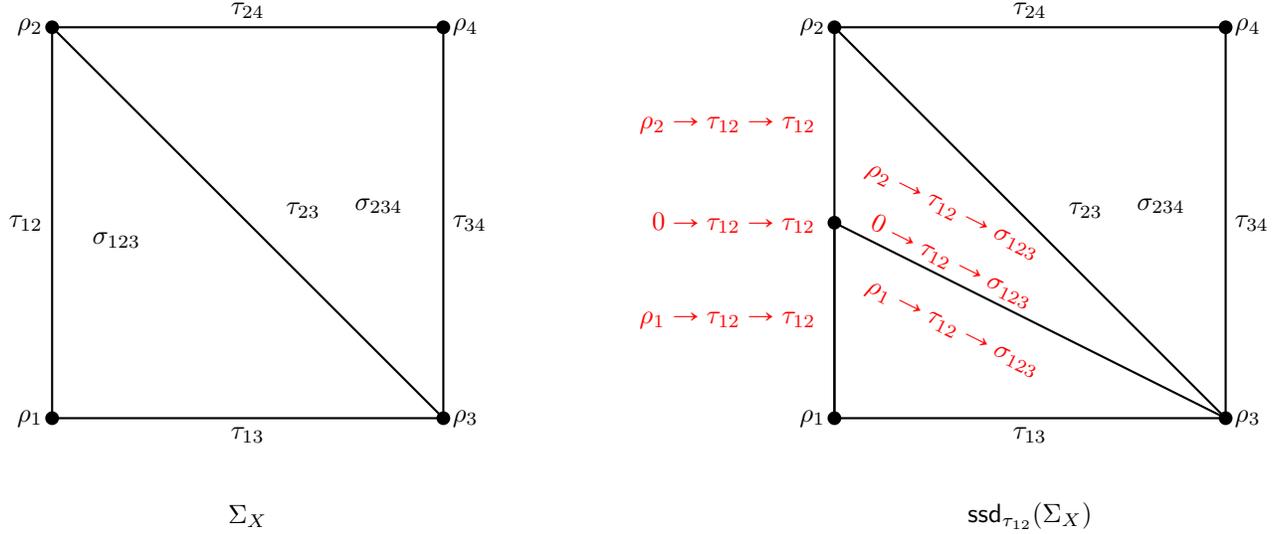
\begin{figure}[htb]
\centering
\begin{tikzpicture}[scale=1.3]

\coordinate (S1) at (0,0);
\coordinate (S2) at (0,4);
\coordinate (S3) at (4,0);
\coordinate (S4) at (4,4);

\draw [thick] (S1) -- (S2) -- (S4) -- (S3) -- cycle;
\draw [thick] (S2) -- (S3);

\node [left] at (S2) {$\rho_2$};
\node [left] at (S1) {$\rho_1$};
\node [right] at (S4) {$\rho_4$};
\node [right] at (S3) {$\rho_3$};

\node [above] at (2,4) {$\tau_{24}$};
\node [below] at (2,0) {$\tau_{13}$};
\node [left] at (0,2) {$\tau_{12}$};
\node [right] at (4,2) {$\tau_{34}$};
\node [below right] at (2.3,2.3) {$\tau_{23}$};

\node [below left] at (1,2) {$\sigma_{123}$};
\node [above right] at (3,2) {$\sigma_{234}$};

\fill (S1) circle (2pt);
\fill (S2) circle (2pt);
\fill (S3) circle (2pt);
\fill (S4) circle (2pt);

\begin{scope}[xshift=8cm]
\coordinate (R1) at (0,0);
\coordinate (R2) at (0,4);
\coordinate (R3) at (4,0);
\coordinate (R4) at (4,4);
\coordinate (M12) at (0,2); 

\draw [thick] (R1) -- (R2) -- (R4) -- (R3) -- cycle;
\draw [thick] (R2) -- (R3);
\draw [thick] (R1) -- (M12);
\draw [thick] (M12) -- (R3);

\node [left] at (R2) {$\rho_2$};
\node [left] at (R1) {$\rho_1$};
\node [right] at (R4) {$\rho_4$};
\node [right] at (R3) {$\rho_3$};

\node [above] at (2,4) {$\tau_{24}$};
\node [below] at (2,0) {$\tau_{13}$};
\node [right] at (4,2) {$\tau_{34}$};
\node [below right] at (2.3,2.3) {$\tau_{23}$};

\node [above right] at (3,2) {$\sigma_{234}$};

\node [red, left] at (-0.1,3) {$\rho_2 \to \tau_{12} \to \tau_{12}$};
\node [red, left] at (-0.1,2) {$0 \to \tau_{12} \to \tau_{12}$};
\node [red, left] at (-0.1,1) {$\rho_1 \to \tau_{12} \to \tau_{12}$};
\node [red, rotate=-26.57] at (1.2,2.1) {$\rho_2 \to \tau_{12} \to \sigma_{123}$};
\node [red, rotate=-26.57] at (1.2,1.6) {$0 \to \tau_{12} \to \sigma_{123}$};
\node [red, rotate=-26.57] at (1.2,0.9) {$\rho_1 \to \tau_{12} \to \sigma_{123}$};


\fill (R1) circle (2pt);
\fill (R2) circle (2pt);
\fill (R3) circle (2pt);
\fill (R4) circle (2pt);
\fill (M12) circle (2pt);

\node at (2,-1) {$\mathsf{ssd}_{\tau_{12}}(\Sigma_X)$};
\end{scope}
\node at (2,-1) {$\Sigma_X$};
\end{tikzpicture}
\caption{A cone stack $\Sigma_X$ and its star subdivision at an object $\tau_{12}$, with new cones on the right indicated in red. We draw here a slice through the underlying $3$-dimensional picture.}
\label{fig:star subdivision}
\end{figure}

Having defined the two types of objects $(\tau \xrightarrow{h} \sigma \xrightarrow{h'} \sigma')$ and $\widetilde \sigma$ of $\mathsf{ssd}_\sigma(\Sigma_X)$, the morphisms between them are given as follows:
\begin{align*}
\mathsf{Mor}_{\mathsf{ssd}_\sigma(\Sigma_X)}(\tau_1 \xrightarrow{h_1} \sigma \xrightarrow{h_1'} \sigma_1', \tau_2 \xrightarrow{h_2} \sigma \xrightarrow{h_2'} \sigma_2') &= \left\{
\begin{tikzcd}[ampersand replacement=\&]
    \tau_1 \arrow[r,"h_1"] \arrow[d] \& \sigma \arrow[r,"h'_1"] \arrow[d] \& \sigma_1' \arrow[d]\\
    \tau_2 \arrow[r,"h_2"] \& \sigma \arrow[r,"h_2'"] \& \sigma_2'
\end{tikzcd}
\right\}\\
\mathsf{Mor}_{\mathsf{ssd}_\sigma(\Sigma_X)}(\tau \xrightarrow{h} \sigma \xrightarrow{h'} \sigma', \widetilde \sigma) &= \emptyset\\
\mathsf{Mor}_{\mathsf{ssd}_\sigma(\Sigma_X)}(\widetilde \sigma, \tau \xrightarrow{h} \sigma \xrightarrow{h'} \sigma') &= \{(\widetilde \sigma \xrightarrow{i} \sigma') \in \mathsf{Mor}_{\Sigma_X}(\widetilde \sigma, \sigma') : \mathsf{im}(i) \subseteq C(\tau \xrightarrow{h} \sigma \xrightarrow{h'} \sigma')\}\\
\mathsf{Mor}_{\mathsf{ssd}_\sigma(\Sigma_X)}(\widetilde \sigma_1, \widetilde \sigma_2) &= \mathsf{Mor}_{\Sigma_X}(\widetilde \sigma_1, \widetilde \sigma_2)
\end{align*}
Both the composition of such morphisms and their associated face maps under the functor $C : \mathsf{ssd}_\sigma(\Sigma_X) \to  \mathbf{RPC}^f$ are straightforward to define. 
Similarly, we verify that $\mathsf{ssd}_\sigma(\Sigma_X)$ is a cone stack and it has a natural map
\begin{equation} \label{eqn:b_sigma_trop}
    b_{\sigma, \Sigma_X}^\textup{trop} : \mathsf{ssd}_\sigma(\Sigma_X) \to \Sigma_X, (\tau \xrightarrow{h} \sigma \xrightarrow{h'} \sigma') \mapsto \sigma', \widetilde \sigma \mapsto \widetilde \sigma\,.
\end{equation}
\end{construction}

\begin{proposition} \label{Pro:blow_up_description}
The morphism $b_{\sigma, \Sigma_X}: \mathcal{A}_{\mathsf{ssd}_\sigma(\Sigma_X)} \to \mathcal{A}_{\Sigma_X}$ induced by the map $b_{\sigma, \Sigma_X}^\textup{trop}$ is the blowup of the strata closure $\overline{\mathcal{S}}_{\sigma} \subseteq \mathcal{A}_{\Sigma_X}$.
\end{proposition}
\begin{proof}
As in the proof of Lemma \ref{Lem:Sigma_sigma_construction}, it is sufficient to verify the claim on the toric atlas $\coprod_{\sigma_0 \in \Sigma_X} V_{\sigma_0} \to \mathcal{A}_{\Sigma_X}$. On the fan of each such $V_{\sigma_0}$, there are two cases for the pullback of the subdivision $b_{\sigma, \Sigma_X}^\textup{trop}$ under the map $\sigma_0 \to \Sigma_X$:
\begin{itemize}
    \item If there exists a map $\sigma \to \sigma_0$ in $\Sigma_X$, the condition from Lemma \ref{Lem:strata_closure_smooth} implies that the image of $\sigma$ in $\sigma_0$ is unique. Then the pullback subdivision is the star subdivision at the image of the barycenter of $\sigma$ in $\sigma_0$ (which only depends on the image of $\sigma$). 
    \item If there exists no map $\sigma \to \sigma_0$, the cone $\sigma_0$ is not subdivided when pulling back $\mathsf{ssd}_\sigma(\Sigma_X) \to \Sigma_X$ under $\sigma_0 \to \Sigma_X$.
\end{itemize}
Comparing with the classical toric correspondence between star subdivisions and blowups of smooth strata (\cite[Proposition 3.3.15]{CLS11}), we see that this precisely corresponds to the blowup of the preimage of $\overline{\mathcal{S}}_{\sigma}$ under the map $V_{\sigma_0} \to \mathcal{A}_{\Sigma_X}$. This correspondence is also compatible on the overlaps of the $V_{\sigma_0}$ and thus the original map is given by the blowup of $\overline{\mathcal{S}}_{\sigma}$ in $\mathcal{A}_{\Sigma_X}$ as claimed.
\end{proof}

\begin{example}[continues Example \ref{ex:a2z2}]
We consider the star subdivision in the two dimensional cone $\sigma$. We see we end up with the following cone stack $\tilde{\Sigma} = \mathsf{ssd}_\sigma(\Sigma_X)$:
\begin{equation}
\label{eqn:cone_stack_a2z2_blowup}
    \begin{tikzcd}
     & (0 \to \sigma \to \sigma) \arrow[loop right]{u} \arrow[rd]& \\
    \tilde{0} \arrow[r] \arrow[ru] & \tilde{\rho} \arrow[r] \arrow[r] & (\rho \to \sigma \to \sigma) 
    \end{tikzcd}
\end{equation}
On the scheme level, the blowup $\tilde{X}$ of $X$ in $D^{(2)}$ has two codimension one strata: the strict transform of $D$ and the exceptional divisor. They intersect in a single connected stratum that is abstractly isomorphic to $\G_m$. 
And indeed we can obtain $\tilde X$ as the fiber product
\[
\begin{tikzcd}
\tilde X  \arrow[r] \arrow[d]& X \arrow[d] \\
\mathcal{A}_{\widetilde \Sigma} \arrow[r] & \mathcal{A}_{\Sigma_X}
\end{tikzcd}
\]
making $\tilde X \to \mathcal{A}_{\widetilde \Sigma}$ a choice of \emph{an} Artin fan (in the sense of Definition \ref{Def:an_Artin_fan}). 

However, we warn the reader that $\widetilde \Sigma$ is \emph{not} the canonical cone stack $\Sigma_{\tilde{X}}$ of the blowup $\tilde X$, as $\Sigma_{\tilde{X}}$ does not have a non-trivial automorphism of the exceptional divisor while $\widetilde \Sigma$ does. In fact, $\widetilde \Sigma$ is the \emph{relative} cone stack (or the cone stack corresponding to the relative Artin fan), and there is no map $\Sigma_{\tilde{X}} \to \widetilde \Sigma$. In fact, $\tilde{X} \to X$ is the standard example where functoriality of the canonical Artin fan fails, see \cite[Section~5.4]{ACMUW}. In contrast, the star subdivision is functorial, by construction.
\end{example}

\subsection{Tautological systems} \label{Sect:TautSystems}
In the following, given a normal crossings pair $(X,D)$ with choice of cone stack $\Sigma_X$ and Artin fan $\mathcal{A}_X$, we define the notion of tautological classes on $X$. These will not be fully intrinsic to $(X,D)$, but depend on an additional choice of systems of Chow classes on the monodromy torsors $P_\sigma$ from Definition \ref{Def:monodromy_torsor}. For this recall that the strata $S_\sigma$ of $X$ correspond to the cones $\sigma \in \Sigma_X$, and for a given cone we defined a map $P_\sigma \to \widetilde S_\sigma \to X$ parameterizing the normalization $\widetilde S_\sigma$ of the strata closure of $S_\sigma$. We denoted by $(\Sigma_\sigma, \Delta_\sigma)$ the cone stack with boundary associated to $P_\sigma$ (Notation \ref{Not:Sigma_sigma}).
\begin{definition}
\label{def:tautsystems}
A \emph{system of tautological rings} $\R_X = (\R^\star(P_\sigma))_{\sigma \in \Sigma_X}$ on $(X,D)$ is data of a tautological subring
\begin{equation} \label{eqn:strata_taut_ring_inclusion}
    \Psi(\mathsf{sPP}^\star(\Sigma_\sigma, \Delta_\sigma)) = \Psi(\sPP_\star(P_\sigma)) \subseteq  \R^\star(P_\sigma) \subseteq \mathsf{CH}^\star(P_\sigma)
\end{equation}
which contains all classes induced from strict piecewise polynomials on $\Sigma_\sigma$ vanishing on the boundary. 


Furthermore, 
we assume that pushforward and pullback under $\iota_{\sigma \to \sigma'}$ give maps
\begin{equation} \label{eqn:gluing_map_preserving}
    \iota_{\sigma \to \sigma' \star} : \R^\star(P_{\sigma'}) \to \R^\star(P_\sigma) \text{ and }\iota_{\sigma \to \sigma'}^\star : \R^\star(P_{\sigma}) \to \R^\star(P_{\sigma'})
\end{equation}
respecting the tautological subrings on the monodromy torsors $P_\sigma$ and $P_{\sigma'}$.
\end{definition}
Note that pushing the tautological ring on $P_\sigma$ forward under the principal bundle map $P_\sigma \to \widetilde S_\sigma$ gives the tautological ring
\begin{equation}
    \R^\star(\widetilde S_\sigma) = p_\star \R^\star(P_\sigma) \cong \R^\star(P_\sigma)^{\Aut(\sigma)}
\end{equation}
where the equality with the $\Aut(\sigma)$-invariant part of $\R^\star(P_\sigma)$ follows from \cite[Lemma 2.20]{BaeSchmitt2}.
\begin{example}  \label{eqn:maximal_tautological_system}
The \emph{Chow system} $\mathsf{CH}_X$ of tautological rings on $(X,D)$ is given by setting $\R^\star(P_\sigma) = \mathsf{CH}^\star(P_\sigma)$ to be the full Chow ring. This trivially satisfies the compatibility conditions \eqref{eqn:strata_taut_ring_inclusion} and \eqref{eqn:gluing_map_preserving}.  
\end{example}

\begin{example}
For $X = \Mbar_{g,n}$ with its usual boundary divisor we take $\Sigma_X = \Sigma_{g,n}$. Its cones $\sigma$ are associated to stable graphs $\Gamma$ and for $\sigma = \sigma_\Gamma$ we have an associated principal bundle
\[
P_\Gamma = \Mbar_\Gamma \to \widetilde S_\Gamma = \Mbar_\Gamma / \mathsf{Aut}(\Gamma)
\]
for the monodromy group $\Aut(\sigma) = \mathsf{Aut}(\Gamma)$ (see \cite[Section 6.2.5]{HMPPS}). The standard system of tautological rings $\R_{\Mbar_{g,n}}$ is defined by setting  $\R^\star(P_\Gamma) = \R^\star(\Mbar_\Gamma)$ to be the image of the tensor product of tautological rings of the factors $\Mbar_{g(v), n(v)}$ as usual. The relevant ring of strict piecewise polynomials is given by
\[
\mathsf{sPP}^\star(\Mbar_\Gamma) = \left(\prod_{e \in E(\Gamma)} 1 \otimes \ell_e \right) \cdot  \bigotimes_{v \in V(\Gamma)} \mathsf{sPP}^\star(\Sigma_{g(v),n(v)}) \otimes_\mathbb{Q}  \mathbb{Q}[\ell_e : e \in E(\Gamma)] \,.
\]
Under the map $\Psi$, the strict piecewise polynomials on the fans $\Sigma_{g(v),n(v)}$ map to tautological classes in $\R^\star(\Mbar_{g(v),n(v)})$, whereas for an edge $e = (h,h') \in E(\Gamma)$, the length function $\ell_e$ maps to $-\psi_{h} - \psi_{h'}$. In particular, this implies the desired inclusion \eqref{eqn:strata_taut_ring_inclusion} of classes from strict piecewise polynomials in the tautological rings of $\Mbar_\Gamma$.

Finally, the morphisms $\sigma_{\Gamma'} \to \sigma_\Gamma$ in $\Sigma_{g,n}$ are in correspondence with stable graph morphisms $\varphi: \Gamma \to \Gamma'$ and the associated map of principal bundles is a partial gluing morphism 
\begin{equation}
    \iota_\varphi : \Mbar_{\Gamma} \to \Mbar_{\Gamma'}\,,
\end{equation}
which glue together all pairs of nodes associated to edges in $E(\Gamma) \setminus E(\Gamma')$ (see \cite[Notation 2.8]{SvZ}). By \cite[Appendix A]{GP03}, the pushforward and pullback under $\iota_\varphi$ indeed preserves the tautological rings, so that condition \eqref{eqn:gluing_map_preserving} holds.
\end{example}
Given a smooth log blowup $\pi: (\widehat X, \widehat D) \to (X,D)$ associated to a subdivision $\varphi: \widehat \Sigma \to \Sigma_X$, we claim that there is a natural system of tautological rings $(\R^\star(P_{\widehat \sigma}))_{\widehat \sigma \in \widehat \Sigma}$ on $(\widehat X, \widehat D)$. To describe it, let $\widehat \sigma \in \widehat \Sigma$ be a cone of the subdivision, and $\sigma = \varphi(\widehat \sigma) \in \Sigma_X$ its image in $\Sigma_X$. Then given any map $\sigma \to \sigma'$ in $\Sigma_X$ we can take the fiber product\footnote{Note that since the map $\iota_{\sigma \to \sigma'}$ below is strict, the fiber product in the category of fs log stacks agrees with the fiber product of stacks, by \cite[Remark~III.2.1.3]{Ogu06}. 
} 
\begin{equation} \label{eqn:def_P_sigmahat_sigma}
\begin{tikzcd}
P_{\widehat \sigma \to \sigma'} \arrow[r, "\iota_{\widehat \sigma \to \sigma'}"] \arrow[d, "\pi_{\sigma'}", swap] & P_{\widehat \sigma} \arrow[d, "s_{\widehat \sigma \to \sigma}"]\\
P_{\sigma'} \arrow[r, "\iota_{\sigma \to \sigma'}"] & P_{\sigma}
\end{tikzcd}
\end{equation}
as in Proposition \ref{Pro:P_sigmahat_sigmaprime_cover}.
The cone stack $\Sigma_{\widehat \sigma \to \sigma'}$ associated to $P_{\widehat \sigma \to \sigma'}$ is given by the analogous fiber product $\widehat \Sigma_{\widehat \sigma} \times_{\Sigma_\sigma} \Sigma_{\sigma'}$ in the category of cone stacks and has a natural boundary $\Delta_{\widehat \sigma \to \sigma}^0$.
Denote by
\begin{equation}
\sPP_\star(P_{\widehat \sigma \to \sigma'}) = \sPP^\star(\Sigma_{\widehat \sigma \to \sigma}, \Delta_{\widehat \sigma \to \sigma}^0)
\end{equation}
the strict piecewise polynomials on $\Sigma_{\widehat \sigma \to \sigma}$ vanishing on the boundary, as usual.

\begin{definition}
\label{def:logdecstratum}
Let $\pi: (\widehat X, \widehat D) \to (X,D)$ be  a smooth log blowup associated to a subdivision $\varphi: \widehat \Sigma \to \Sigma_X$. Let $\widehat \sigma \in \widehat \Sigma$ be a cone mapping to $\sigma = \varphi(\widehat \sigma) \in \Sigma_X$. 
A \emph{decorated log-strata class} on $P_{\widehat \sigma}$ is described by a triple $[\sigma', f, \deco]$ of a cone $\sigma' \in \Sigma_X$ admitting a map\footnote{In fact the class depends on the choice of such a map, but for simplicity we suppress this in the notation. \label{footnote:morphism_notation_suppression}} $\sigma \to \sigma'$ in $\Sigma_X$, a piecewise polynomial $f \in \sPP_\star(P_{\widehat \sigma \to \sigma'})$ and a decoration $\deco \in \R^\star(P_{\sigma'})$. Its associated class in $\mathsf{CH}^\star(X)$ is given by
\begin{equation} \label{eqn:decoratedlogstratum_X}
[\sigma', f, \deco] = \iota_{\widehat \sigma \to \sigma' \star} \left(\pi_{\sigma'}^\star \deco \cdot \Psi(f) \right) \in \mathsf{CH}^\star(P_{\widehat \sigma})\,.
\end{equation}
We denote by $(\pi^\star\R_X)(P_{\widehat \sigma})$ the $\mathbb{Q}$-vector subspace of $\mathsf{CH}^\star(P_{\widehat \sigma})$ spanned by all decorated log-strata classes.
\end{definition}

Our next goal is to prove the following:
\begin{theorem} \label{Thm:taut_system_induction}
The collection $\R_{\widehat X} = ((\pi^\star\R_X)(P_{\widehat \sigma}))_{\widehat \sigma \in \widehat \Sigma}$ forms a system of tautological rings on $(\widehat X, \widehat D)$. 
\end{theorem}
\begin{definition}
Given a smooth log blowup $\pi: (\widehat X, \widehat D) \to (X,D)$, and a system $\R_X$ of tautological rings on $X$, we denote by $\R_{\widehat X} = \pi^\star \R_X$ the induced tautological system on $\widehat X$ from Theorem \ref{Thm:taut_system_induction}.
\end{definition}

\begin{remark}
\label{rem:pushforwardtautsystems}
As the tautological system $\R_X$ contains classes of homological piecewise polynomials per definition, there are also pushforward maps $\pi_\star: \R_{\widehat X}(P_{\widehat \sigma}) \to \R_{X}(P_\sigma)$ for a cone $\widehat \sigma \in \widehat \Sigma$ mapping to $\sigma \in \Sigma_X$.
\end{remark}

This shows that once we define tautological classes on $X$ and its strata, we obtain a notion of tautological classes on any smooth log blowup of $X$, allowing us to define log tautological rings in Section \ref{Sect:log_tautological_rings_in_general}. 

To start approaching the proof of Theorem \ref{Thm:taut_system_induction}, the first step is to give a slightly different generating set of $(\pi^\star\R_X)(P_{\widehat \sigma})$, for which some of the calculations are easier to express explicitly. For this choose $\widehat \sigma \to \widehat \sigma'$ a morphism in $\widehat \Sigma$ mapping to $\sigma \to \sigma'$ under $\varphi: \widehat \Sigma \to \Sigma_X$. From equation \eqref{eqn:t_functoriality} we see that we have a commutative diagram
\[
\begin{tikzcd}
P_{\widehat \sigma'} \arrow[r, "\iota_{\widehat \sigma \to \widehat \sigma'}"] \arrow[d, "s_{\widehat \sigma' \to \sigma'}", swap] & P_{\widehat \sigma}\arrow[d, "s_{\widehat \sigma \to \sigma}"]\\
P_{\sigma'} \arrow[r, "\iota_{\sigma \to \sigma'}"] & P_{\sigma}
\end{tikzcd}
\]
Given $\beta \in \R_X(P_{\sigma'})$ and $g \in \sPP_\star(P_{\widehat \sigma'})$ we can define the class
\begin{equation} \label{eqn:strict_decoratedlogstratum_X}
\{\widehat \sigma', g, \beta\} = (\iota_{\widehat \sigma \to \widehat \sigma'})_\star( (s_{\widehat \sigma' \to \sigma'})^\star \beta \cdot \Psi(g) ) \in \CH^\star(P_{\widehat \sigma})\,,
\end{equation}
which we call a \emph{strictly decorated log stratum class}.
\begin{lemma}  \label{Lem:alternative_generators}
Inside $\CH^\star(P_{\widehat \sigma})$, the $\mathbb{Q}$-linear spans of all classes $[\sigma', f, \deco]$ from \eqref{eqn:decoratedlogstratum_X} and of all classes $\{\widehat \sigma', g, \beta\}$ from \eqref{eqn:strict_decoratedlogstratum_X} coincide.
\end{lemma}
\begin{proof}
Consider a decorated log-stratum class $[\sigma', f, \deco]$ recall from Proposition \ref{Pro:P_sigmahat_sigmaprime_cover} that we have the diagram 
\[
\begin{tikzcd}
\bigsqcup P_{\widehat \sigma'} \arrow[r, "J"] \arrow[rr, bend left, "\iota_{\widehat \sigma \to \widehat \sigma'}"] \arrow[d, "\sqcup s_{\widehat \sigma' \to \sigma''}", swap] & P_{\widehat \sigma \to \sigma'} \arrow[d, "\pi_{\sigma'}"] \arrow[r, "\iota_{\widehat \sigma \to \sigma'}"] & P_{\widehat \sigma} \arrow[d, "s_{\widehat \sigma \to \sigma}"]\\
\bigsqcup P_{\sigma''} \arrow[r, "\sqcup \iota_{\sigma' \to \sigma''}"] & P_{\sigma'} \arrow[r, "\iota_{\sigma \to \sigma'}"] & P_{\sigma}
\end{tikzcd}
\]
with the disjoint union indexed by $(\widehat \sigma \to \widehat \sigma', \sigma' \to \sigma'') \in \mathfrak{H}_{\widehat \sigma \to \sigma'}$ and the map $J$ being proper, representable and surjective.
By Proposition \ref{Pro:proper_pushforward_surjective}, we can then find $g_{\widehat \sigma'} \in \sPP_\star(P_{\widehat \sigma'})$ for $(\widehat \sigma \to \widehat \sigma', \sigma' \to \sigma'') \in \mathfrak{H}_{\widehat \sigma \to \sigma'}$ such that
\[
(J^\textup{trop})_\star(\sum g_{\widehat \sigma'}) = f \in \sPP_\star(P_{\widehat \sigma \to \sigma'})\,.
\]
Then we can conclude
\[
[\sigma', f, \deco] = \sum_{(\widehat \sigma \to \widehat \sigma', \sigma' \to \sigma'') \in \mathfrak{H}_{\widehat \sigma \to \sigma'}}  \{\widehat \sigma', g_{\widehat \sigma'}, \iota_{\sigma' \to \sigma''}^\star\deco \}\,,
\]
where we use that the tautological system on $X$ is closed under pullbacks by the maps $\iota_{\sigma' \to \sigma''}$.

On the other hand, fix a strictly decorated log-stratum class $\{\widehat \sigma', g, \beta\}$ with $\varphi(\widehat \sigma') = \sigma' \in \Sigma_X$. Then we have a solid diagram of morphisms
\[
\begin{tikzcd}
P_{\widehat \sigma'} \arrow[r, dashed, "j"] \arrow[rr, bend left, "\iota_{\widehat \sigma \to \widehat \sigma'}"] \arrow[dr, "s_{\widehat \sigma' \to \sigma'}", swap] & P_{\widehat \sigma \to \sigma'} \arrow[d, "\pi_{\sigma'}"] \arrow[r, "\iota_{\widehat \sigma \to \sigma'}"] & P_{\widehat \sigma} \arrow[d, "s_{\widehat \sigma \to \sigma}"]\\
 & P_{\sigma'} \arrow[r, "\iota_{\sigma \to \sigma'}"] & P_{\sigma}
\end{tikzcd}
\]
which is commutative by (the geometric realization of) Lemma \ref{Lem:iota_sigmahat_section}. Then the dashed arrow $j$ exists by the universal property of the fiber product $P_{\widehat \sigma \to \sigma'}$ and we have
\[
\{\widehat \sigma', g, \beta\} = [\sigma', j^\textup{trop}_\star g, \beta]\,.
\]
Thus we have expressed all classes $[\sigma', f, \deco]$ as linear combinations of cycles $\{\widehat \sigma', g, \beta\}$, and vice versa, and hence their linear spans coincide.
\end{proof}

An important first step in the proof of Theorem \ref{Thm:taut_system_induction} is to show that each group $(\pi^\star R_X)(P_{\widehat \sigma})$ forms a $\mathbb{Q}$-algebra, i.e. is closed under intersection products. In fact, it is possible to give an explicit formula for the product of two decorated log-strata, analogous to the formula for products of decorated boundary strata in $\Mbar_{g,n}$ presented in \cite[Appendix A]{GP03}. We begin by defining the relevant notation to state this formula.
\begin{definition}
Given two maps $\sigma_1 \leftarrow \sigma \rightarrow \sigma_2$ in $\Sigma_X$, a \emph{generic $(\sigma_1, \sigma_2)$-structure over $\sigma$} is given by a triple $(\sigma', \varphi_1:  \sigma_1 \to \sigma', \varphi_2: \sigma_2 \to \sigma')$ of a cone $\sigma' \in \Sigma_X$ and morphisms $\varphi_1, \varphi_2$ in $\Sigma_X$ such that the diagram
\begin{equation} \label{eqn:comm_diag_sigmas}
\begin{tikzcd}
\sigma \arrow[r] \arrow[d] & \sigma_1 \arrow[d, "\varphi_1"]\\
\sigma_2 \arrow[r, "\varphi_2"] & \sigma'
\end{tikzcd}
\end{equation}
commutes, and such that each ray of $\sigma'$ is in the image of either $\varphi_1$ or $\varphi_2$.

A second triple $(\sigma'', \varphi_1', \varphi_2')$ is called isomorphic if there is an isomorphism $\sigma'' \to \sigma'$ in $\Sigma_X$ making the obvious diagrams commute. Denote by $\mathfrak{G}_{\sigma_1 \leftarrow \sigma \rightarrow \sigma_2}$ the set of isomorphism classes of generic $(\sigma_1, \sigma_2)$-structures.
\end{definition}
When $\sigma = 0$ is the trivial cone, we will often just write
\[
\mathfrak{G}_{\sigma_1, \sigma_2} = \mathfrak{G}_{\sigma_1 \leftarrow \sigma \rightarrow \sigma_2}\,.
\]

\begin{lemma} \label{Lem:gluing_fiber_product_stacks}
Given two cone stack morphisms $\iota_{\sigma \to \sigma_i}^\textup{trop}: \Sigma_{\sigma_i} \to \Sigma_\sigma$ associated to maps $\sigma \to \sigma_i$ in $\Sigma_X$ (for $i=1,2$) consider the map of cone stacks with boundary
\begin{equation} \label{eqn:cone_stack_normalization_sortof}
\bigsqcup_{(\sigma', \varphi_1, \varphi_2) \in \mathfrak{G}_{\sigma_1 \leftarrow \sigma \rightarrow \sigma_2}} \Sigma_{\sigma'} \to \Sigma_{\sigma_1} \times_{\Sigma_\sigma} \Sigma_{\sigma_2}
\end{equation}
induced by the commutative diagram
\begin{equation} \label{eqn:gluing_comm_diag}
\begin{tikzcd}
\bigsqcup \Sigma_{\sigma'} \arrow[r,"\iota_{\varphi_2}^\textup{trop}"] \arrow[d,"\iota_{\varphi_1}^\textup{trop}", swap] & \Sigma_{\sigma_2} \arrow[d, "\iota_{\sigma \to \sigma_2}^\textup{trop}"] \\
\Sigma_{\sigma_1} \arrow[r, "\iota_{\sigma \to \sigma_1}^\textup{trop}"] & \Sigma_\sigma
\end{tikzcd}\,.
\end{equation}
Then the map \eqref{eqn:cone_stack_normalization_sortof} induces an isomorphism
\begin{equation}  \label{eqn:cone_stack_fiber_diagram_isom}
\bigsqcup_{(\sigma', \varphi_1, \varphi_2) \in \mathfrak{G}_{\sigma_1 \leftarrow \sigma \rightarrow \sigma_2}} \Sigma_{\sigma'}^0 \xrightarrow{\sim} (\Sigma_{\sigma_1} \times_{\Sigma_\sigma} \Sigma_{\sigma_1})^0
\end{equation}
on the interiors of its domain and target.
\end{lemma}
\begin{proof}
By definition of the fiber product, the existence of the map \eqref{eqn:cone_stack_normalization_sortof} follows from the commutativity of the diagram \eqref{eqn:gluing_comm_diag}. This commutativity in turn follows from the commutativity of the diagram \eqref{eqn:comm_diag_sigmas} and the functoriality 
\[
\mathsf{Star}_{\sigma' \to \sigma_1 \to \sigma} = \mathsf{Star}_{\sigma' \to \sigma_2 \to \sigma}\,.
\]
It remains to verify that \eqref{eqn:cone_stack_normalization_sortof} induces an isomorphism on the interiors. For this, recall that objects of the cone stack fiber product $\Sigma_{\sigma_1} \times_{\Sigma_\sigma} \Sigma_{\sigma_2}$ are indexed by objects in $\Sigma_1, \Sigma_2$ together with an isomorphism of their images in $\Sigma_\sigma$. This data boils down to a diagram of the form
\[
\begin{tikzcd}
\sigma \arrow[r] \arrow[rd] & \sigma_1 \arrow[r, "j_1'"] & \sigma_1' \arrow[d, "\varphi'"] & \sigma_1'' \arrow[l, "j_1''"] \arrow[d,"\varphi''"]\\
& \sigma_2 \arrow[r, "j_2'"] & \sigma_2' & \sigma_2'' \arrow[l, "j_2''"]
\end{tikzcd}
\]
in $\Sigma$ with $\varphi', \varphi''$ isomorphisms. The original objects are contained in $\Sigma_1^0, \Sigma_2^0$ if and only if also $j_1'', j_2''$ are isomorphisms. In this case, denoting $\sigma'' = \sigma_2''$ the data of the diagram boils down (up to unique isomorphism) to the data of the maps
\[
\widetilde \varphi_1 = \varphi'' \circ (j_1'')^{-1} \circ j_1' : \sigma_1 \to \sigma'' \text{ and } \widetilde \varphi_2 = (j_2'')^{-1} \circ j_2' : \sigma_2 \to \sigma''\,,
\]
commuting with the morphisms from $\sigma$. 
By the properties of the cone stack $\Sigma_X$, there is a unique minimal morphism $j : \sigma' \to \sigma''$ such that $\widetilde \varphi_1, \widetilde \varphi_2$ factor through $j$:
\[
\begin{tikzcd}
& \sigma_1 \arrow[dr, "\varphi_1", swap] \arrow[drr, "\widetilde{\varphi}_1", bend left] & & \\
\sigma\arrow[ru] \arrow[rd] & & \sigma' \arrow[r, "j"] & \sigma'' \\
& \sigma_2 \arrow[ur, "\varphi_2"] \arrow[urr, "\widetilde{\varphi}_2", swap, bend right] & &
\end{tikzcd}
\]
Then by the minimality of $j$ we have $(\sigma', \varphi_1, \varphi_2) \in \mathfrak{G}_{\sigma_1 \leftarrow \sigma \rightarrow \sigma_2}$ and $(\sigma' \xrightarrow{j} \sigma'' \xleftarrow{\mathrm{id}} \sigma'')$ gives an object of $\Sigma_{\sigma'}^0$. Conversely it is straightforward to check that $(\iota_{\varphi_1}^\textup{trop}, \iota_{\varphi_2}^\textup{trop})$ sends this object back to the original cone in $\Sigma_{\sigma_1} \times_{\Sigma_\sigma} \Sigma_{\sigma_2}$, and that these inverse functors induce isomorphisms on the cone $C(\sigma_2'') = C(\sigma'')$ associated to the objects in domain and target by the structure of their respective cone stacks.
\end{proof}
\begin{proposition} \label{Prop:iota_pushforward_intersection}
Let $\Sigma_X$ be the cone stack of $(X,D)$ and $\sigma_1 \leftarrow \sigma \rightarrow \sigma_2$ morphisms in $\Sigma_X$. Then there is a fiber diagram
\begin{equation} \label{eqn:gluing_fiber_diag}
\begin{tikzcd}
\bigsqcup P_{\sigma'} \arrow[r,"\iota_{\varphi_2}"] \arrow[d,"\iota_{\varphi_1}", swap] & P_{\sigma_2} \arrow[d, "\iota_{\sigma_2}"] \\
P_{\sigma_1} \arrow[r, "\iota_{\sigma_1}"] & P_\sigma
\end{tikzcd}\,.
\end{equation}
where the disjoint union is over $(\sigma', \varphi_1, \varphi_2) \in \mathfrak{G}_{\sigma_1 \leftarrow \sigma \rightarrow \sigma_2}$. Let $\deco_i \in \CH^\star(P_{\sigma_i})$ and $f_i \in \sPP_\star(P_{\sigma_i})$ for $i=1,2$ and denote by $F_\sigma \in \sPP_\star(P_\sigma)$ the strict piecewise polynomial with $\Psi(F_\sigma) = [P_\sigma]$ from Proposition \ref{Prop:B_Faces}. Then we have an equality
\begin{equation}  \label{eqn:iota_pushforward_products}
(\iota_{\sigma_1 \star} \deco_1 \cdot \Psi(f_1)) \cdot (\iota_{\sigma_2 \star} \deco_2 \cdot \Psi(f_2)) = \sum_{(\sigma', \varphi_1, \varphi_2) \in \mathfrak{G}_{\sigma_1 \leftarrow \sigma \rightarrow \sigma_2}} \iota_{\sigma' \star} \left((\iota_{\varphi_1}^\star \deco_1) \cdot (\iota_{\varphi_2}^\star \deco_2) \cdot \Psi \left(\frac{(\iota_{\varphi_1}^\textup{trop})^\star f_1 \cdot (\iota_{\varphi_2}^\textup{trop})^\star f_2}{(\iota_{\sigma' \to \sigma}^\textup{trop})^\star F_\sigma} \right) \right)
\end{equation}
in $\CH_\star(P_\sigma)$, where we use that the above fraction gives a well-defined element of $\sPP_\star(P_{\sigma'})$.
\end{proposition}
\begin{proof}
The fact that \eqref{eqn:gluing_fiber_diag} is a fiber diagram just follows from the geometric realization of the isomorphism \eqref{eqn:cone_stack_fiber_diagram_isom} above. For proving the formula of the intersection product \eqref{eqn:iota_pushforward_products} consider, similar to $F_\sigma$, the polynomials $F_{\sigma_i} \in \sPP_\star(P_{\sigma_i})$ with $\Psi(F_{\sigma_i}) = [P_{\sigma_i}]$ (and analogously for $\sigma'$). 
It follows immediately that
\[
\Psi(f_i) = \Phi\left(\frac{f_i}{F_{\sigma_i}} \right) \cdot [P_{\sigma_i}]\,,
\]
where $\Phi: \sPP^\star(P_{\sigma_i}) \to \CH^\star(P_{\sigma_i})$ is the usual map. Here the fraction $f_i/F_{\sigma_i}$ is well-defined by \eqref{eqn:sPP_hom_free_module}.
Using Proposition \ref{Prop:B_Faces}, we obtain a piecewise linear function $x_\rho$ on $P_{\sigma_i}$ for each $\rho \in \sigma_i(1)$ and have that
\[
F_{\sigma_i} = \prod_{\rho \in \sigma_i(1)} x_{\rho} \eqqcolon x_{\sigma_i(1)} \text{ and } F_{\sigma'} = \prod_{\rho \in \sigma'(1)} x_{\rho} \eqqcolon x_{\sigma'(1)}\,.
\]
Moreover, again from Proposition \ref{Prop:B_Faces}, the top Chern classes of normal bundles of the two horizontal maps in the diagram \eqref{eqn:gluing_fiber_diag} are given by
\[
e(\mathcal{N}_{\iota_{\sigma_1}}) = \Phi\left(\prod_{\rho \in \sigma_1(1) \setminus \sigma(1)} x_\rho\right) = \Phi\left(\frac{x_{\sigma_1(1)}}{x_{\sigma(1)}}\right) \text{ and }e(\mathcal{N}_{\iota_{\varphi_2}}) = \Phi\left(\prod_{\rho \in \sigma'(1) \setminus \sigma_2(1)} x_\rho\right) = \Phi\left(\frac{x_{\sigma'(1)}}{x_{\sigma_2(1)}}\right)\,.
\]
By the excess intersection formula \cite[Proposition 17.4.1]{Ful98}, the intersection product \eqref{eqn:iota_pushforward_products} is given by a sum of contributions from the components $P_{\sigma'}$ of the fiber product \eqref{eqn:gluing_fiber_diag}, with the contribution of $P_{\sigma'}$ given by
\begin{align*}
\iota_{\sigma' \star} \left((\iota_{\varphi_1}^\star \deco_1 \cdot \Phi\left(\frac{f_1}{F_{\sigma_1}}\right)) \cdot (\iota_{\varphi_2}^\star \deco_2 \cdot \Phi\left(\frac{f_1}{F_{\sigma_1}}\right)) \cdot e\left(\frac{\mathcal{N}_{\iota_{\sigma_1}}}{\mathcal{N}_{\iota_{\varphi_2}}}\right) \cdot [P_{\sigma'}] \right)\,.
\end{align*}
Using the formula $h^\star \Phi(g)=\Phi((h^\textup{trop})^\star g)$ for $h=\iota_{\varphi_1}, \iota_{\varphi_2}$, together with the formulas derived above (and identifying rays of $\sigma, \sigma_1, \sigma_2$ as elements of the common set $\sigma'(1)$ via the cone morphisms $\varphi_1, \varphi_2$), this simplifies to
\[
\iota_{\sigma' \star} \left((\iota_{\varphi_1}^\star \deco_1) \cdot (\iota_{\varphi_2}^\star \deco_2) \cdot \Phi\left(\frac{(\iota_{\varphi_1}^\textup{trop})^\star f_1 \cdot (\iota_{\varphi_2}^\textup{trop})^\star f_2}{x_{\sigma_1(1)} \cdot x_{\sigma_2(1)}} \right) \cdot \Phi\left(\frac{x_{\sigma_1(1)} \cdot x_{\sigma_2(1)}}{x_{\sigma(1)} \cdot x_{\sigma'(1)} } \right) \cdot \Psi(x_{\sigma'(1)}) \right)\,.
\]
Using that $\Phi(f) \cdot \Psi(g) = \Psi(f \cdot g)$, this readily simplifies to the formula \eqref{eqn:iota_pushforward_products}.
\end{proof}


Given a subdivision $\widehat \Sigma \to \Sigma_X$ and $\varphi_i : \widehat \sigma' \to \widehat \sigma_i$ a morphism in $\widehat \Sigma$, we denote by $\overline \varphi_i : \sigma' \to \sigma_i$ the induced morphism in $\Sigma_X$ given as the image of $\varphi_i$ under the map $\widehat \Sigma \to \Sigma_X$ of cone stacks.

\begin{proposition} \label{Prop:product_declogstrata_general}
Given a cone $\widehat \sigma \in \widehat \Sigma$, the product of two strictly decorated log-strata on $P_{\widehat \sigma}$ is given by
\begin{equation}
\label{eq:product_declogstrata_general}
\{\widehat \sigma_1, g_1, \beta_1\} \cdot \{\widehat \sigma_2, g_2, \beta_2\} = \sum_{(\widehat \sigma', \widehat \varphi_1, \widehat \varphi_2) \in \mathfrak{G}_{\widehat \sigma_1 \leftarrow \widehat \sigma \rightarrow \widehat \sigma_2}} \left\{\widehat \sigma', \frac{(\iota_{\widehat \varphi_1}^\textup{trop})^\star g_1 \cdot (\iota_{\widehat \varphi_2}^\textup{trop})^\star g_2}{(\iota_{\widehat \sigma' \to \widehat \sigma}^\textup{trop})^\star  \prod_{\rho \in \widehat \sigma(1)} x_\rho},  (\iota_{\overline \varphi_1})^\star \beta_1 \cdot (\iota_{\overline \varphi_2})^\star \beta_2 \right\}\,.
\end{equation}
\end{proposition}
\begin{proof}
This follows immediately by applying the formula from Proposition \ref{Prop:iota_pushforward_intersection} on the cone stack $\widehat \Sigma$ with $\deco_i = (s_{\widehat \sigma_i \to \sigma_i})^\star \beta_i$, which explains the summation over $\mathfrak{G}_{\widehat \sigma_1 \leftarrow \widehat \sigma \rightarrow \widehat \sigma_2}$ and the decoration by strict piecewise polynomials. To see how the decorations 
$(\iota_{\overline \varphi_i})^\star \beta_i$ arise, we observe  that 
$$\iota_{\varphi_i}^\star \deco_i = (s_{\widehat \sigma' \to \sigma'})^\star (\iota_{\overline \varphi_i})^\star \beta_i$$
and apply the definition of strictly decorated log strata classes.
\end{proof}

Continuing with preparations for the proof of Theorem \ref{Thm:taut_system_induction} we show functoriality of tautological classes along the map $s_{\widehat \sigma \to \sigma}$. 
\begin{lemma} \label{Lem:taut_pullback_hat}
Pushforward and pullback under the morphism $s_{\widehat \sigma \to \sigma}$ induce well-defined maps
\begin{equation}
s_{\widehat \sigma \to \sigma \star} : (\pi^\star \R_X)(P_{\widehat \sigma})  \to \R_X(P_\sigma) \text{ and } s_{\widehat \sigma \to \sigma}^\star : \R_X(P_\sigma) \to (\pi^\star \R_X)(P_{\widehat \sigma})\,.
\end{equation}
\end{lemma}
\begin{proof}
For the pushforward consider a strictly decorated log-stratum class $\{\widehat \sigma', g, \beta\}$ on $P_{\widehat \sigma}$ defined via the commutative diagram
\[
\begin{tikzcd}
~ & P_{\widehat \sigma'} \arrow[l, phantom, red, "\Psi(g)", near start] \arrow[r, "\iota_{\widehat \sigma \to \widehat \sigma'}"] \arrow[d, "s_{\widehat \sigma' \to \sigma'}", swap] & P_{\widehat \sigma} \arrow[d, "s_{\widehat \sigma \to \sigma}"] \\
~ & P_{\sigma'} \arrow[l, phantom, red, "\beta", near start] \arrow[r, "\iota_{\sigma \to \sigma'}"] & P_\sigma
\end{tikzcd}\,.
\]
Then we have
\[
s_{\widehat \sigma \to \sigma \star} \{\widehat \sigma', g, \beta\} = (\iota_{\sigma \to \sigma'})_\star \left( \beta \cdot s_{\widehat \sigma' \to \sigma' \star} \Psi(g) \right) = (\iota_{\sigma \to \sigma'})_\star \left( \beta \cdot  \Psi((s_{\widehat \sigma' \to \sigma'}^\textup{trop})_\star g) \right) \in \R_X(P_\sigma)\,.
\]
where the last containment follows as the tautological system $\R_X$ contains classes in the image of $\Psi$ and is closed under intersection products and pushforwards via $\iota_{\sigma \to \sigma'}$.
In the equality we also used that the map $s_{\widehat \sigma \to \sigma}$ is proper of relative log dimension $0$ (Definition \ref{Def:relative_log_dimension_0}) and thus its tropicalization $s_{\widehat \sigma \to \sigma}^\textup{trop} : \widehat \Sigma_{\widehat \sigma} \to \Sigma_\sigma$ admits a pushforward compatible with the map $\Psi$ (as in Proposition \ref{Prop:sPP_pushforward}).

For the pullback via $s_{\widehat \sigma \to \sigma}$, let $\deco \in \R_X(P_\sigma)$ and let $F_{\widehat \sigma} \in \sPP_\star(P_{\widehat \sigma})$ be the piecewise polynomial from Proposition \ref{Prop:B_Faces} which satisfies $\Psi(F_{\widehat \sigma})=[P_{\widehat \sigma}]$. Then taking $\sigma' = \sigma$ in the fiber diagram
\begin{equation}
\begin{tikzcd}
P_{\widehat \sigma \to \sigma} \arrow[r, "\mathsf{id}"] \arrow[d, "\pi_\sigma", swap] & P_{\widehat \sigma} \arrow[d, "s_{\widehat \sigma \to \sigma}"] \\
P_{\sigma} \arrow[r, "\mathsf{id}"] & P_\sigma
\end{tikzcd}
\end{equation}
we see
\[
[\sigma, F_{\widehat \sigma}, \deco] = (\mathsf{id})_\star \pi_\sigma^\star \deco \cdot \Psi(F_{\widehat \sigma}) = s_{\widehat \sigma \to \sigma}^\star \deco\,,
\]
proving that the pullback of tautological classes on $P_\sigma$ under $s_{\widehat \sigma \to \sigma}$ indeed lands in $(\pi^\star \R_X)(P_{\widehat \sigma})$.
\end{proof}

\begin{proof}[Proof of Theorem \ref{Thm:taut_system_induction}]
For the entire proof fix $\widehat \sigma \in \widehat \Sigma$ and let $\sigma = \varphi(\widehat \sigma)$ be its image in $\Sigma_X$.
By Proposition \ref{Prop:product_declogstrata_general} we have that $\R^\star(P_{\widehat \sigma}) \subseteq \mathsf{CH}^\star(P_{\widehat \sigma})$ is a sub-$\mathbb{Q}$-algebra. To see that it contains the class $\Psi(f)$ for any $f \in \mathsf{sPP}_\star(P_{\widehat \sigma})$, just note that 
\[
\Phi(f) = [\sigma, f, 1]\,,
\]
by a computation similar to the one presented in the proof of Lemma \ref{Lem:taut_pullback_hat}.

On the other hand, consider any morphism $\widehat{\sigma} \to \widehat{\sigma}'$ in $\widehat \Sigma$, mapping to the morphism $\sigma \to \sigma'$ under $\varphi$. Then we want to check invariance of the rings under pushforwards and pullbacks by $\iota_{\widehat{\sigma} \to \widehat{\sigma}'} : P_{\widehat \sigma'} \to P_{\widehat \sigma}$. By Lemma \ref{Lem:alternative_generators} we can verify this property on the generators of the tautological rings of $P_{\widehat \sigma}, P_{\widehat \sigma'}$ given by strictly decorated log-strata classes.

For the pushforward, let $\widehat \sigma' \to \widehat \sigma''$ be another morphism in $\widehat \Sigma$. Then the class $\{\widehat \sigma'', g, \beta\}$ on $P_{\widehat \sigma'}$ is defined by the following commutative diagram
\[
\begin{tikzcd}
 ~& P_{\widehat \sigma''} \arrow[l, phantom, red, "\Psi(g)", near start] \arrow[d] \arrow[r] & P_{\widehat \sigma'} \arrow[d] \arrow[r,  "\iota_{\widehat{\sigma} \to \widehat{\sigma}'}"] & P_{\widehat \sigma}\arrow[d]\\
~& P_{\sigma''} \arrow[l, phantom, red, "\beta", near start] \arrow[r] & P_{\sigma'} \arrow[r] & P_\sigma
\end{tikzcd}
\]
Then we just observe
\[
(\iota_{\widehat{\sigma} \to \widehat{\sigma}'})_\star \{\underbrace{\widehat \sigma''}_{\widehat \sigma'' \to \widehat \sigma'}, g, \beta\} = \{\underbrace{\widehat \sigma''}_{\widehat \sigma'' \to \widehat \sigma}, g, \beta\} \in (\pi^\star \R_X)(P_{\widehat \sigma})\,,
\]
where for clarity we have temporarily undone the abuse of notation discussed in Footnote \ref{footnote:morphism_notation_suppression}.

On the other hand, consider a class $\{\widehat \sigma_0, g, \beta\}$ on $P_{\widehat \sigma}$ for a morphism $\widehat \sigma \to \widehat \sigma_0$ with $\sigma_0 = \varphi(\widehat \sigma_0)$. Then from Proposition \ref{Prop:product_declogstrata_general} we know that the fiber product $P_{\widehat \sigma_0} \times_{P_{\widehat \sigma}} P_{\widehat \sigma'}$ is given by the disjoint union of spaces $P_{\widehat \sigma'''}$ for $(\widehat \sigma''', \varphi_1, \varphi_2) \in \mathfrak{G}_{\widehat \sigma_0 \leftarrow \widehat \sigma \to \widehat \sigma'}$. Then we obtain a commutative diagram
\[
\begin{tikzcd}
~& ~ & \bigsqcup P_{\widehat \sigma'''}  \arrow[rr] \arrow[ld, "\bigsqcup \iota_{\widehat \sigma_0 \to \widehat \sigma'''}", swap] \arrow[dd] &  & P_{\widehat \sigma'} \arrow[ld,  "\iota_{\widehat{\sigma} \to \widehat{\sigma}'}", swap] \arrow[dd] \\
~&P_{\widehat \sigma_0} \arrow[l, phantom, red, "\Psi(g)", near start] \arrow[rr] \arrow[dd] & & P_{\widehat \sigma} \arrow[dd]& \\
~&~ & \bigsqcup P_{\varphi(\widehat \sigma''')} \arrow[ld, "\bigsqcup \iota_{\sigma_0 \to \widehat \varphi(\sigma''')}"] \arrow[rr]  & & P_{\sigma'} \arrow[ld]\\
~&P_{\sigma_0} \arrow[l, phantom, red, "\beta", near start] \arrow[rr] & & P_\sigma &
\end{tikzcd}
\]
whose top face is a fiber diagram. Similar to the proof of Proposition \ref{Prop:iota_pushforward_intersection} we use that
\[
\Psi(g) = \Phi\left(\frac{g}{\prod_{\rho \in \widehat \sigma_0(1)} x_\rho} \right) \cdot [P_{\widehat \sigma_0}] \text{ and } [P_{\widehat \sigma'''}] = \Psi\left(\prod_{\rho \in \widehat \sigma'''(1)} x_\rho \right)\,.
\]
Then by a short diagram chase (using commutativity of proper pushforwards and Gysin pullbacks in the top fiber diagram) and the excess intersection formula, we have
\[
\iota_{\widehat \sigma \to \widehat \sigma'}^\star \{\widehat \sigma_0, g, \beta\} = \sum_{(\widehat \sigma''', \varphi_1, \varphi_2) \in \mathfrak{G}_{\widehat \sigma_0 \leftarrow \widehat \sigma \to \widehat \sigma'}} \{\widehat \sigma''', (\iota_{\widehat \sigma_0 \to \widehat \sigma'''}^\textup{trop})^\star g \cdot \prod_{\rho \in \widehat \sigma'(1) \setminus \mathrm{im}(\widehat \sigma(1))} x_\rho,  (\iota_{\sigma_0 \to \widehat \varphi(\sigma''')})^\star \beta \} \in (\pi^\star \R_X)(P_{\widehat \sigma'})\,.
\]
This shows that the tautological rings are indeed closed under pushforwards and pullbacks by the maps $\iota_{\widehat{\sigma} \to \widehat{\sigma}'}$, concluding the proof.
\end{proof}
We conclude Section \ref{Sect:TautSystems} with a basic compatibility check, verifying that the process of inducing tautological rings is transitive for a composition of log blowups. 
\begin{proposition} \label{Prop:tautological_system_functoriality}
Let $X_2 \xrightarrow{\pi_2} X_1 \xrightarrow{\pi_1} X$ be a sequence of log blowups with $X_1, X_2$ smooth and assume that $X$ carries a system of tautological rings $\R_X$. Then $\pi_2^\star \pi_1^\star \R_X = (\pi_1 \circ \pi_2)^\star \R_X$.
\end{proposition}



\begin{proof}
Let $\Sigma_2 \xrightarrow{\varphi_2} \Sigma_1 \xrightarrow{\varphi_1} \Sigma_X$ be the sequence of subdivisions associated to the above sequence of smooth log blowups. Given $\sigma_2 \in \Sigma_2$ a cone, we check the claimed equality of tautological systems on $X_2$ by verifying that they give the same subring of $\CH^\star(P_{\sigma_2})$. For this, we recall that both  $(\pi_2^\star \pi_1^\star \R_X)(P_{\sigma_2})$ and $( (\pi_1 \circ \pi_2)^\star \R_X)(P_{\sigma_2})$ are defined as the $\mathbb{Q}$-linear span of certain generators, given by (strictly) decorated strata classes. Our proof will proceed by showing that each type of generator of one ring can be expressed in terms of the generators of the other.

Denote by $\sigma_1 = \varphi_2(\sigma_2)$ and $\sigma = \varphi_1(\sigma_1)$ the images of $\sigma_2$ in $\Sigma_1, \Sigma$, respectively. For a choice of $\sigma_2 \to \sigma_2'$ in $\Sigma_2$, we have that generators of  $( (\pi_1 \circ \pi_2)^\star \R_X)(P_{\sigma_2})$ are given by the classes $\{\sigma_2', g, \beta\}$ defined via the commutative diagram
\[
\begin{tikzcd}
~& P_{\sigma_2'} \arrow[l, phantom, red, "\Psi(g)", near start] \arrow[r] \arrow[d] & P_{\sigma_2} \arrow[d]\\
~& P_{\sigma_1'} \arrow[r] \arrow[d, "s_{\sigma_1' \to \sigma}", swap] & P_{\sigma_1} \arrow[d]\\
~& P_{\sigma'} \arrow[l, phantom, red, "\beta", near start] \arrow[r] & P_\sigma
\end{tikzcd}
\]
By Lemma \ref{Lem:taut_pullback_hat} we know that $s_{\sigma_1' \to \sigma}^\star \beta \in (\pi_1^\star \R_X)(P_{\sigma_1'})$, and thus
\[
\{\sigma_2', g, \beta\} = \{\sigma_2', g, s_{\sigma_1' \to \sigma}^\star \beta\} \in (\pi_2^\star \pi_1^\star \R_X)(P_{\sigma_2})\,.
\]
For the other inclusion, we recall that a set of generators
\begin{equation}  \label{eqn:pi2pi1pullback_generator}
    \{\sigma_2', g_2, \{\sigma_1'', g_1, \beta\}\} \in (\pi_2^\star \pi_1^\star \R_X)(P_{\sigma_2})
\end{equation}
by strictly decorated log-stratum classes is specified by
\begin{itemize}
    \item choosing a morphism $\sigma_2 \to \sigma_2'$ in $\Sigma_2$, mapping to $\sigma_1' \to \sigma_1$ under $\varphi_2$, and an element $g_2 \in \sPP_\star(P_{\sigma_2'})$,
    \item choosing a morphism $\sigma_1' \to \sigma_1''$ in $\Sigma_1$ mapping to $\sigma' \to \sigma''$ in $\Sigma$ under $\varphi_1$, an element $g_1 \in \sPP_\star(P_{\sigma_1''})$ and a tautological class $\beta \in \R_X(P_{\sigma''})$.
\end{itemize}
These fit into the following commutative diagram of solid arrows
\begin{equation}
\begin{tikzcd}
\bigsqcup P_{\widetilde \sigma_2} \arrow[dr, dashed, "J"] \arrow[drr, dashed, bend left, "\iota_{\sigma_2' \to \widetilde \sigma_2}"] \arrow[dd, dashed, "s_{\widetilde \sigma_2 \to \widetilde \sigma_1}", swap] & ~ & ~ & ~ \\
& P_{\sigma_2' \to \sigma_1''} \arrow[r] \arrow[d] & P_{\sigma_2'} \arrow[ur, phantom, red, "\Psi(g_2)", very near start] \arrow[r] \arrow[d] & P_{\sigma_2} \arrow[d]\\
\bigsqcup P_{\widetilde \sigma_1} \arrow[d, dashed] \arrow[r, dashed, "\iota_{\sigma_1'' \to \widetilde \sigma_1}"] & P_{\sigma_1''} \arrow[r] \arrow[d] \arrow[ld, phantom, red, "\Psi(g_1)", very near start] & P_{\sigma_1'} \arrow[r] \arrow[d] & P_{\sigma_1} \arrow[d]\\
\bigsqcup P_{\widetilde \sigma} \arrow[r, dashed, "\iota_{\sigma'' \to \widetilde \sigma}"]  & P_{\sigma''} \arrow[r] \arrow[d, phantom, red, "\beta", near start] & P_{\sigma'} \arrow[r] & P_{\sigma}\\
~ & ~ &~ & 
\end{tikzcd}
\end{equation}
where the disjoint union is indexed by $(\sigma_2' \to \widetilde \sigma_2, \sigma_1'' \to \widetilde \sigma_1) \in \mathfrak{H}_{\sigma_2' \to \sigma_1''}$. As before let $\widetilde g_1, \widetilde g_2$ be strict piecewise polynomials such that $\Psi(g_1) = \Phi(\widetilde g_1) \cdot [P_{\sigma_1''}]$ and $\Psi(g_2) = \Phi(\widetilde g_2) \cdot [P_{\sigma_2'}]$. Abusing notation, we'll use the same notation for their pullback to the components $P_{\widetilde \sigma_2}$ of the disjoint union in the above diagram.

To conclude, note that the definition of the class \eqref{eqn:pi2pi1pullback_generator} involves pushing forward along $P_{\sigma_1''} \to P_{\sigma_1'}$ and pulling back along $P_{\sigma_2'} \to P_{\sigma_1'}$. First, we use compatibility of Gysin pullbacks with proper maps to reroute this via a refined Gysin pullback to and pushforward from $P_{\sigma_2' \to \sigma_1''}$. 


To analyze the result of this procedure, we first observe that since all spaces involved are idealized log smooth, and all maps are pulled back from corresponding maps of Artin fans, we conclude that the fiber product $P_{\sigma_2' \to \sigma_1''}$ is reduced, and indeed idealized log smooth. Then we claim that the above refined Gysin pullback can be calculated via the normalization morphism $J$. Indeed, the main term\footnote{For similar excess intersection theory calculations see, for example, \cite[Section 5]{canning2024tautological}.} will be given by taking a Gysin pullback from $P_{\sigma_1''}$ to the disjoint union of the spaces $P_{\widetilde \sigma_2}$, multiplying by the Euler class of a suitable excess bundle $E$ and then pushing forward to $P_{\sigma_2'}$. As before, the Euler class of $E$ comes from some piecewise polynomial $g_{E, \widetilde \sigma_2}$ on $P_{\widetilde \sigma_2}$.
Combining this with the usual projection formula, we can conclude that the \emph{main term} of the class \eqref{eqn:pi2pi1pullback_generator} is given by
\[
\sum_{(\sigma_2' \to \widetilde \sigma_2, \sigma_1'' \to \widetilde \sigma_1) \in \mathfrak{H}_{\sigma_2' \to \sigma_1''}}  \{\widetilde \sigma_2, \widetilde g_1 \cdot \widetilde g_2 \cdot g_{E, \widetilde \sigma_2} \cdot F_{\widetilde \sigma_2} , (\iota_{\sigma'' \to \widetilde \sigma})^\star \beta \} \in (\pi_2 \circ \pi_1)^\star \R_X)(P_{\sigma_2})\,,
\]
where $F_{\widetilde \sigma_2}$ is the usual piecewise polynomial with $\Phi(F_{\widetilde \sigma_2}) = [P_{\widetilde \sigma_2}]$.

In general, there will also be correction terms coming from intersections of different components $P_{\widetilde \sigma_2}$ in the normalization. These are again supported on smaller strata $P_{\widetilde \sigma_2'}$ and their contribution is given by a similar refined Gysin pullback of $\beta$ from $P_{\sigma''}$, acting on a class of a homological piecewise polynomial on $P_{\widetilde \sigma_2'}$ and then pushed forward to $P_{\sigma_2'}$. In particular, all these correction terms are also given by strictly decorated log strata classes giving generators of $(\pi_2 \circ \pi_1)^\star \R_X)(P_{\sigma_2})$. This concludes the argument.
\end{proof}

\subsection{Tautological rings of log blowups}  \label{Sect:taut_log_blowups}
Next, we want to show that for particularly simple log blowups $\widehat X \to X$, which correspond to the blowup of a smooth stratum closure in $X$, the pullback of a tautological system on $X$ to $\widehat X$ is determined purely by the data of the original tautological system and the combinatorial data of the blowup (and does \emph{not} require further knowledge of the geometry of $X$).

The basic tool for computing the system of tautological rings on the blowup are the \emph{projective bundle formula} and the \emph{blowup formula} for Chow groups. Recall that for $p: E=\mathbb{P}(\mathcal{E}) \to S$ the projectivization of a vector bundle $\mathcal{E}$ of rank $r$ on a smooth stack $S$, we have
\begin{equation} \label{eqn:projective_bundle_formula}
    \CH^\star(E) = \CH^\star(S)[\xi]/(c_r(\mathcal{E}) + c_{r-1}(\mathcal{E}) \xi + \ldots + c_1(\mathcal{E}) \xi^{r-1} + \xi^r)\,.
\end{equation}
Here the isomorphism is induced by the map sending $p^\star: \CH^\star(S) \to \CH^\star(E)$ via pullback and sending $\xi$ to $c_1(\mathcal{O}_E(1))$. 

On the other hand, let $X$ be a smooth stack and $Z \subseteq X$ a smooth closed substack. Consider the blowup $\widehat X$ of $X$ at $Z$, fitting into a fiber diagram
\begin{equation}
\begin{tikzcd}
E \arrow[r, "j"] \arrow[d, "\pi_E"] & \widehat X \arrow[d,"\pi"]\\
Z \arrow[r, "i"] & X
\end{tikzcd}
\end{equation}
where $E=\mathbb{P}(\mathcal{N}) \to Z$ is the exceptional divisor, given by the projectivization of the normal bundle $\mathcal{N}=\mathcal{N}_{Z/X}$. Denote by $\mathcal{Q} = \pi_E^\star \mathcal{N} / \mathcal{O}_E(-1)$ the universal quotient bundle on $E$ and consider the map
\[
h : \CH^\star(Z) \to \CH^\star(E), \alpha \mapsto - c_{m-1}(\mathcal{Q}) \pi_E^\star(\alpha)\,.
\]
Then there is an exact sequence
\begin{equation} \label{eqn:blow_up_sequence}
    0 \to \CH^\star(Z) \xrightarrow{(i_\star, h)} \CH^\star(X) \oplus \CH^\star(E) \xrightarrow{(\alpha, \beta) \mapsto \pi^\star(\alpha) + j_\star(\beta)} \CH^\star(\widehat X) \to 0\,.
\end{equation}
This sequence represents the Chow group of $\widehat X$ as a quotient of $\CH^\star(X) \oplus \CH^\star(E)$, where $\CH^\star(E)$ is determined by equation \eqref{eqn:projective_bundle_formula}. In fact, the natural ring structure on $\CH^\star(\widehat X)$ descends from a product on $\CH^\star(X) \oplus \CH^\star(E)$ given by the rules
\[
(\alpha_1, 0) \cdot (\alpha_2, 0) = (\alpha_1 \alpha_2, 0), \quad (\alpha, 0) \cdot (0, \beta) = (0, \beta \cdot \pi_E^\star i^\star \alpha), \quad (0, \beta_1) \cdot (0, \beta_2) = (0, -\beta_1 \cdot \beta_2 \cdot \xi)\,,
\]
with $\xi = c_1(\mathcal{O}_E(1))$ as before. For a reference see \cite[Exercise 8.3.9]{Ful98} for the case where $X$ is a variety and \cite[Theorem 7.1]{Abramovich_blow_up} for the case when $X$ is a quotient stack.

\begin{theorem} \label{Thm:R_blow_up_determined}
For a log blowup $\pi: \widehat X \to X$ that can be represented as a sequence of blowups of smooth strata closures, and a tautological system $\R_X$ on $X$, the data of the tautological rings \eqref{eqn:strata_taut_ring_inclusion} and the pushforward and pullback maps \eqref{eqn:gluing_map_preserving} between them determine the induced tautological system $\pi^\star \R_X$ via applications of the projective bundle formula \eqref{eqn:projective_bundle_formula} and the blowup exact sequence \eqref{eqn:blow_up_sequence}.
\end{theorem}
\begin{proof}
By Proposition \ref{Prop:tautological_system_functoriality} it suffices to show the claim for a single blowup of a smooth stratum closure $\overline S_\sigma = \widetilde S_\sigma$ for some $\sigma \in \Sigma_X$.
As we have seen in Lemma \ref{Lem:strata_closure_smooth}, the fact that $\overline S_\sigma$ is smooth is reflected on the cone stack side by the map $\Sigma_\sigma^0 / \mathsf{Aut}(\sigma) \to \Sigma$ being a fully faithful embedding of categories.
By Proposition \ref{Pro:blow_up_description}, the blowup map $\pi: \widehat X \to X$ corresponds to the map
\[
b_{\sigma, \Sigma_X}^\textup{trop} : \widehat \Sigma = \mathsf{ssd}_\sigma(\Sigma_X) \to \Sigma_X
\]
from the star-subdivision of $\Sigma_X$ at $\sigma$ that we specified in Construction \ref{Const:ssd}. 
The claim of the above theorem is then that for any $\widehat \sigma \in \mathsf{ssd}_\sigma(\Sigma_X)$ it is possible to understand the subring $(\pi^\star \mathsf{R}_X)(P_{\widehat \sigma}) \subseteq \CH^\star(P_{\widehat \sigma})$ purely in terms of the tautological maps \eqref{eqn:strata_taut_ring_inclusion}, \eqref{eqn:gluing_map_preserving} of the original system, and the combinatorics of $\sigma \in \Sigma_X$.
Using the description of the objects $\widehat \sigma \in \mathsf{ssd}_\sigma(\Sigma_X)$ from Construction \ref{Const:ssd}, we distinguish two cases:

\noindent \textbf{Case 1:} $\widehat \sigma = (\tau \to \sigma \to \sigma')$

The cone $\widehat \sigma$ corresponds to a stratum of $\widehat X$ inside the exceptional divisor of $\pi$.
The map $\widehat \Sigma \to \Sigma$ sends $\widehat \sigma$ to $\sigma'$, so as in \eqref{eqn:s_widehat_sigma_to_sigma} we obtain a map 
$q = s_{\widehat \sigma \to \sigma'}: P_{\widehat \sigma} \to P_{\sigma'}$. We claim that this map can be identified canonically as the projective bundle
\begin{equation} \label{eqn:q_projective_bundle}
    q : P_{\widehat \sigma} = \mathbb{P}\left( \mathcal{E}_{\tau \to \sigma} \right) \to P_{\sigma'},\text{ for } \mathcal{E}_{\tau \to \sigma} = \bigoplus_{\rho \in \sigma(1) \setminus \tau(1)} \pi_\sigma^\star \mathcal{L}_\rho
\end{equation}
where the line bundles $\pi_\sigma^\star \mathcal{L}_\rho$ are defined as in Proposition \ref{Pro:pi_sigma}. Intuitively this is plausible, as we are considering a stratum in the exceptional divisor of the blowup. 
For a slightly more rigorous argument, recall that the strata of $P_{\sigma'}$ correspond to the objects $(\sigma' \to \overline \sigma') \in \Sigma_{\sigma'}^0$ in the interior of the cone stack $\Sigma_{\sigma'}$. Unravelling the definitions, the strata of $P_{\widehat \sigma}$ mapping to them are then indexed by 
objects in $\widehat \Sigma_{\widehat \sigma}^0$ corresponding to diagrams
\[
\begin{tikzcd}
\tau \arrow[r] \arrow[d] & \sigma \arrow[r] \arrow[d] & \sigma'\arrow[d] \\
\overline \tau \arrow[r] & \sigma \arrow[r] & \overline \sigma'
\end{tikzcd}
\]
where the rightmost arrow is the one given by $(\sigma' \to \overline \sigma')$. Using the remaining automorphisms of such objects, we can make the middle arrow $\sigma \to \sigma$ be the identity. Then combinatorially, the objects are indexed by composition of face inclusions $(\tau \prec \overline \tau \precneq \sigma)$, whose associated cone $C(\tau \prec \overline \tau \precneq \sigma)$ is spanned by $b_\sigma, \overline \tau(1)$ and the rays of $\overline{\sigma}$ not in the image of $\sigma$. 
The unique smallest (or primitive) object in this collection is given by choosing $\overline \tau = \tau$. 
Using the theory of fibers of toric morphisms (see \cite[Proposition 2.1.4]{Toric_fibers}), we then confirm the identification \eqref{eqn:q_projective_bundle} of $P_{\widehat \sigma}$ as a projective bundle over $P_{\sigma'}$. The direct summands of the vector bundle $\mathcal{E}_{\tau \to \sigma}$, corresponding to the rays of $\sigma$ not in $\tau$, are determined by the fact that an object $(\tau \prec \overline \tau \precneq \sigma)$ as above is precisely determined by the choice of $\overline \tau(1) \setminus \tau(1) \subsetneq \sigma(1) \setminus \tau(1)$.

By the projective bundle formula \eqref{eqn:projective_bundle_formula}, the Chow group of $P_{\widehat \sigma}$ is determined as
\begin{equation} \label{eqn:CHX_projbundle}
 \CH^\star(P_{\widehat \sigma}) = \CH^\star(P_{\sigma'})[\xi]/(c_r(\mathcal{E}_{\tau \to \sigma}) + c_{r-1}(\mathcal{E}_{\tau \to \sigma}) \xi + \ldots + c_1(\mathcal{E}_{\tau \to \sigma}) \xi^{r-1} + \xi^r)\,.
\end{equation}
We claim that the same equality holds when replacing the full Chow group $\CH$ with the tautological rings:
\begin{equation} \label{eqn:RX_projbundle}
 (\pi^\star \R_X)(P_{\widehat \sigma}) = \R_X^\star(P_{\sigma'})[\xi]/(c_r(\mathcal{E}_{\tau \to \sigma}) + c_{r-1}(\mathcal{E}_{\tau \to \sigma}) \xi + \ldots + c_1(\mathcal{E}_{\tau \to \sigma}) \xi^{r-1} + \xi^r)\,.   
\end{equation}
If we prove this, then indeed $(\pi^\star \R_X)(P_{\widehat \sigma})$ is determined by $\R_X(P_{\sigma'})$ as the Chern classes $c_i(\mathcal{E}_{\tau \to \sigma})$ are simply the elementary symmetric polynomials in $c_1(\pi_\sigma^\star \mathcal{L}_\rho)$, coming from piecewise linear functions on $P_{\sigma'}$ by Proposition \ref{Prop:B_Faces}. 
Moreover, we note that there is a piecewise linear function $e \in \sPP_\star(\widehat \Sigma_{\widehat \sigma})$ such that $\Phi(e) = \xi$. Indeed we can choose $e$ as the pullback of $- \min(x_\rho: \rho \in \sigma(1))$ from $\mathsf{Faces}(\sigma)$.

To show the inclusion $\subseteq$ of \eqref{eqn:RX_projbundle} note that the generators of $(\pi^\star \R_X)(P_{\widehat \sigma})$ are given by the decorated log strata classes $[\sigma'', f, \deco]$ for a choice of $\sigma' \to \sigma''$ in $\Sigma_X$, a homological strict piecewise polynomial $f \in \sPP_\star(P_{\widehat \sigma \to \sigma''})$ and $\deco \in \R_X(P_{\sigma''})$, appearing in the diagram
\[
\begin{tikzcd}
~& P_{\widehat \sigma \to \sigma''} \arrow[l, phantom, red, "\Psi(f)", midway] \arrow[r] \arrow[d, "q'"] & P_{\widehat \sigma} \arrow[d, "q"]\\
~& P_{\sigma''} \arrow[l, phantom, red, "\deco", midway] \arrow[r, "\iota_{\sigma' \to \sigma''}"]  & P_{\sigma'}
\end{tikzcd}
\]
To prove that $[\sigma'', f, \deco]$ is contained in the right-hand side of \eqref{eqn:RX_projbundle}, it is sufficient to show that 
$$q_\star([\sigma'', f, \deco] \cdot \xi^a) \in \R_X^\star(P_{\sigma'})\text{ for all }0 \leq a \leq r.$$
Indeed it is a general result that for a cycle $\gamma$ in a projective bundle, its coefficients (as Chow classes on the base) in the projective bundle formula \eqref{eqn:CHX_projbundle} can be reconstructed from the pushforwards $q_\star(\gamma \cdot \xi^a)$, $a=0, \ldots, r$, in terms of a matrix of classes depending only on the Chern classes of the bundle. But noting that $\xi = \Phi(e) = \Psi(e \cdot F_{\widehat \sigma})$, with $F_{\widehat \sigma} \in \sPP_\star(P_\sigma)$ as in Proposition \ref{Prop:B_Faces}, we have $[\sigma'', f, \deco] \cdot \xi^a \in (\pi^\star \R_X)(P_{\widehat \sigma})$. Therefore its pushforward under $q = s_{\widehat \sigma \to \sigma}$ is tautological by Lemma \ref{Lem:taut_pullback_hat}, concluding the proof of the inclusion $\subseteq$.


To show the inclusion $\supseteq$ of \eqref{eqn:RX_projbundle} it remains to observe that
by Lemma \ref{Lem:taut_pullback_hat} we have $q^\star \R_X(P_{\sigma'}) \subseteq (\pi^\star \R_X)(P_{\widehat \sigma})$. Since also $\xi = \Psi(e \cdot F_{\widehat \sigma})$ is contained in this ring, we have concluded the equality \eqref{eqn:RX_projbundle} and Case 1.

\noindent \textbf{Case 2:} $\widehat \sigma = \widetilde \sigma$ for $\widetilde \sigma \in \Sigma_X$ not containing $\sigma$ as a face

A cone $\widehat \sigma$ of this type corresponds to a stratum of $\widehat X$ that is the strict transform of the stratum in $X$ corresponding to $\widetilde \sigma$. In the following we always write $\widehat \sigma$ for the cone in $\widehat \Sigma$, and $\widetilde \sigma$ for the cone in $\Sigma_X$.
We claim that the map 
\begin{equation}
q : P_{\widehat \sigma} = \mathsf{Bl}_{Z_{\widetilde \sigma, \sigma}} P_{\widetilde \sigma} \to P_{\widetilde \sigma}
\end{equation}
is a blowup of a union $Z_{\widetilde \sigma, \sigma} \subseteq P_{\widetilde \sigma}$ of smooth {and disjoint} strata closures in $P_{\widetilde \sigma}$, which are precisely the preimage of the blowup center $\overline S_{\sigma} \subseteq X$ under the map $\iota_{\widetilde \sigma} : P_{\widetilde \sigma} \to X$.
To characterize this preimage of $\overline{S}_\sigma$, 
note that the map $\iota_\sigma : P_\sigma \to X$ has exactly image $\overline{S}_\sigma$, and forms an $\mathsf{Aut}(\sigma)$-torsor over this image. Then from Proposition \ref{Prop:product_declogstrata_general} we know that that the fiber product of $\iota_{\widetilde \sigma}$ and the map $\iota_\sigma$ is given by the disjoint union
\begin{equation}
\widetilde{Z}_{\widetilde \sigma, \sigma} = \bigsqcup_{(\sigma', \widetilde \varphi, \varphi) \in \mathfrak{G}_{\widetilde \sigma, \sigma}} P_{\sigma'}
\end{equation}
over all isomorphism classes of generic $(\widetilde \sigma, \sigma)$-structures $(\widetilde \sigma \xrightarrow{\widetilde \varphi} \sigma' \xleftarrow{\varphi} \sigma)$.
These fit in an iterated fiber diagram as follows
\begin{equation}
\begin{tikzcd}
\widetilde{Z}_{\widetilde \sigma, \sigma} \arrow[r] \arrow[d] & {Z}_{\widetilde \sigma, \sigma} \arrow[r, hook] \arrow[d] & P_{\widetilde \sigma} \arrow[d, "\iota_{\widetilde \sigma}"]\\
P_{\sigma} \arrow[r] \arrow[rr, "\iota_\sigma", bend right, swap] & \overline{S}_\sigma \arrow[r, hook] & X
\end{tikzcd}
\end{equation}
Since the pullback of the $\mathsf{Aut}(\sigma)$-torsor $P_\sigma \to \overline{S}_\sigma$ is again such a torsor, we have that the center
\begin{equation} \label{eqn:Z_union_quotient}
{Z}_{\widetilde \sigma, \sigma} = \left( \bigsqcup_{(\sigma', \widetilde \varphi, \varphi) \in \mathfrak{G}_{\widetilde \sigma, \sigma}} P_{\sigma'} \right)/\mathsf{Aut}(\sigma)
\end{equation}
of the blowup $q$ is a union of strata closures of $P_{\widetilde \sigma}$. 
Let $\pi_E : E \to {Z}_{\widetilde \sigma, \sigma}$ be the exceptional divisor of $q$. Then from the blowup formula \eqref{eqn:blow_up_sequence} we obtain the exact sequence
\begin{equation} \label{eqn:CH_Xhat_sequence}
    0 \to \CH^\star({Z}_{\widetilde \sigma, \sigma}) \xrightarrow{(i_\star, h)} \CH^\star(P_{\widetilde \sigma}) \oplus \CH^\star(E) \xrightarrow{(\alpha, \beta) \mapsto q^\star(\alpha) + j_\star(\beta)} \CH^\star(P_{\widehat \sigma}) \to 0\,.
\end{equation}
In this sequence, the representation  \eqref{eqn:Z_union_quotient} of ${Z}_{\widetilde \sigma, \sigma}$ shows that its Chow group is given as
\begin{equation} \label{eqn:Chow_Z_sigmatilde}
\CH^\star({Z}_{\widetilde \sigma, \sigma}) = \left( \bigoplus_{(\sigma', \widetilde \varphi, \varphi) \in \mathfrak{G}_{\widetilde \sigma, \sigma}} \CH^\star(P_{\sigma'}) \right)^{\mathsf{Aut}(\sigma)}
\end{equation}
and the Chow group $\CH^\star(E)$ is given as a finite algebra over $\CH^\star({Z}_{\widetilde \sigma, \sigma})$ by a suitable projective bundle formula \eqref{eqn:projective_bundle_formula}.
Thus we see that if we have full control over all Chow rings $(\CH^\star(P_{\sigma_0}))_{\sigma_0 \in \Sigma_X}$, the natural pushforward maps between them and the Chern classes of normal bundles of the maps $\iota_{\sigma_0}$, then we also have full control over the Chow rings of the spaces $P_{\widehat \sigma}$ above. We claim that replacing the full Chow rings with the relevant tautological rings, the same result holds, and in particular
\begin{equation} \label{eqn:R_Xhat_sequence}
    0 \to \R_X^\star({Z}_{\widetilde \sigma, \sigma}) \xrightarrow{(i_\star, h)} \R_X^\star(P_{\widetilde \sigma}) \oplus \R_X^\star(E) \xrightarrow{(\alpha, \beta) \mapsto q^\star(\alpha) + j_\star(\beta)} (\pi^\star \R_X)(P_{\widehat \sigma}) \to 0\,.
\end{equation}
Here $\R_X^\star({Z}_{\widetilde \sigma, \sigma})$ is defined by the tautological analogue of the formula \eqref{eqn:Chow_Z_sigmatilde}, and $\R_X^\star(E)$ is defined by the projective bundle formula over it. The fact that the first map $(i_\star, h)$ in the sequence \eqref{eqn:R_Xhat_sequence} is well defined follows from the invariance of tautological rings under the pushforwards $\iota_{\widetilde \sigma \to \sigma' *}$ and the fact that the Chern classes of their normal bundles come from piecewise polynomial functions. For the second arrow, the pullback map $q^\star: \R_X^\star(P_{\widetilde \sigma}) \to (\pi^\star \R_X)(P_{\widehat \sigma})$ is well-defined by Lemma \ref{Lem:taut_pullback_hat}. To see that $j_\star : \R_X^\star(E) \to (\pi^\star \R_X)^\star(P_{\widehat \sigma})$ is well-defined, observe that taking the fiber diagram of $E \to Z_{\widetilde \sigma, \sigma}$ with the $\mathsf{Aut}(\sigma)$-torsor $\widetilde Z_{\widetilde \sigma, \sigma} \to  Z_{\widetilde \sigma, \sigma}$, we obtain a similar torsor $\widetilde E \to E$. By the previous step of the proof, we see easily that
\[
\widetilde E = \bigsqcup_{(\sigma', \widetilde \varphi, \varphi) \in \mathfrak{G}_{\widetilde \sigma, \sigma}} P_{\langle b_{\sigma'}, \widetilde \sigma \rangle}\,,
\]
where $b_{\sigma'} \in \sigma'$ is the barycenter of $\sigma'$, and $\langle b_{\sigma'}, \widetilde \sigma \rangle$ is the sub-cone of $\sigma'$ spanned by $b_{\sigma'}$ and the face $\widetilde \varphi: \widehat \sigma \to \sigma'$. This cone is part of $\widehat \Sigma$ and the natural morphism $\widehat \sigma \to \langle b_{\sigma'}, \widetilde \sigma \rangle$ induces a codimension $1$ map $P_{\langle b_{\sigma'}, \widetilde \sigma \rangle} \to P_{\widehat \sigma}$. Again by Case 1 of the current proof, the tautological ring of $\widetilde E$ is given by the projective bundle formula over $\R_X(\widetilde Z_{\widetilde \sigma, \sigma})$ and hence
\[
\R_X^\star(E) = (\pi^\star \R_X)^\star(\widetilde E)^{\mathsf{Aut}(\sigma)} = \left(\bigoplus_{(\sigma', \widetilde \varphi, \varphi) \in \mathfrak{G}_{\widetilde \sigma, \sigma}} (\pi^\star \R_X)(P_{\langle b_{\sigma'}, \widetilde \sigma \rangle}) \right)^{\mathsf{Aut}(\sigma)}\,.
\]
Then the fact that $j_\star$ is well-defined follows from the invariance of $\pi^\star \R_X$ under pushforward by the maps $P_{\langle b_{\sigma'}, \widetilde \sigma \rangle} \to P_{\widehat \sigma}$. Thus we conclude that the sequence \eqref{eqn:R_Xhat_sequence} is well-defined.

To show its exactness, we note that each term is naturally a subset of the full sequence \eqref{eqn:CH_Xhat_sequence}, which is exact. From this, exactness of \eqref{eqn:R_Xhat_sequence} at $\R_X^\star(Z_{\widetilde \sigma, \sigma})$ is automatic. For exactness in the middle, assume that there is a class $(\alpha, \beta) \in \R_X^\star(P_{\widetilde \sigma}) \oplus \R_X^\star(E)$ mapping to zero in $(\pi^\star \R_X)^\star(P_{\widehat \sigma})$. Then by exactness of \eqref{eqn:CH_Xhat_sequence} it comes from some class $\gamma \in \CH^\star(Z_{\widetilde \sigma, \sigma})$ and we must show that $\gamma$ is tautological. But the map $(\alpha, \beta) \mapsto (\pi_E)_\star \beta$ is a section of the map $(i_\star, h)$, and thus $\gamma = (\pi_E)_\star \beta$. Since the pushforward by $\pi_E$ sends tautological classes to tautological classes (see Lemma \ref{Lem:taut_pullback_hat}), we have that $\gamma$ is tautological as desired.

It remains to prove exactness on the right. So let $\delta \in (\pi^\star \R_X)(P_{\widehat \sigma})$ then by the exactness of the sequence of Chow groups, there exist $(\alpha, \beta) \in \CH^\star(P_{\widetilde \sigma}) \oplus \CH^\star(E)$ mapping to $\delta$. By modifying this pair via a class coming from $\CH^\star(Z_{\widetilde \sigma, \sigma})$, we can assume that $\beta$ satisfies $(\pi_E)_\star \beta = 0$. Equivalently, writing
\[
\beta = \beta_0 + \beta_1 \xi + \ldots + \beta_{r-2} \xi^{r-2} + \beta_{r-1} \xi^{r-1} \in \CH^\star(E)
\]
for unique $\beta_i \in \CH^\star(Z_{\widetilde \sigma, \sigma})$, we can assume $\beta_{r-1} = 0$.
But from this it follows
\[
q_\star \delta = q_\star q^\star \alpha + q_\star j_\star \beta = \alpha + i_\star \underbrace{(\pi_E)_\star \beta}_{=0} = \alpha\,.
\]
Since $q_\star$ sends tautological classes to tautological classes by Lemma \ref{Lem:taut_pullback_hat}, we have $\alpha \in \R_X(P_{\widetilde \sigma})$. By replacing $\delta$ with $\delta - q^\star \alpha$, we may assume without loss of generality that $\alpha = 0$ and thus $\delta = j_\star \beta\in (\pi^\star \R_X)^\star(P_{\widehat \sigma})$. Since the tautological system $\pi^\star \R_X$ is closed under pullback by the map $j$ (from Theorem \ref{Thm:taut_system_induction}), we have $j^\star \delta \in \R_X(E)$. But on the other hand
\[
j^\star \delta = j^\star j_\star \beta = \beta \cdot \xi = \beta_0 \xi + \beta_1 \xi^2 + \ldots + \beta_{r-2} \xi^{r-1} \in \R_X(E)\,.
\]
Uniqueness of this representation implies that all $\beta_i \in \R_X(Z_{\widetilde \sigma, \sigma})$, and hence also our original $\beta$ was contained in $\R_X(E)$. This concludes the proof of exactness of \eqref{eqn:R_Xhat_sequence}, the claim of Case 2 and thus of the theorem itself.
\end{proof}

\subsection{Log tautological rings and a generating set for \texorpdfstring{$\logCH^\star(X,D)$}{logCH*(X,D)}} \label{Sect:log_tautological_rings_in_general}
Let $(X,D)$ be a smooth normal-crossings pair as before. Then using the notion of induced tautological systems on log blowups $\widehat X$ of $X$, we can define the notion of a log tautological class.
\begin{definition}
\label{def:logtautring}
For $\R = \R_X$ a tautological system on $X$, we define the \emph{logarithmic tautological ring} of $X$ as the colimit
\begin{equation}
    \mathsf{logR}^\star(X) = \varinjlim_{\pi: \widehat X \to X} (\pi^\star\R_X)(\widehat X) \subseteq \mathsf{logCH}^\star(X)\,.
\end{equation}
of the induced tautological rings on smooth log blowups $\widehat X$ of $X$.
\end{definition}

By Remark~\ref{rem:pushforwardtautsystems}, this is indeed a logarithmic lift of the tautological ring $\R(X) \subset \CH(X)$.
\begin{proposition}
The image of $\mathsf{logR}^\star(X)$ under the pushforward map $\LogCH(X) \to \CH(X)$ is $\R^\star(X)$.
\end{proposition}

Since the induced tautological rings above are by definition spanned by decorated log strata classes, we can write down a formal $\Q$-algebra with an explicit surjection onto $\mathsf{logR}^\star(X)$.
\begin{definition}
\label{def:logS}
Let $\R_X$ be a tautological system on $X$, then its associated \emph{log strata algebra} is given by
\begin{equation} \label{eqn:logSstar}
\logS^\star(X) = \bigoplus_{\sigma \in \Sigma_X} \pPP_\star(\Sigma_\sigma) \otimes_{\sPP^\star(\Sigma_\sigma)} \R_X(P_\sigma)\,,
\end{equation}
where the sum goes over a set of \emph{representatives} of isomorphism classes $\sigma \in \Sigma_X$ and $\R_X(P_\sigma)$ is a module over $\sPP^\star(\Sigma_\sigma)$ via the map $\Phi$.
It admits a natural map
\begin{equation} \label{eqn:logS_to_logCH}
\logS^\star(X) \to \logCH^\star(X), \sum_{\sigma \in \Sigma_X} \underbrace{f_\sigma \otimes \gamma_\sigma}_{=:[\sigma, f_\sigma, \gamma_\sigma]} \mapsto \sum_{\sigma \in \Sigma_X} (\iota_\sigma)_\star \left(\gamma_\sigma \cdot \Psi_{P_\sigma}^\mathrm{log}(f_\sigma) \right)
\end{equation}
to the logarithmic Chow ring of $X$.
\end{definition}

Our first remark is that the notation $[\sigma, f_\sigma, \gamma_\sigma]$ above is compatible with the notation for decorated strata classes in Definition \ref{def:logdecstratum}. Indeed, for the piecewise polynomial $f_\sigma$ on $\Sigma_\sigma$ we can find a subdivision $\widehat \Sigma^\sigma \to \Sigma_\sigma$ making it a strict piecewise polynomial. Choose a subdivision $\widehat \Sigma \to \Sigma_X$ which contains the image of all new walls in $\widehat \Sigma^\sigma$. Then the image of $[\sigma, f_\sigma, \gamma_\sigma]$ is a decorated stratum class (given by the same notation). Indeed, in Definition \ref{def:logdecstratum} we choose $\widehat \sigma = 0 \in \widehat \Sigma$ mapping to $0 \in \Sigma_X$, which admits a map $0 \to \sigma$ in $\Sigma_X$. The crucial insight is simply that we can see $f_\sigma$ as an element of $\sPP_\star(P_{0 \to \sigma})$, since by construction, the cone stack $\Sigma_\sigma \times_{\Sigma_X} \widehat \Sigma$ is a refinement of $\widehat \Sigma^\sigma$.

We can define a natural multiplication on the log strata algebra $\logS^\star(X)$ by
\begin{equation} \label{eqn:dec_logstrata_product_nonstrict}
[\sigma_1, f_1, \gamma_1] \cdot [\sigma_2, f_2, \gamma_2] = \sum_{(\sigma', \varphi_1, \varphi_2) \in \mathfrak{G}_{\sigma_1, \sigma_2}} [\sigma', (\iota_{\varphi_1}^\textup{trop})^\star f_1 \cdot (\iota_{\varphi_2}^\textup{trop})^\star f_2 , \iota_{\varphi_1}^\star \gamma_1 \cdot \iota_{\varphi_2}^\star \gamma_2] \,.
\end{equation}

\begin{theorem}
\label{thm:logSproduct}
The formula \eqref{eqn:dec_logstrata_product_nonstrict} defines a product on the log strata algebra $\logS^\star(X)$, making the map $\logS^\star(X) \to \logCH^\star(X)$ from \eqref{eqn:logS_to_logCH} a ring homomorphism with image $\mathsf{logR}^\star(X)$.
\end{theorem}
To show this theorem, we prove a more general result about intersections of classes combining piecewise polynomials and operational Chow classes.
\begin{proposition} \label{Pro:general_spp_intersection_formula}
Let $X_1, X_2, Y$ be algebraic log stacks with $X_1, X_2$ idealized log smooth and $Y$ smooth and log smooth. Assume that the diagram
\[
\begin{tikzcd}
Z \arrow[d, "\rho_1",swap] \arrow[r, "\rho_2"] & X_2 \arrow[d, "\xi_2"]\\
X_1 \arrow[r, "\xi_1",swap] & Y
\end{tikzcd}
\]
is a fiber diagram with the maps $\xi_i$ being strict.
Let $\iota : Z \to Y$ be the map induced from this diagram.

Then for any $\gamma_i \in \CH^\star_\mathsf{op}(X_i)$ and $f_i \in \sPP_\star(X_i)$ we have
\begin{equation}  \label{eqn:fricking_intersection_formula}
\left( (\xi_1)_\star (\gamma_1 \cap \Psi(f_1)) \right) \cdot \left( (\xi_2)_\star (\gamma_2 \cap \Psi(f_2)) \right) = \iota_\star \left( (\rho_1^\star \gamma_1) \cdot (\rho_2^\star \gamma_2) \cap \Psi( (\rho_1^\textup{trop})^\star f_1 \cdot (\rho_2^\textup{trop})^\star f_2 )\right) \in \CH_\star(Y)\,.
\end{equation}
\end{proposition}
\begin{proof}
First, observe that both sides of \eqref{eqn:fricking_intersection_formula} are bilinear in $f_1, f_2$. Our goal is to decompose $f_1, f_2$ into a sum of simpler contributions, which can be analyzed separately using standard excess intersection theory. To find this decomposition, we claim that for $i=1,2$, the map
\begin{equation} \label{eqn:fricking_spp_surjection}
\bigoplus_{\sigma \in \Sigma_{X_i}^{\min}} \sPP_\star(P_{\sigma}) \xrightarrow{\oplus (\iota_\sigma^\textup{trop})_\star} \sPP_\star(X_i)\,,
\end{equation}
is surjective, where $\sigma$ runs through representatives of the minimal cones $\Sigma_{X_i}^{\min}$ of the cone stack of $X_i$ and $\iota_\sigma: P_\sigma \to X_i$ is the associated map from the monodromy torsor associated to $\sigma$. This is just a translation of the fact that the representable, surjective and proper map
\[
\coprod_{\sigma \in \Sigma_{X_i}^{\min}} \mathcal{P}_\sigma \to \mathcal{B}_{X_i}
\]
induces a surjection on Chow groups, using the identification in Theorem \ref{Thm:sPP_lowerstar_isom} and the compatibility of tropical and Chow pushforwards in Proposition \ref{Prop:sPP_pushforward}.
Using the surjectivity of \eqref{eqn:fricking_spp_surjection} we can without loss of generality assume that the polynomials $f_i$ are of the form $(\iota_i^\textup{trop})_\star \widetilde f_i$ for some monodromy torsor $\iota_i : P_i \to X_i$ of $X_i$. Given this data, we form the following fiber diagram:
\begin{equation} \label{eqn:fricking_fiber_diagram}
\begin{tikzcd}
\widehat Z \arrow[dd, bend right, "\widehat \xi_1", swap] \arrow[rr, bend left, "\widehat \xi_2"] \arrow[dr, dashed, "\zeta"] \arrow[d] \arrow[r] & \widehat P_2 \arrow[d] \arrow[r] & P_2\arrow[d, "\iota_2"]\\
\widehat P_1 \arrow[d] \arrow[r] & Z \arrow[d, "\rho_1",swap] \arrow[r, "\rho_2"] \arrow[dr, dashed, "\iota"] & X_2 \arrow[d, "\xi_2"]\\
P_1 \arrow[r, "\iota_1",swap] & X_1 \arrow[r, "\xi_1",swap] & Y
\end{tikzcd}
\end{equation}
Assume for now that we are able to show the claim of the proposition for $(X_i, f_i, \gamma_i)$ replaced by $(P_i, \widetilde f_i, (\iota_i)^\star \gamma_i)$. Then we would have shown that the left-hand side of \eqref{eqn:fricking_intersection_formula} is equal to
\begin{equation}
(\iota \circ \zeta)_\star \left( (\zeta^\star \rho_1^\star \gamma_1) \cdot (\zeta^\star \rho_2^\star \gamma_2) \cap \Psi( (\widehat \xi_1^\textup{trop})^\star \widetilde f_1 \cdot (\widehat \xi_2^\textup{trop})^\star \widetilde f_2) \right)\,.
\end{equation}
Using the projection formula, the desired equality with the right-hand side of \eqref{eqn:fricking_intersection_formula} then follows once we prove
\begin{equation} \label{eqn:fricking_pp_pushforward}
(\zeta^\textup{trop})_\star \left( (\widehat \xi_1^\textup{trop})^\star \widetilde f_1 \cdot (\widehat \xi_2^\textup{trop})^\star \widetilde f_2 \right) = (\rho_1^\textup{trop})^\star (\iota_1^\textup{trop})_\star \widetilde f_1 \cdot (\rho_2^\textup{trop})^\star (\iota_2^\textup{trop})_\star \widetilde f_2\,.
\end{equation}
To see \eqref{eqn:fricking_pp_pushforward}, we first make a couple of remarks:
\begin{enumerate}
    \item[a)] Associated to the fiber diagram \eqref{eqn:fricking_fiber_diagram} there exists a corresponding fiber diagram of cone stacks, connected by the maps $\zeta^\textup{trop}$, $\widehat \xi_i^\textup{trop}$ etc. In particular, slightly abusing notation\footnote{To avoid such an abuse, we could spell out the fiber product of cone stacks, saying: "Objects of $\Sigma_{\widehat P_i}$ are pairs of cones $(\sigma, \tau) \in \Sigma_{P_i} \times \Sigma_Z$ together with isomorphisms of their image cones in $X_i$". Since all relevant morphisms of cone stacks below induce isomorphisms on all their cones (see point b)), the set-theoretic description we give is sufficient to finish the argument.} we have
    \begin{align}
        \Sigma_{\widehat P_i} &= \{(p_i, z) : \iota_i^\textup{trop}(p_i) = \rho_i^\textup{trop}(z)\}\,, \nonumber\\
        \Sigma_{\widehat Z} &= \{(p_1, p_2, z) : \iota_i^\textup{trop}(p_i)=\rho_i^\textup{trop}(z) \text{ for }i=1,2\}\,. \label{eqn:fricking_cone_fiber_description}
    \end{align}
    \item[b)] The maps $\iota_i^\textup{trop}$ are associated to star fans of cones in $\Sigma_{X_i}$, and it follows that for each cone $\sigma \in \Sigma_{P_i}$ the map $\iota_i^\textup{trop}$ defines an isomorphism from $\sigma$ to its image cone in $\Sigma_{X_i}$.
    \item[c)] This implies that for both the maps $\iota_i$ and their base changes, the pushforwards of piecewise polynomials can be computed as 
    $$((\iota_i^\textup{trop})_\star g)(x_i) = \sum_{p_i \in (\iota_i^\textup{trop})^{-1}(x_i)} g(p_i)\,,$$
    where we take an appropriate groupoid sum as in Proposition \ref{prop:pushforwardppstack}. Since $\zeta^\textup{trop}$ is a composition of base-changes of the $\iota_i^\textup{trop}$, the same formula holds for $(\zeta^\textup{trop})_\star$.
\end{enumerate}
To prove formula \eqref{eqn:fricking_pp_pushforward} we now evaluate it at a point $z \in \Sigma_Z$. The left-hand side gives
\begin{equation} \label{eqn:fricking_pp_pushforward_LHS}
\sum_{(p_1, p_2, z) \in \Sigma_{\widehat Z}} \widetilde f_1(p_1) \cdot \widetilde f_2(p_2)\,.
\end{equation}
On the other hand, the right-hand side evaluates to
\begin{equation} \label{eqn:fricking_pp_pushforward_RHS}
\sum_{p_1 \in (\iota_1^\textup{trop})^{-1}(\sigma_1^\textup{trop}(z))} \widetilde f_1(p_1) \cdot \sum_{p_2 \in (\iota_2^\textup{trop})^{-1}(\sigma_2^\textup{trop}(z))} \widetilde f_2(p_2)\,.
\end{equation}
Comparing with the description in \eqref{eqn:fricking_cone_fiber_description} it follows that these two expressions are equal.

We are left with showing the claim of the proposition for $X_i$ of the form $P_i$ as above. The one advantage that the monodromy torsors have over the general situation is that they are smooth stacks and by \eqref{eqn:sPP_hom_free_module} they satisfy
\begin{equation}
\sPP_\star(P_i) \cong F_{P_i} \cdot \sPP^\star(P_i)\,,
\end{equation}
where $F_{P_i} \in \sPP_\star(P_i)$ is the polynomial with $\Psi(F_{P_i})=[P_i] \in \CH_\star(P_i)$.
Then the calculation can be finished using a standard excess intersection analysis modelled on the proof of Proposition \ref{Prop:iota_pushforward_intersection}. 
\end{proof}

\begin{proof}[Proof of Theorem \ref{thm:logSproduct}]
By  Lemma \ref{Lem:gluing_fiber_product_stacks} we have a fiber diagram
\begin{equation} \label{eqn:fib_diag_before_blowup}
\begin{tikzcd}
\bigsqcup_{\sigma'} P_{\sigma'} \arrow[r] \arrow[d] & P_{\sigma_2} \arrow[d]\\
P_{\sigma_1} \arrow[r] & X
\end{tikzcd}
\end{equation}
where the disjoint union is over $(\sigma', \varphi_1, \varphi_2) \in \mathfrak{G}_{\sigma_1, \sigma_2}$. Now in general, the polynomials $f_i \in \pPP_\star(P_{\sigma_i})$ will not be strict piecewise polynomials on $P_{\sigma_i}$. To remedy this, we can find a non-singular log blowup $\widehat X \to X$ such that for the fiber diagram
\[
\begin{tikzcd}
\bigsqcup_{\sigma'} \widehat P_{\sigma'} \arrow[r] \arrow[d] & \widehat P_{\sigma_2} \arrow[d, "\xi_2"]\\
\widehat P_{\sigma_1} \arrow[r, "\xi_1"] & \widehat X
\end{tikzcd}
\]
obtained by base-change of \eqref{eqn:fib_diag_before_blowup} we do have $f_i \in \sPP_\star(\widehat P_{\sigma_i})$. Then the left-hand side of \eqref{eqn:dec_logstrata_product_nonstrict} is defined as the intersection
\begin{equation}
(\xi_1)_\star\left(\gamma_1|_{\widehat P_{\sigma_1}} \cap \Psi(f_1) \right) \cdot (\xi_2)_\star\left(\gamma_2|_{\widehat P_{\sigma_2}} \cap \Psi(f_2) \right)\,.
\end{equation}
This is precisely the type of intersection covered by Proposition \ref{Pro:general_spp_intersection_formula}. Comparing with equation \eqref{eqn:fricking_intersection_formula} we see that the above intersection is precisely given by the right-hand side of \eqref{eqn:dec_logstrata_product_nonstrict} (using that $\widehat P_{\sigma'} \to P_{\sigma'}$ is a log blowup on which $(\iota_{\varphi_1}^\textup{trop})^\star f_1 \cdot (\iota_{\varphi_2}^\textup{trop})^\star f_2$ is a strict piecewise polynomial).

Finally, the fact that the image of \eqref{eqn:logS_to_logCH} equals $\mathsf{logR}^\star(X)$ simply follows as the tautological classes on a log blowup $\widehat X \to X$ are \emph{defined} as the span of decorated log strata classes as above.
\end{proof}

\begin{remark} \label{Rmk:strata_log_lift}
\begin{enumerate}
    \item[a)] For $\mathsf{S}^\star_{g,n}$ the strata algebra of $\R^\star(\Mbar_{g,n})$ as defined in Section \ref{straa}, there is a natural morphism $\mathsf{S}^\star_{g,n} \to \logS^\star(\Mbar_{g,n})$ of $\Q$-algebras. It sends a decorated strata class $[\Gamma, \gamma]$ (with $\gamma$ a product of $\kappa$- and $\psi$-classes on $\Mbar_\Gamma$) to the class $[\Gamma, F_\Gamma, \gamma]$ where $F_\Gamma = \prod_{e \in E(\Gamma)} \ell_e \in \sPP_\star(\Mbar_\Gamma)$ is the piecewise polynomial with $\Psi(F_\Gamma)=[\Mbar_\Gamma]$. From this equation, it is clear that the element in $\R^\star(\Mbar_{g,n})$ associated to $[\Gamma, \gamma]$ equals the image of the decorated log stratum $[\Gamma, F_\Gamma, \gamma]$. Moreover, the fact that $\mathsf{S}^\star_{g,n} \to \logS^\star(\Mbar_{g,n})$ is a morphism of $\Q$-algebras follows from the standard product formula for decorated strata classes \emph{combined} with the fact that in \eqref{eqn:logSstar} we take the tensor product over the strict piecewise polynomials. This allows us to convert higher powers of $\ell_e$ into decorations $-\psi_h - \psi_{h'}$ appearing in the excess intersection formula in $\mathsf{S}^\star_{g,n}$ for edges $e=(h,h') \in E(\Gamma)$.
    \item[b)] In the definition of the log strata algebra $\logS^\star(\Mbar_{g,n})$, the allowed decorations are Chow classes $\gamma_\Gamma \in \R^\star(\Mbar_\Gamma)$. A purely symbolic strata algebra $\logS_{g,n}^\star$ surjecting onto $\logS^\star(\Mbar_{g,n})$ was defined in \eqref{eqn:logSstar_introduction}, where the allowed symbols $\gamma_\Gamma$ are decorated strata classes (i.e. stable graphs with $\kappa, \psi$-decorations). One input that is needed for the tensor product appearing in \eqref{eqn:logSstar_introduction} is the existence of a natural factorization
    \[
    \begin{tikzcd}
        \sPP^\star(\Sigma_\Gamma) \arrow[r] \arrow[rr, bend left, "\Phi"] & \bigotimes_{v \in V(\Gamma)} \mathsf{S}_{g(v),n(v)}^\star \arrow[r] & \R^\star(\Mbar_\Gamma)
    \end{tikzcd}
    \]
    of the map $\Phi$ from strict piecewise polynomials on $\Sigma_\Gamma$ to tautological classes on $\Mbar_\Gamma$ via the strata algebra of $\Mbar_\Gamma$. Such a factorization exists due to the identification of images of $\Phi$ with normally decorated strata classes. Here polynomials on the factors $\Sigma_{g(v),n(v)}$ of $\Sigma_\Gamma$ map to decorated strata in $\mathcal{S}_{g(v),n(v)}$, whereas the additional coordinate functions $\ell_e$ for $e=(h,h') \in E(\Gamma)$ map to $-\psi_h - \psi_{h'}$ in the tensor product of strata algebras.

    \item[c)] Consider a generator $[\Gamma, f, \gamma] \in \logS^\star_{g,n}$ where $\gamma=\prod_{v \in V(\Gamma)} \gamma_v$ is a product of decorated strata classes such that at least one of the $\gamma_v$ is supported on some proper stratum of the space $\Mbar_{g(v),n(v)}$. Then the associated decoration $\gamma$ in $\R^\star(\Mbar_\Gamma)$ is a pushforward $(\iota_{\Gamma' \to \Gamma})_\star \gamma_0$ of a product $\gamma_0$ of $\kappa$ and $\psi$-polynomials on the vertices and edges of some specialization $\Gamma'$ of $\Gamma$. Using the projection formula and the pullback of homological piecewise polynomials for the partial gluing map $\iota_{\Gamma' \to \Gamma}$ (see the proof of Theorem \ref{Thm:taut_system_induction}), we can show that
    \[
    [\Gamma, f, \gamma] = [\Gamma', \left((\iota_{\Gamma' \to \Gamma}^\mathrm{trop})^\star f \right) \cdot \prod_{e \in E(\Gamma') \setminus \mathrm{im} E(\Gamma)} x_e, \gamma_0] \in \logCH^\star(\Mbar_{g,n})\,.
    \]
    Thus $\mathsf{logR}^\star(\Mbar_{g,n})$ is generated by decorated log strata classes $[\Gamma', F, \gamma_0]$ where $\gamma_0$ is a product of $\kappa$ and $\psi$-classes. 
\end{enumerate}
\end{remark}

Taking the maximal tautological system $\mathsf{CH}_X$ of all Chow classes from Example \ref{eqn:maximal_tautological_system}, the blowup formula also implies that the induced tautological system on an iterated blowup of smooth strata is still maximal:
\begin{proposition}
For an iterated blowup $\pi: \widehat X \to X$ of smooth strata, we have $\mathsf{CH}_{\widehat X} = \pi^\star \mathsf{CH}_X$.
\end{proposition}
\begin{proof}
Using Proposition \ref{Prop:tautological_system_functoriality} we can again reduce to the case of a single blowup of a smooth stratum. Then we have to show that for each $\widehat \sigma \in \Sigma_{\widehat X}$ we have an equality $(\pi^\star \mathsf{CH}_X)(P_{\widehat \sigma}) = \CH^\star(P_{\widehat \sigma})$. As in the proof of Theorem \ref{Thm:R_blow_up_determined} the cones $\widehat \sigma$ come in two variants (according to whether they describe strata in the exceptional divisor or not). The desired equality then comes from combining equations \eqref{eqn:CHX_projbundle} and \eqref{eqn:RX_projbundle} in one case, and equations \eqref{eqn:CH_Xhat_sequence} and \eqref{eqn:R_Xhat_sequence} in the other.
\end{proof}
Since the above blowups are cofinal in the system of all log blowups, the logarithmic tautological ring induced from $\R = \mathsf{CH}_X$ on $X$ is indeed the full log Chow ring $\mathsf{logCH}^\star(X)$. Combining this with Theorem \ref{Thm:R_blow_up_determined}, we  immediately obtain:
\begin{corollary} \label{Cor:logCH_generators}
The log Chow group $\mathsf{logCH}^\star(X)$ is generated as a $\mathbb{Q}$-vector space by decorated log strata $[\sigma, f, \deco]$ for $\sigma \in \Sigma_X$, $f \in \mathsf{PP}_\star(P_\sigma)$ and $\deco \in \mathsf{CH}^\star(P_\sigma)$. It is uniquely determined by the collection of maps
\[
\Psi: \mathsf{sPP}_\star(P_\sigma) \to  \mathsf{CH}^\star(P_\sigma)
\]
for $\sigma \in \Sigma$ (together with the $\Aut(\sigma)$-action on both sides) as well as the pushforwards and pullbacks
\[
 \iota_{\sigma' \to \sigma \star}:  \mathsf{CH}^\star(P_{\sigma'}) \to \mathsf{CH}^\star(P_{\sigma})  \text{ and } \iota_{\sigma' \to \sigma}^\star:  \mathsf{CH}^\star(P_\sigma) \to   \mathsf{CH}^\star(P_{\sigma'})
\]
for any morphism $\sigma' \to \sigma$ in $\Sigma_X$.

\end{corollary}

\begin{remark}
Using the techniques from Section \ref{Sect:log_tautological_rings_in_general}, we can also easily define the log tautological group of a smooth idealised log smooth DM stack, such as a stratum $\Mbar_{\Gamma}$ with its strict induced log structure by the gluing map $\iota_\Gamma$. We give a brief sketch of the construction: for $X$ smooth and idealised log smooth, let $(\Sigma,\Sigma^0,\Delta)$ be its cone stack with boundary. Then for  $(\tilde{\Sigma},\tilde{\Sigma}^0,\tilde{\Delta})$ a smooth subdivision let $\widehat \sigma \in \widehat \Sigma^0$ be a cone mapping to $\sigma = \varphi(\widehat \sigma) \in \Sigma^0$.

To define a \emph{decorated log-strata class} on $P_{\widehat \sigma}$ consider a triple $[\sigma' \to \sigma, f, \deco]$ of a morphism $\sigma' \to \sigma$ in $\Sigma_X$, 
a piecewise polynomial $f \in \sPP_\star(P_{\widehat \sigma \to \sigma'})$ and a decoration $\deco \in \R^\star(P_{\sigma'})$. Its associated class in $\logCH_\star(X)$ is given by the pushforward of
\begin{equation} \label{eqn:decoratedlogstratum_X_generalized}
[\sigma', f, \deco] = \iota_{\widehat \sigma \to \sigma' \star} \left(\pi_{\sigma'}^\star \deco \cdot \Psi(f) \right) \in \mathsf{CH}^\star(P_{\widehat \sigma})\,.
\end{equation}
under the map $P_{\widehat \sigma} \to \widehat X$. 
The collection of these form the \emph{log tautological group of $X$} 
\[
\mathsf{logR}_\star(X) \subset \logCH_\star(X).
\]
If $X$ is log smooth, then under the Poincar\'e isomorphism $\logCH_\star(X) \cong \logCH^\star(X)$, the log tautological group $\mathsf{logR}_\star(X)$ is identified with the log tautological ring $\mathsf{logR}^\star(X)$.

There is no direct analogue of Proposition~\ref{Prop:product_declogstrata_general}, as $\logCH_\star(X)$ does not have a ring structure.
The direct analogue of Corollary~\ref{Cor:logCH_generators} does hold, with the same proof. 
\end{remark}

\subsection{Functoriality of log tautological classes on moduli of curves}


\begin{proposition}
Let $p: X \to Y$ be a log lci map of smooth, log smooth DM-stacks, with a tropicalization $p: \Sigma_X \to \Sigma_Y$. Assume that for each map of cones $\sigma \to \sigma'$ with $\sigma \in \Sigma, \sigma' \in \Sigma'$ with induced pullback $p^\star: \CH^\star(P_{\sigma'}) \to \CH^\star(P_{\sigma})$ we have $p^\star(\mathsf{R}^\star(P_{\sigma'})) \subset \mathsf{R}^\star(P_{\sigma})$.

Then the pullback map
\[
p^\star: \LogCH^\star(Y) \to \LogCH^\star(X)
\]
restricts to a pullback back on log tautological classes
\[
p^\star: \mathsf{logR}^\star(Y) \to \mathsf{logR}^\star(X)
\]
\end{proposition}
\begin{proof}
Take a log tautological class $[\sigma', f, \gamma] \in \mathsf{logR}^\star(Y)$, with $\sigma' \in \Sigma'$. The pullback $\tau = \Star_{\sigma'}(\Sigma')^0 \times_{\Sigma'} \Sigma$ is a union of cones of $\Sigma$. Write $P_{\tau} = Y \times_{\Acal_{\Sigma'}} \Bcal_\tau$. Then $f$ pulls back to a homological piecewise polynomial $p^f \in \sPP_\star(P_\tau)$, and $\gamma$ pulls back to a tautological class on $P_{\tau}$. It suffices to show that the class $p^\star\gamma \cdot \Psi(p^\star f) \in \LogCH^\star(P_\tau)$ is a linear combination of pushforwards of log tautological classes on $P_{\sigma}$ for $\sigma \in \tau$. This follows by induction on the number of cones in $\tau$.
\end{proof}

\begin{proposition}
The forgetful map $\pi: \Mbar_{g,n+1} \to \Mbar_{g,n}$ induces a pushforward map
\[
\pi_\star: \mathsf{logR}^\star(\Mbar_{g,n+1}) \to \mathsf{logR}^\star(\Mbar_{g,n}).
\]
\end{proposition}
\begin{proof}
Let $[\Gamma,f, \gamma] \in \mathsf{logR}^\star(\Mbar_{g,n+1})$. Let $\overline{\Gamma}$ be the graph obtained from $\Gamma$ by forgetting the marking $n+1$ and stabilizing. Then we have a commutative diagram
\begin{equation}
\begin{tikzcd}
\Mbar_\Gamma \arrow[r, "\xi_\Gamma"] \arrow[d, "\pi_\Gamma"] & \Mbar_{g,n+1} \arrow[d, "\pi"]\\
\Mbar_{\overline \Gamma}  \arrow[r, "\xi_{\overline \Gamma}"] & \Mbar_{g,n}
\end{tikzcd}
\end{equation}
Note that $[\Gamma,f, \gamma] = (\xi_\Gamma)_\star \gamma \cdot \Psi^\log(f)$ is a log pushforward from $\Mbar_\Gamma$. Since log pushforwards are functorial, it suffices to show that 
\begin{equation} \label{eqn:Gamma_push_log_taut}
(\pi_\Gamma)_\star \gamma \cdot \Psi^\log(f) \in \logCH_\star(\Mbar_{\overline \Gamma})
\end{equation}
is log tautological. By this we mean that it is of the form $\overline \gamma \cdot \Psi^\log(\overline f)$. Indeed, then we would have
\[
\pi_\star [\Gamma,f, \gamma] = [\overline{\Gamma}, \overline{f}, \overline{\gamma}] \in \mathsf{logR}^\star(\Mbar_{g,n})\,,
\]
finishing the proof.

To prove the claim that \eqref{eqn:Gamma_push_log_taut} is log tautological, take a diagram
\begin{equation}
\label{eq:forgetfulpushforwardsquare}
\begin{tikzcd}
\Mtilde_\Gamma \arrow[r,] \arrow[d, "\tilde{\pi}_\Gamma"] & \Mbar_\Gamma \arrow[d, "\pi_\Gamma"]\\
\Mtilde_{\overline \Gamma}  \arrow[r, "\xi_{\overline \Gamma}"] & \Mbar_{\overline \Gamma} 
\end{tikzcd}
\end{equation}
where the horizontal maps are log blowups, $\tilde{\pi}_\Gamma$ is tropically transverse and $f \in \sPP_\star(\Mtilde_\Gamma)$. Then the pushforward $(\pi_\Gamma)_\star \gamma \cdot \Psi^\log(f) \in \LogCH_\star(\Mbar_{\ol{\Gamma}})$ is defined as $(\tilde{\pi}_\Gamma)_\star(\gamma \cdot \Psi(f))$. Let $\widehat{\Mcal_\Gamma}/\Mbar_\Gamma$ denote the pullback of $\Mtilde_{\ol{\Gamma}}/\Mbar_{\ol{\Gamma}}$. Then $\tilde{\pi}_{\Gamma}$ factors as
\[
\Mtilde_\Gamma \xrightarrow{p_1} \widehat{\Mcal}_\Gamma \xrightarrow{p_2} \Mtilde_{\ol{\Gamma}}.
\]
Here $p_1$ is a log blowup, and since $\gamma$ is a pullback from $\Mbar_\Gamma$, we have
\[
(p_1)_\star \left(\gamma \cdot \Psi(f)\right) = \gamma \cdot \Psi((p_1)^\textup{trop}_\star f)
\]
After replacing $f$ with its tropical pushforward under $p_1$, we can thus assume without loss of generality that $\Mtilde_\Gamma = \widehat{\Mcal_\Gamma}$, that is, \eqref{eq:forgetfulpushforwardsquare} is a fiber square.

The graph $\overline{\Gamma}$ either has a single contracted edge, or no contracted edge. We first treat the case where there is a single contracted edge $e$. In this case, the map \[\tilde{\pi}_{\Gamma}: \Mtilde_{\Gamma} \to \Mtilde_{\ol{\Gamma}}\] is an isomorphism of underlying algebraic stacks, and the map on the level of cone stacks with boundary is \[\Sigma_{\Gamma} = \Sigma_{\ol{\Gamma}} \times (\RR_{\geq 0}, 0)  \to \Sigma_{\ol{\Gamma}}.\] 
We write $\ell_e \in \sPP^\star((\RR_{\geq 0}, 0))$
for the piecewise linear function associated to the contracted edge. Then $f \in \sPP_\star(\Sigma_{\Gamma})$ is a sum of terms of the form $\ol{f} \cdot \ell_e^b$ with $b \geq 1$, and $\ol{f} \in \sPP_\star(\Sigma_{\ol{\Gamma}})$. 
The cohomological piecewise polynomial function $\ell_e \in \sPP^\star(\Mbar_{\Gamma})$ maps to the operational Chow class $-\psi_h - \psi_{h'} \in \mathsf{R}^\star(\Mbar_{\Gamma})$. We find
\begin{align*}
    \gamma \cdot \Psi(f) = (\gamma \cdot \Phi(\ell_e^{b-1})) \cdot \Psi(\ol{f} \cdot \ell_e) 
\end{align*}
As $\tilde{\pi}_{\Gamma}$ is an isomorphism, this pushes forward to
\[
\tilde{\pi}_{\Gamma,\star}(\gamma \cdot \Phi(\ell^{b-1})) \cdot \Psi(\ol{f}) \in \CH_\star(\Mtilde_{\ol{\Gamma}}).
\]
This is a tautological class multiplied by a homological piecewise polynomial, hence it is log tautological.

Now we treat the case where $\overline{\Gamma}$ has no contracted edges. Then $\sPP_\star(\Mbar_{\Gamma})$ is generated by homological piecewise polynomials on smaller strata and pullbacks of homological piecewise polynomial on $\Mtilde_{\Gamma}$. By Remark~\ref{Rmk:strata_log_lift}c) we can reduce to the second case. There the statement follows from the projection formula.

\end{proof}

%

\begin{proposition}
Let $\Gamma$ be a stable graph with associated gluing map $\xi_\Gamma : \Mbar_\Gamma^\str \to \Mbar_{g,n}$. Then pushforward by $\xi_\Gamma$ induces a map
\begin{equation} \label{eqn:xi_Gamma_log_pushforward}
(\xi_\Gamma)_\star : \bigotimes_{v \in V(\Gamma)} \mathsf{logR}^\star(\Mbar_{g(v),n(v)}) \to \mathsf{logR}^\star(\Mbar_{g,n})\,.
\end{equation}
\end{proposition}
\begin{proof}
To \emph{define} the map \eqref{eqn:xi_Gamma_log_pushforward} on the level of the full log Chow groups, note that we have log maps $\pi_v : \Mbar_\Gamma^\str \to \Mbar_{g(v),n(v)}$ and so we can first pull back from the factors $\logCH^\star(\Mbar_{g(v),n(v)})$ to $\logCH^\star(\Mbar_\Gamma^\str)$, take the product, and push forward under $\xi_\Gamma$. It remains to show that this map sends products of log tautological classes $[\Gamma_v, f_v, \gamma_v] \in \mathsf{logR}^\star(\Mbar_{g(v),n(v)})$ to tautological classes. To see this, let $\Gamma'$ be the stable graph obtained by gluing all graphs $\Gamma_v$ into the vertices of $\Gamma$. Then we have a diagram
\begin{equation}
\begin{tikzcd}
& & \Mbar_{\Gamma'}^\str \arrow[r] \arrow[d] & \Mbar_\Gamma^\str \arrow[r] \arrow[d, "\mathrm{id}"] & \Mbar_{g,n}\\
~& ~ & \prod_v \Mbar_{\Gamma_v}^\str \arrow[r, "\prod_v \xi_{\Gamma_v}"] \arrow[d, phantom, red, "\deco_v", midway] \arrow[ll, phantom, red, "\Psi^\log(f_v)"] & \prod_v \Mbar_{g(v),n(v)} & \\
& ~& ~& ~& ~
\end{tikzcd}
\end{equation}
where the left square is Cartesian. In the diagram above the terms $\Psi^\log(f_v)$ and $\gamma_v$ are placed next to the product of the $\Mbar_{g(v),n(v)}$ to indicate that this is where these classes live. The vertical maps in the diagram forget the extra log structure coming from the edges $E(\Gamma)$. Correspondingly, the pullback of log classes under these vertical maps corresponds to multiplication by $\prod_{e \in E(\Gamma)} \Phi(x_e)$, where $x_e \in \sPP^1(\Mbar_\Gamma^\str)$ is the piecewise linear function associated to the edge $e$. 
Then the diagram together with compatibility of log pushforwards and pullbacks shows that
\[
(\xi_\Gamma)_\star \prod_{v \in V(\Gamma)} \pi_v^\star [\Gamma_v, f_v, \gamma_v] = \left[\Gamma', \prod_{v \in V(\Gamma)} f_v \cdot \prod_{e \in E(\Gamma)} x_e, \prod_{v \in V(\Gamma)} \pi_v^\star \gamma_v \right] \in \mathsf{logR}^\star(\Mbar_{g,n})\,. \qedhere
\]
\end{proof}




\bibliographystyle{siam}
\bibliography{LogTaut}

\begin{thebibliography}{10}

\bibitem{ACGS15}
{\sc D.~Abramovich, Q.~Chen, M.~Gross, and B.~Siebert}, {\em {Decomposition of
  degenerate Gromov-Witten invariants}}, arXiv:1709.09864,  (2017).

\bibitem{ACMUW}
{\sc D.~Abramovich, Q.~Chen, S.~Marcus, M.~Ulirsch, and J.~Wise}, {\em
  Skeletons and fans of logarithmic structures}, in Nonarchimedean and Tropical
  Geometry, M.~Baker and S.~Payne, eds., Simons Symposia, Springer, 2016,
  pp.~287--336.

\bibitem{ACMW}
{\sc D.~Abramovich, Q.~Chen, S.~Marcus, and J.~Wise}, {\em Boundedness of the
  space of stable logarithmic maps}, {J. Eur. Math. Soc.}, 19 (2017),
  pp.~2783--2809.

\bibitem{AW}
{\sc D.~Abramovich and J.~Wise}, {\em {Birational invariance in logarithmic
  Gromov--Witten theory}}, Comp. Math., 154 (2018), pp.~595--620.

\bibitem{AC98}
{\sc E.~Arbarello and M.~Cornalba}, {\em Calculating cohomology groups of
  moduli spaces of curves via algebraic geometry}, Inst. Hautes \'{E}tudes Sci.
  Publ. Math.,  (1998), pp.~97--127.

\bibitem{Abramovich_blow_up}
{\sc V.~Arena, S.~Obinna, and D.~Abramovich}, {\em The integral chow ring of
  weighted blow-ups},
  \href{https://arxiv.org/abs/2307.01459}{arXiv:2307.01459},  (2023).

\bibitem{BaePark_comparison}
{\sc Y.~Bae and H.~Park}, {\em {A comparison theorem for cycle theories for
  algebraic stacks}}.
\newblock In preparation.

\bibitem{BaeSchmitt1}
{\sc Y.~Bae and J.~Schmitt}, {\em Chow rings of stacks of prestable curves
  {I}}, Forum Math. Sigma, 10 (2022), pp.~Paper No. e28, 47.
\newblock With an appendix by Bae, Schmitt and Jonathan Skowera.

\bibitem{BaeSchmitt2}
{\sc Y.~Bae and J.~Schmitt}, {\em Chow rings of stacks of prestable curves
  {II}}, Journal f{\"u}r die reine und angewandte Mathematik (Crelles Journal),
  2023 (2023), pp.~55--106.

\bibitem{Barrott2019Logarithmic-Cho}
{\sc L.~J. Barrott}, {\em Logarithmic {C}how theory},
  {\href{https://arxiv.org/abs/1810.03746}{arXiv:1810.03746}},  (2019).

\bibitem{Bishop_integral}
{\sc M.~{Bishop}}, {\em {The integral Chow ring of $\mathcal{M}_{1,n}$ for
  $n=3,\dots,10$}}, \href{https://arxiv.org/abs/2311.11408}{arXiv:2311.11408},
  (2023).

\bibitem{Bloch}
{\sc S.~Bloch}, {\em Algebraic cycles and higher {$K$}-theory}, Adv. in Math.,
  61 (1986), pp.~267--304.

\bibitem{Brion}
{\sc M.~Brion}, {\em Piecewise polynomial functions, convex polytopes and
  enumerative geometry}, in Parameter spaces ({W}arsaw, 1994), vol.~36 of
  Banach Center Publ., Polish Acad. Sci. Inst. Math., Warsaw, 1996, pp.~25--44.

\bibitem{canning2024tautological}
{\sc S.~Canning, D.~Oprea, and R.~Pandharipande}, {\em Tautological and
  non-tautological cycles on the moduli space of abelian varieties},
  \href{https://arxiv.org/abs/2408.08718v2}{arXiv:2408.08718},  (2024).

\bibitem{CCUW}
{\sc R.~Cavalieri, M.~Chan, M.~Ulirsch, and J.~Wise}, {\em A moduli stack of
  tropical curves}, {Forum Math. Sigma}, 8 (2020), pp.~1--93.

\bibitem{CLS11}
{\sc D.~A. Cox, J.~B. Little, and H.~K. Schenck}, {\em Toric varieties},
  vol.~124 of Graduate Studies in Mathematics, American Mathematical Society,
  Providence, RI, 2011.

\bibitem{dCP95}
{\sc C.~De~Concini and C.~Procesi}, {\em Wonderful models of subspace
  arrangements}, Selecta Math. (N.S.), 1 (1995), pp.~459--494.

\bibitem{FP-TNT}
{\sc C.~Faber and R.~Pandharipande}, {\em Tautological and non-tautological
  cohomology of the moduli space of curves}, in Handbook of moduli. {V}ol. {I},
  vol.~24 of Adv. Lect. Math. (ALM), Int. Press, Somerville, MA, 2013,
  pp.~293--330.

\bibitem{FY04}
{\sc E.~M. Feichtner and S.~Yuzvinsky}, {\em Chow rings of toric varieties
  defined by atomic lattices}, Invent. Math., 155 (2004), pp.~515--536.

\bibitem{Ful93}
{\sc W.~Fulton}, {\em Introduction to Toric Varieties}, Princeton University
  Press, 1993.

\bibitem{Ful98}
\leavevmode\vrule height 2pt depth -1.6pt width 23pt, {\em Intersection
  theory}, vol.~2 of Ergebnisse der Mathematik und ihrer Grenzgebiete. 3.
  Folge. A Series of Modern Surveys in Mathematics [Results in Mathematics and
  Related Areas. 3rd Series. A Series of Modern Surveys in Mathematics],
  Springer-Verlag, Berlin, second~ed., 1998.

\bibitem{GM07}
{\sc A.~Gibney and D.~Maclagan}, {\em Equations for {Chow} and {Hilbert
  Quotients}}, Journal of Algebra and Number Theory, 4 (2010), pp.~855--885.

\bibitem{GP03}
{\sc T.~Graber and R.~Pandharipande}, {\em Constructions of nontautological
  classes on moduli spaces of curves}, Michigan Math. J., 51 (2003),
  pp.~93--109.

\bibitem{HKT}
{\sc P.~Hacking, S.~Keel, and J.~Tevelev}, {\em {Stable pair, tropical, and log
  canonical compactifications of moduli spaces of del Pezzo surfaces}}, Invent.
  Math., 178 (2009), pp.~173--227.

\bibitem{Herr}
{\sc L.~Herr}, {\em The log product formula}, Algebra Number Theory, 17 (2023),
  pp.~1281--1323.

\bibitem{HMPW_survey}
{\sc L.~Herr, S.~Molcho, R.~Pandharipande, and J.~Wise}, {\em {Birational
  models of logarithmic Gromov--Witten theory}}, In preparation,  (2023).

\bibitem{HMPPS}
{\sc D.~Holmes, S.~Molcho, R.~Pandharipande, A.~Pixton, and J.~Schmitt}, {\em
  Logarithmic double ramification cycles},
  \href{https://arxiv.org/abs/2207.06778}{arXiv:2207.06778},  (2022).

\bibitem{HS22}
{\sc D.~Holmes and R.~Schwarz}, {\em Logarithmic intersections of double
  ramification cycles}, Algebr. Geom., 9 (2022), pp.~574--605.

\bibitem{Holmes2023LogarithmicCohomologicalFT}
{\sc D.~Holmes and P.~Spelier}, {\em Logarithmic cohomological field theories},
  \href{https://arxiv.org/abs/2308.01099}{arXiv:2308.01099},  (2023).

\bibitem{Toric_fibers}
{\sc Y.~Hu, C.-H. Liu, and S.-T. Yau}, {\em Toric morphisms and fibrations of
  toric {C}alabi-{Y}au hypersurfaces}, Adv. Theor. Math. Phys., 6 (2002),
  pp.~457--506.

\bibitem{Jan17}
{\sc F.~Janda}, {\em {Gromov--Witten theory of target curves and the
  tautological ring}}, Michigan Mathematical Journal, 66 (2017), pp.~683--698.

\bibitem{KKN21}
{\sc T.~Kajiwara, K.~Kato, and C.~Nakayama}, {\em Logarithmic abelian
  varieties, part {VII}: moduli}, Yokohama Math. J., 67 (2021), pp.~9--48.

\bibitem{KSZ91}
{\sc M.~Kapranov, B.~Sturmfels, and A.~Zelevinsky}, {\em Quotients of toric
  varieties}, Math. Ann., 290 (1991), pp.~643--655.

\bibitem{Kap93}
{\sc M.~M. {Kapranov}}, {\em {Chow quotients of Grassmannians. I}}, in I. M.
  Gelfand seminar. Part 2: Papers of the Gelfand seminar in functional analysis
  held at Moscow University, Russia, September 1993, Providence, RI: American
  Mathematical Society, 1993, pp.~29--110.

\bibitem{Kee92}
{\sc S.~Keel}, {\em Intersection theory of moduli space of stable $n$-pointed
  curves of genus zero}, Transactions of the American Mathematical Society, 330
  (1992), pp.~545--574.

\bibitem{KHNSZ}
{\sc P.~Kennedy-Hunt, N.~Nabijou, Q.~Shafi, and W.~Zheng}, {\em Divisors and
  curves on logarithmic mapping spaces}, Selecta Math. (N.S.), 30 (2024),
  p.~Paper No. 75.

\bibitem{Khan}
{\sc A.~A. {Khan}}, {\em {Virtual fundamental classes of derived stacks I}},
  \href{https://arxiv.org/abs/1909.01332}{arXiv:1909.01332},  (2019).

\bibitem{Kresch_cycle}
{\sc A.~Kresch}, {\em Cycle groups for {A}rtin stacks}, Invent. Math., 138
  (1999), pp.~495--536.

\bibitem{Eric_Larson_higher_Chow}
{\sc E.~Larson}, {\em The integral {C}how ring of {$\overline M_2$}}, Algebr.
  Geom., 8 (2021), pp.~286--318.

\bibitem{Hannah_Larson_higher_Chow}
{\sc H.~{Larson}}, {\em {The intersection theory of the moduli stack of vector
  bundles on $\mathbb{P}^1$}},
  \href{https://arxiv.org/abs/2104.14642}{arXiv:2104.14642},  (2021).

\bibitem{LP09}
{\sc M.~Levine and R.~Pandharipande}, {\em Algebraic cobordism revisited},
  Inventiones mathematicae, 176 (2009), pp.~63--130.

\bibitem{MS14}
{\sc D.~Maclagan and B.~Sturmfels}, {\em Introduction to {T}ropical
  {G}eometry}, vol.~161 of Graduate Studies in Mathematics, American
  Mathematical Society, Providence, RI, 2015.

\bibitem{MR2110098}
{\sc E.~Miller and B.~Sturmfels}, {\em Combinatorial commutative algebra},
  vol.~227 of Graduate Texts in Mathematics, Springer-Verlag, New York, 2005.

\bibitem{Mol22}
{\sc S.~Molcho}, {\em {Pullbacks of Brill-Noether classes under Abel-Jacobi
  Sections}}, \href{https://arxiv.org/abs/2212.14368}{arXiv:2212.14368},
  (2022).

\bibitem{MPS23}
{\sc S.~Molcho, R.~Pandharipande, and J.~Schmitt}, {\em The {H}odge bundle, the
  universal 0-section, and the log {C}how ring of the moduli space of curves},
  Compos. Math., 159 (2023), pp.~306--354.

\bibitem{MR21}
{\sc S.~Molcho and D.~Ranganathan}, {\em A case study of intersections on
  blowups of the moduli of curves},
  \href{https://arxiv.org/abs/2106.15194}{arXiv:2106.15194},  (2021).

\bibitem{Ogu06}
{\sc A.~Ogus}, {\em Lectures on logarithmic algebraic geometry}, vol.~178 of
  Cambridge Studies in Advanced Mathematics, Cambridge University Press,
  Cambridge, 2018.

\bibitem{Opr06}
{\sc D.~Oprea}, {\em Tautological classes on the moduli spaces of stable maps
  to pr via torus actions}, Advances in Mathematics, 207 (2006), pp.~661--690.

\bibitem{Pan99}
{\sc R.~Pandharipande}, {\em {Intersections of $\mathbb Q$-divisors on
  Kontsevich's moduli space $\overline{M}_{0,n}(\mathbb{P}^r,d)$ and
  enumerative geometry}}, Trans. Amer. Math. Soc., 351 (1999), pp.~1481--1505.

\bibitem{Pan15}
{\sc R.~Pandharipande}, {\em A calculus for the moduli space of curves}, in
  Algebraic geometry: {S}alt {L}ake {C}ity 2015, vol.~97.1 of Proc. Sympos.
  Pure Math., Amer. Math. Soc., Providence, RI, 2018, pp.~459--487.

\bibitem{PPZ}
{\sc R.~Pandharipande, A.~Pixton, and D.~Zvonkine}, {\em Relations on
  {$\overline{\mathcal M}_{g,n}$} via {$3$}-spin structures}, J. Amer. Math.
  Soc., 28 (2015), pp.~279--309.

\bibitem{Pixton_relations}
{\sc A.~{Pixton}}, {\em {Conjectural relations in the tautological ring of
  $\overline{\mathcal M}_{g,n}$}},
  \href{https://arxiv.org/abs/1207.1918}{arXiv:1207.1918},  (2012).

\bibitem{R15b}
{\sc D.~Ranganathan}, {\em {Skeletons of stable maps I: rational curves in
  toric varieties}}, J. Lond. Math. Soc., 95 (2017), pp.~804--832.

\bibitem{R19b}
\leavevmode\vrule height 2pt depth -1.6pt width 23pt, {\em A note on cycles of
  curves in a product of pairs},
  \href{https://arxiv.org/abs/1910.00239}{arXiv:1910.00239},  (2019).

\bibitem{R19}
\leavevmode\vrule height 2pt depth -1.6pt width 23pt, {\em Logarithmic
  {G}romov-{W}itten theory with expansions}, Algebr. Geom., 9 (2022),
  pp.~714--761.

\bibitem{RUK22}
{\sc D.~Ranganathan and A.~U. Kumaran}, {\em Logarithmic {G}romov-{W}itten
  theory and double ramification cycles}, J. Reine Angew. Math., 809 (2024),
  pp.~1--40.

\bibitem{RW19}
{\sc D.~Ranganathan and J.~Wise}, {\em Rational curves in the logarithmic
  multiplicative group}, Proc. Amer. Math. Soc., 148 (2020), pp.~103--110.

\bibitem{SvZ}
{\sc J.~Schmitt and J.~van Zelm}, {\em Intersections of loci of admissible
  covers with tautological classes}, Selecta Math. (N.S.), 26 (2020), pp.~Paper
  No. 79, 69.

\bibitem{Tev07}
{\sc J.~Tevelev}, {\em Compactifications of subvarieties of tori}, Amer. J.
  Math., 129 (2007), pp.~1087--1104.

\bibitem{stacks-project}
{\sc {The Stacks Project Authors}}, {\em \itshape stacks project}.
\newblock \url{http://stacks.math.columbia.edu}, 2019.

\bibitem{Tot14}
{\sc B.~Totaro}, {\em Chow groups, {C}how cohomology, and linear varieties},
  Forum Math. Sigma, 2 (2014), pp.~Paper No. e17, 25.

\bibitem{U16}
{\sc M.~Ulirsch}, {\em Non-{A}rchimedean geometry of {A}rtin fans}, Adv. Math.,
  345 (2019), pp.~346--381.

\end{thebibliography}

\end{document}